\documentclass{amsart}

\usepackage{amsmath, amssymb, amsthm, amsfonts, amsxtra, amscd}

\input xy
\xyoption{all}

\usepackage{hyperref}

\swapnumbers
\numberwithin{equation}{section}

\theoremstyle{plain}
\newtheorem{theorem}[subsubsection]{Theorem}
\newtheorem{maintheorem}[subsubsection]{Main Theorem}
\newtheorem{lemma}[subsubsection]{Lemma}
\newtheorem{prop}[subsubsection]{Proposition}
\newtheorem{cor}[subsubsection]{Corollary}
\newtheorem{ques}[subsubsection]{Question}
\newtheorem{conj}[subsubsection]{Conjecture}

\theoremstyle{definition}
\newtheorem{defn}[subsubsection]{Definition}

\newtheorem{remark}[subsubsection]{Remark}

\def\AA{\mathbb{A}}
\def\BB{\mathbb{B}}
\def\CC{\mathbb{C}}
\def\DD{\mathbb{D}}

\def\FF{\mathbb{F}}
\def\GG{\mathbb{G}}

\def\PP{\mathbb{P}}
\def\QQ{\mathbb{Q}}
\def\RR{\mathbb{R}}

\def\TT{\mathbb{T}}

\def\ZZ{\mathbb{Z}}

\def\calA{\mathcal{A}}

\def\calE{\mathcal{E}}
\def\calF{\mathcal{F}}
\def\calG{\mathcal{G}}
\def\calH{\mathcal{H}}

\def\calK{\mathcal{K}}
\def\calL{\mathcal{L}}

\def\calO{\mathcal{O}}
\def\calP{\mathcal{P}}

\def\bH{\mathbf{H}}
\def\bI{\mathbf{I}}

\def\bP{\mathbf{P}}

\def\bR{\mathbf{R}}

\newcommand{\Qbar}{\overline{\QQ}}
\newcommand{\Nbar}{\overline{N}}

\newcommand{\chibar}{\overline{\chi}}
\newcommand{\kbar}{\overline{k}}

\newcommand{\Xbar}{\overline{X}}

\newcommand{\Gammabar}{\overline{\Gamma}}
\newcommand{\gammabar}{\overline{\gamma}}

\newcommand{\tilK}{\widetilde{K}}
\newcommand{\tilT}{\widetilde{T}}
\newcommand{\tilW}{\widetilde{W}}
\newcommand{\tilY}{\widetilde{Y}}
\newcommand{\tilw}{\widetilde{w}}
\newcommand{\tpi}{\widetilde{\pi}}

\newcommand{\tilA}{\widetilde{A}}
\newcommand{\tila}{\widetilde{a}}
\newcommand{\tilb}{\widetilde{b}}
\newcommand{\tTT}{\widetilde{\TT}}
\newcommand{\tZG}{\widetilde{ZG}}

\newcommand{\hatG}{\widehat{G}}
\newcommand{\hatg}{\widehat{\mathfrak{g}}}
\newcommand{\hatH}{\widehat{H}}
\newcommand{\hath}{\widehat{\mathfrak{h}}}
\newcommand{\hatT}{\widehat{T}}

\newcommand{\una}{\underline{a}}

\newcommand{\unc}{\underline{c}}

\newcommand{\unG}{\underline{G}}

\newcommand{\unK}{\underline{K}}

\newcommand{\unU}{\underline{U}}

\newcommand{\unu}{\underline{u}}

\newcommand{\frr}{\mathfrak{r}}
\newcommand{\frg}{\mathfrak{g}}
\newcommand{\frl}{\mathfrak{l}}
\newcommand{\frt}{\mathfrak{t}}
\newcommand{\frA}{\mathfrak{A}}
\newcommand{\frB}{\mathfrak{B}}

\newcommand\act{\textup{act}}

\newcommand\an{\textup{an}}

\newcommand{\bIH}{\mathbf{IH}}

\newcommand{\cont}{\textup{cont}}
\newcommand{\dom}{\textup{dom}}
\newcommand\ev{\textup{ev}}
\newcommand\even{\textup{even}}

\newcommand\Frob{\textup{Frob}}
\newcommand\Gal{\textup{Gal}}
\newcommand\geom{\textup{geom}}
\newcommand{\Gr}{\textup{Gr}}

\newcommand\id{\textup{id}}

\newcommand\Inv{\textup{Inv}}
\newcommand\inv{\textup{inv}}
\newcommand\IC{\textup{IC}}

\newcommand\Irr{\textup{Irr}}

\newcommand\Lie{\textup{Lie}}
\newcommand\Loc{\textup{Loc}}

\newcommand\Mot{\textup{Mot}}

\newcommand\Out{\textup{Out}}
\newcommand\odd{\textup{odd}}
\newcommand\opp{\textup{opp}}
\newcommand\Perv{\textup{Perv}}

\newcommand{\pr}{\textup{pr}}
\newcommand{\pt}{\textup{pt}}
\newcommand{\pur}{\textup{pur}}
\newcommand\rank{\textup{rank}}

\newcommand{\rel}{\textup{rel}}
\newcommand\Rep{\textup{Rep}}
\newcommand{\Res}{\textup{Res}}

\newcommand\rk{\textup{rk}}

\newcommand{\sgn}{\textup{sgn}}

\newcommand\Spec{\textup{Spec}}

\newcommand\Stab{\textup{Stab}}
\newcommand\Sym{\textup{Sym}}

\renewcommand{\Vec}{\textup{Vec}}

\newcommand\Aut{\textup{Aut}}
\newcommand\Hom{\textup{Hom}}
\newcommand\End{\textup{End}}

\newcommand\uHom{\underline{\Hom}}

\newcommand\GL{\textup{GL}}

\newcommand\PGL{\textup{PGL}}
\newcommand\SL{\textup{SL}}

\newcommand\SO{\textup{SO}}
\newcommand\PSO{\textup{PSO}}

\newcommand{\Gm}{\GG_m}

\newcommand{\ad}{\textup{ad}}
\newcommand{\Ad}{\textup{Ad}}

\newcommand{\der}{\textup{der}}

\newcommand\xch{\mathbb{X}^*}
\newcommand\xcoch{\mathbb{X}_*}

\newcommand{\isom}{\stackrel{\sim}{\to}}
\newcommand{\leftexp}[2]{{\vphantom{#2}}^{#1}{#2}}

\newcommand{\pD}{\leftexp{p}{D}}

\newcommand{\twtimes}[1]{\stackrel{#1}{\times}}
\newcommand{\htimes}{\widehat{\otimes}}
\newcommand{\cohog}[2]{\textup{H}^{#1}({#2})}     
\newcommand{\cohoc}[2]{\textup{H}^{#1}_{c}({#2})}
\newcommand{\jiao}[1]{\langle{#1}\rangle}
\newcommand{\wt}[1]{\widetilde{#1}}
\newcommand{\wh}[1]{\widehat{#1}}
\newcommand{\mat}[4]{\left(\begin{array}{cc}{#1} & {#2} \\{#3} & {#4}\end{array}\right)}
\newcommand{\oleft}{\overleftarrow}
\newcommand{\oright}{\overrightarrow}
\newcommand{\un}[1]{\underline{#1}}
\newcommand{\upH}{\textup{H}}

\newcommand{\Bun}{\textup{Bun}}
\newcommand{\Fl}{\textup{Fl}}
\newcommand{\GR}{\textup{GR}}

\newcommand{\fl}{f\ell}
\newcommand{\Hk}{\textup{Hk}}
\newcommand{\Sat}{\textup{Sat}}
\newcommand{\tSat}{\widetilde{\textup{Sat}}}
\newcommand{\topo}{\textup{top}}

\newcommand{\tame}{\textup{tame}}
\newcommand{\qm}{\textup{qm}}

\newcommand{\X}[1]{\PP^1_{#1}-\{0,1,\infty\}}
\newcommand{\Z}{\ZZ[1/2\ell N]}

\newcommand{\Ql}{\QQ_\ell}
\newcommand{\Qlbar}{\overline{\QQ}_\ell}
\newcommand{\GQ}{\Gal(\overline{\QQ}/\QQ)}
\newcommand{\GQp}{\Gal(\overline{\QQ}_{p}/\QQ_{p})}
\newcommand{\Gk}{\Gal(\overline{k}/k)}

\title[Motives with exceptional Galois groups]{Motives with exceptional Galois groups\\ and the inverse Galois problem}
\author{Zhiwei Yun}
\address{Department of Mathematics, MIT, 77 Massachusetts Avenue, 2-173, Cambridge, MA 02139}
\email{zyun@math.mit.edu}
\date{}
\subjclass[2000]{14D24, 11F80, 20G41}
\keywords{exceptional groups, motives}

\begin{document}

\begin{abstract}
We construct motivic $\ell$-adic representations of $\GQ$ into exceptional groups of type $E_7,E_8$ and $G_2$ whose image is Zariski dense. This answers a question of Serre.  The construction is uniform for these groups and uses the Langlands correspondence for function fields. As an application, we solve new cases of the inverse Galois problem: the finite simple groups $E_{8}(\FF_{\ell})$ are Galois groups over $\QQ$ for large enough primes $\ell$.
\end{abstract}

\maketitle
\tableofcontents
\section{Introduction}

\subsection{Serre's question}
About two decades ago, Serre raised the following question which he described as ``plus hasardeuse'' (English translation: more risky):
\begin{ques}[Serre {\cite[Section 8.8]{Serre}}]
Is there a motive $M$ (over a number field) such that its motivic Galois group is a simple group of exceptional type $G_{2}$ or $E_{8}$?
\end{ques}
The purpose of this paper is to give an affirmative answer to a variant of Serre's question for $E_7, E_8$ and $G_2$.

\subsubsection{Motivic Galois groups} Let us briefly recall the notion of the motivic Galois group, following \cite[Section 1 and 2]{Serre}. Let $k$ and $L$ be number fields. Let $\Mot_k(L)$ be the category of motives over $k$ with coefficients in $L$. This is an abelian category obtained by formally adjoining direct summands of smooth projective varieties over $k$ cut out by correspondences with $L$-coefficients. Assuming the Standard Conjectures, the category $\Mot_k(L)$ becomes a neutral $L$-linear Tannakian category: it admits a tensor structure and a fiber functor $\omega$ into $\Vec_{L}$, the tensor category of $L$-vector spaces. For example, one may take $\omega$ to be the singular cohomology of the underlying analytic spaces (fix an embedding $k\hookrightarrow\CC$) with $L$-coefficients. By Tannakian formalism \cite{DM}, such a structure gives a group scheme $G^{\Mot}_k$ over $L$ as the group of tensor automorphisms of $\omega$. This is {\em the motivic Galois group of $k$}.

Any motive $M\in\Mot_k(L)$ generates a Tannakian subcategory $\Mot(M)$ of $\Mot_k(L)$. Tannakian formalism again gives a group scheme $G^{\Mot}_M$ over $L$, the group of tensor automorphisms of $\omega|_{\Mot(M)}$. This is {\em the motivic Galois group of $M$}. 

Of course Serre's question could be asked for other exceptional types. Although people hoped for an affirmative answer to Serre's question, the difficulty in constructing such motives lies in the fact that one cannot find motives with exceptional Galois groups within abelian varieties, nor do we have Shimura varieties of type $E_8, F_4$ or $G_2$.

\subsubsection{Motivic Galois representations} Let $\ell$ be a prime number. Fix an embedding $L\hookrightarrow\Ql'$, where $\Ql'$ is a finite extension of $\Ql$.  For a motive $M\in\Mot_k(L)$, we have the {\em $\ell$-adic realization} $\upH(M,\Ql')$ which is a continuous $\Gk$-module. 

Let $V$ be a finite dimensional $\Ql$-vector space. A continuous representation $\rho:\Gk\to\GL(V)$ is called {\em motivic} if there exists a motive $M\in\Mot_{k}(\QQ)$ such that $V$ is isomorphic to $\upH(M,\Ql)$ as $\Gk$-modules.  It is called {\em  potentially motivic}, if there exists a finite extension $k'/k$, a finite extension $\Ql'/\Ql$, a number field $L\subset\Ql'$ and a motive $M\in\Mot_{k'}(L)$ such that $V\otimes_{\Ql}\Ql'$ is isomorphic to $\upH(M,\Ql')$ as $\Ql'[\Gal(\kbar/k')]$-modules. 

Let $\hatG$ be a reductive algebraic group over $\Ql$. A continuous representation $\rho:\GQ\to\hatG(\Ql)$ is called {\em motivic} (resp. {\em potentially motivic}) if for some faithful algebraic representation $V$ of $\hatG$, the composition $\GQ\xrightarrow{\rho}\hatG(\Ql)\to\GL(V)$ is motivic (resp. potentially motivic).

For a motive $M\in\Mot_k(L)$ with $\ell$-adic realization $\upH(M,\Qlbar)$, we define the {\em $\ell$-adic motivic Galois group} of $M$ to be the Zariski closure of the image of the representation $\rho_{\ell,M}:\Gk\to\GL(\upH(M,\Qlbar))$. We denote the $\ell$-adic motivic Galois group of $M$ by $G^\ell_M$, which is an algebraic group over $\Qlbar$. It is expected (see \cite[Section 3.2]{Serre}) that $G^\ell_M\cong G^{\Mot}_M\otimes_{L}\Qlbar$. 

\subsection{Main results}

We will answer Serre's question for $\ell$-adic motivic Galois groups instead of the actual motivic Galois groups (whose existence depends on the Standard Conjectures anyway).

\begin{maintheorem}
Let $\hatG$ be a split simple adjoint group of type $E_7, E_8$ or $G_2$. Let $\ell$ be a prime number. Then there exists an integer $N\geq1$ and a continuous representation
\begin{equation*}
\rho: \pi_1(\X{\Z})\to\hatG(\Ql)
\end{equation*}
such that
\begin{enumerate}
\item For each geometric point $\Spec k\to\Spec\Z$, the restriction of $\rho$ to the geometric fiber $\X{k}$:
\begin{equation*}
\rho_k:\pi_1(\X{k})\to\hatG(\Ql)
\end{equation*}
has Zariski dense image.
\item  The restriction of $\rho$ to a rational point $x:\Spec\QQ\to\X{}$:
\begin{equation*}
\rho_{x}:\GQ\to\hatG(\Ql)
\end{equation*}
is either motivic (if $\hatG$ is of type $E_8$ or $G_2$) or potentially motivic (if $G$ is of type $E_7$, in which case it becomes motivic if we restrict $\rho_x$ to $\Gal(\Qbar/\QQ(i))$ and extend $\Ql$ to $\Ql(i)$, $i=\sqrt{-1}$).
\item There exist infinitely many rational points $\{x_{1},x_{2},\cdots\}$ of $\X{\QQ}$ such that $\rho_{x_{i}}$ are mutually non-isomorphic and all have Zariski dense image.
\end{enumerate}
\end{maintheorem}

\begin{cor}
For $\hatG$ a simple adjoint group of type $E_7, E_8$ or $G_2$ over $\Qlbar$, there exist infinitely many non-isomorphic motives $M$ over $\QQ$ (if $\hatG$ is of type $E_8$ or $G_2$) or $\QQ(i)$ (if $\hatG$ is of type $E_7$) whose $\ell$-adic motivic Galois groups are isomorphic to $\hatG$. In particular, Serre's question for $\ell$-adic motivic Galois groups has an affirmative answer for $E_7,E_8$ and $G_2$.
\end{cor}

\subsubsection{A known case}
In \cite{DR}, Dettweiler and Reiter constructed a rigid rank 7 local system on $\X{}$ whose geometric monodromy is dense in $G_2$. The restriction of this local system to a general rational points give motivic Galois representations whose image is dense in $G_2$. We believe that our construction gives the same local system as theirs (see remarks after Conjecture \ref{c:inf}).

On the other hand, Gross and Savin \cite{GS} gave an approach to find a $G_{2}$ motive in the cohomology of a Siegel modular variety.

In \cite{DR}, the authors consider the notion of motivic Galois groups in the category of {\em motives with motivated cycles}, a notion introduced by Andr\'e \cite{Andre}. Using Andr\'e's Theorem (\cite[Theorem 5.2 and Section 5.3]{Andre}, see also \cite[Theorem 3.2.1]{DR}), they show that their motives have motivic Galois groups of type $G_2$ in the sense of Andr\'e. We may apply the same argument to show that our motives also have the expected motivic Galois groups in the sense of Andr\'e. We omit the details here.

\subsubsection{Application to the inverse Galois problem} The inverse Galois problem for $\QQ$ asks whether every finite group can be realized as a Galois group of a finite Galois extension of $\QQ$. A lot of finite simple groups are proved to be Galois groups over $\QQ$, see \cite{MM} and \cite{SerreGal}, yet the problem is still open for many finite simple groups of Lie type. We will be concerned with finite simple groups $G_{2}(\ell)$ and $E_{8}(\ell)$, where $\ell$ is a prime. According to \cite[Chapter II, \S10]{MM}, $G_{2}(\ell)$ is known to be a Galois group over $\QQ$ for all primes $\ell$; however, $E_{8}(\ell)$ is known to be a Galois group over $\QQ$ only for $\ell\equiv\pm3,\pm7,\pm9,\pm10,\pm11,\pm12,\pm13,\pm14(\mod 31)$. As an application of our main construction, we show

\begin{theorem}\label{th:invGalois} For sufficiently large prime $\ell$, the finite simple groups $E_{8}(\FF_{\ell})$ are Galois groups over $\QQ$.
\end{theorem}

\subsection{The case of type $A_1$}
The Main Theorem also holds if $\hatG=\PGL_{2}$, and the construction in this case is easier to spell out. Let $k=\FF_{q}$ be a finite field of characteristic not $2$. The construction starts with an automorphic form of $G=\SL_{2}$. Let $T\subset B\subset G$ be the diagonal torus and the upper triangular matrices. Let $F=k(t)$ be the function field of $\PP^1_k$ where $k$ is a finite field. For each place $v$ of $F$, let $\calO_v,F_{v}$ and $k_{v}$ be the corresponding completed local ring, local field and residue field. For each $v$, we have the Iwahori subgroup $\bI_{v}\subset G(\calO_{v})$ which is the preimage of $B(k_{v})\subset G(k_{v})$ under the reduction map $G(\calO_{v})\to G(k_{v})$. 

Let $\pi=\otimes_v\pi_v$ be an irreducible automorphic representation of $\SL_{2}(\AA_F)$ satisfying the following conditions:
\begin{itemize}
\item $\pi_1$  and $\pi_{\infty}$ both contain a nonzero fixed vector under the Iwahori subgroup $\bI_v$.
\item $\pi_{0}$ contains a nonzero vector on which the Iwahori $\bI_{0}$ acts through the quadratic character $\mu:\bI_{0}\to T(k)=k^{\times}\twoheadrightarrow\{\pm1\}$.
\item For $v\neq 0,1$ or $\infty$, $\pi_v$ is unramified.
\end{itemize}

The following lemma can be proved by the same argument of Theorem \ref{th:clean}.
\begin{lemma}
Suppose $\sqrt{-1}\in k$, then such an automorphic representation $\pi$ exist.
\end{lemma}

Langlands philosophy then predicts that there should be a tame $\PGL_{2}(\Qlbar)$-local system on $\X{k}$ which has unipotent monodromy around the punctures $1$ and $\infty$, and monodromy of order 2 around the puncture $0$. Methods from geometric Langlands theory allow us to write down this local system explicit as follows.

\subsubsection{The local system} We do {\em not} assume $\sqrt{-1}\in k$. Consider the following family of genus 3 projective smooth curves $f:C\to\X{k}$:
\begin{equation}\label{curve}
C_\lambda: y^4=\frac{\lambda x-1}{\lambda x(x-1)},\hspace{1cm}\lambda\in\X{k}.
\end{equation}
Let $\Ql'=\Ql(\sqrt{-1})$. The group $\mu_{4}$ acts on the local system $\bR^{1}f_{*}\Ql'$. Let $k'=k(\sqrt{-1})$. Let $\chi:\mu_{4}(k')\to\Ql'^{\times}$ be a character of order 4. Over $\X{k'}$, we may decompose $\bR^{1}f_{*}\Ql'$ according to the action of $\mu_{4}(k')=\mu_{4}(\kbar)$:
\begin{equation*}
(\bR^{1}f_{*}\Ql')_{k'}=L_{\sgn}\oplus L_{\chi}\oplus L_{\chibar}.
\end{equation*}
Here $L_{\sgn}, L_{\chi}$ and $L_{\chibar}$ are rank two local system defined on $\X{k'}$. In fact $L_{\sgn}$ interpolates the $\upH^{1}$ of the Legendre family of elliptic curves. By Katz's results on the local monodromy of middle convolutions, one can show \footnote{This is communicated to the author by N.Katz} that the local geometric monodromy of both $L_\chi$ and $L_{\chibar}$ at $0,1$ and $\infty$ are conjugate to
\begin{equation*}
\mu_{0}\sim\mat{\sqrt{-1}}{0}{0}{-\sqrt{-1}}; \mu_{1}\sim\mat{1}{1}{0}{1}; \mu_{\infty}\sim\mat{1}{1}{0}{1}.
\end{equation*}
The $\PGL_{2}(\Ql')=\SO_{3}(\Ql')$-local systems $\Sym^{2}(L_{\chi})(1)$ and $\Sym^{2}(L_{\chibar})(1)$ are canonically isomorphic, and they descend to (isomorphic) $\PGL_{2}(\Ql')$-local systems on $\X{k}$. This is the local system predicted by the Langlands correspondence. Moreover, one can argue that this $\PGL_{2}(\Ql')$ local system comes from a $\PGL_{2}(\Ql)$-local system via extension of coefficient fields.

\subsubsection{Switching to $\QQ$} Note that the above construction makes perfect sense if we replace the finite field $k$ by any field of characteristic not equal to 2, and in particular $\QQ$. The resulting $\PGL_{2}$-local system $\Sym^{2}(L_{\chi})(1)$ over $\X{\QQ}$ is output of our Main Theorem in the case of type $A_{1}$. This local system is visibly potentially motivic because it is part of the $\upH^{1}$ of the family of curves \eqref{curve} cut out by the $\mu_{4}$-action.

\subsection{The motives}\label{ss:explicit}
Let $G$ be a split simply-connected group of type $E_7, E_8$ or $G_2$, defined over $\QQ$. Let $\hatG$ be its Langlands dual group defined over $\Ql$.

Let $\alpha^\vee$ be the coroot of $G$ corresponding to the maximal root $\alpha$. Let $V^{\qm}:=V_{\alpha^\vee}$ be the irreducible representation of $\hatG$ with highest weight $\alpha^\vee$, which we call the {\em quasi-minuscule representation} of $\hatG$. This is either the adjoint representation (if $G$ is of type $E_7$ or $E_8$) or the 7-dimensional representation if $G$ is of type $G_2$. For $x\in\QQ-\{0,1\}$, let $\rho_{x}^{\qm}$ be the composition
\begin{equation*}
\rho_{x}^{\qm}:\GQ\xrightarrow{\rho_{x}}\hatG(\Ql)\to\GL(V^{\qm})
\end{equation*}
where $\rho_x$ is as in Main Theorem (2). We will describe $\rho_{x}^{\qm}$ motivically.

Let $P$ be the ``Heisenberg parabolic'' subgroup of $G$ containing $T$ with roots $\{\beta\in\Phi_G|\jiao{\beta,\alpha^\vee}\geq0\}$. The unipotent radical of $P$ is a Heisenberg group whose center has Lie algebra $\frg_\alpha$, the highest root space. The contracted product gives a line bundle over the partial flag variety $G/P$:
\begin{equation*}
Y=G\twtimes{P}\frg^*_\alpha 
\end{equation*}
which is a smooth variety over $\QQ$ of dimension $2h^\vee-2$. Here $h^\vee$ is the dual Coxeter number of $G$, which is 18 for $E_{7}$, 30 for $E_{8}$ and 4 for $G_{2}$. So the corresponding variety $Y$ has dimension 34, 58 and 6 in the case $E_{7}, E_{8}$ and $G_{2}$ respectively. The more complicated part is that there is a divisor $D_x\subset Y$ depending algebraically on $x$. There is also a finite group scheme $\tilA$ over $\QQ$ and an $\tilA$-torsor
\begin{equation}\label{Atorsor}
\tilY_{x}\to Y-D_{x}.
\end{equation}
The group scheme $\tilA$ is well-known to experts in real Lie groups, and will be recalled in Section \ref{ss:A}. It is a central extension of $\mu^{r}_{2}$ by $\mu_{2}$, where $r$ is the rank of $G$. Let $\Ql[\tilA(\Qbar)]$ be group algebra of $\tilA(\Qbar)$, and let $\Ql[\tilA(\Qbar)]_{\odd}$ be the subspace where the central $\mu_{2}$ acts via the sign representation.

Let $\tilA^{(2)}:=\tilA\times\tilA/\Delta(\mu_{2})$, where $\Delta(\mu_{2})$ means the diagonal copy of the central $\mu_{2}$. Left and right multiplication and Galois action gives an $\tilA^{(2)}(\Qbar)\rtimes\GQ$-action on $\tilA(\Qbar)$, and hence on the odd part of the group algebra $\QQ[\tilA(\Qbar)]_{\odd}$. 

We will see there is in fact an $\tilA^{(2)}$-action on $\tilY_{a}$, such that the action of the first copy of $\tilA$ gives the torsor \eqref{Atorsor}. Consider the middle dimensional cohomology $\cohoc{2h^{\vee}-2}{\tilY_{x},\Ql}$. We take its direct summand $\cohoc{2h^{\vee}-2}{\tilY_{x},\Ql}_{\odd}$ where the central $\mu_{2}$ of $\tilA^{(2)}$ acts via the sign representation. Then there is an $\tilA^{(2)}(\Qbar)\rtimes\GQ$-equivariant isomorphism
\begin{equation*}
\cohoc{2h^{\vee}-2}{\tilY_{x},\Ql}_{\odd}(h^{\vee}-1)\cong\rho^{\qm}_{x}\otimes\Ql[\tilA(\Qbar)]_{\odd}.
\end{equation*} 
Here $(h^{\vee}-1)$ means Tate twist. This gives a motivic description of $\rho^{\qm}_x$.

\subsection{The general construction}
In the many body of the paper, we work with a simply-connected almost simple split group $G$ over a field $k$ with $\textup{char}(k)\neq2$. We assume $G$ to satisfy two conditions
\begin{enumerate}
\item The longest element in the Weyl group of $G$ acts by $-1$ on the Cartan subalgebra.
\item $G$ is oddly-laced: i.e., the ratio between the square lengths of long roots and short roots of $G$ is odd.
\end{enumerate}
By the Dynkin diagram classification, the above conditions are equivalent to
\begin{equation*}
\text{\em $G$ is of type $A_{1},D_{2n},E_{7},E_{8}$ or $G_{2}$}.
\end{equation*} 
Our goal is to construct $\hatG(\Ql)$-local systems over $\X{k}$, where $\hatG$ is the Langlands dual group of $G$. For all the above $G$, $\hatG$ is the split adjoint group over $\Ql$ which is of the same type as $G$. Our construction again starts with an automorphic representation of $G$.

\noindent{\bf Step I.} Let $F=k(t)$ be the function field of $\PP^{1}_{k}$, where $k$ is a finite field. We start with a set of local conditions that we want an automorphic representation $\pi=\bigotimes'_{v\in|\PP^{1}|}\pi_{v}$ of $G(\AA_{F})$ to satisfy. For each place $v$, let $F_{v}$ be the corresponding local field, then $\pi_{v}$ is an irreducible admissible representation of $G(F_{v})$. The conditions we will consider are
\begin{itemize}
\item $\pi_{1}$ has a nonzero fixed vector under the Iwahori $\bI_{1}\subset G(F_{1})$;
\item $\pi_{\infty}$ has a nonzero fixed vector under the parahoric $\bP_{\infty}\subset G(F_{\infty})$;
\item $\pi_{0}$ has an eigenvector under which $\bP_{0}\subset G(F_{0})$ acts through a nontrivial quadratic character $\mu:\bP_{0}\to\{\pm1\}$;
\item $\pi$ is unramified away from the places $0,1,\infty$.
\end{itemize}
The parahoric subgroup $\bP_{0}$ is constructed from the element $-1\in W$ under the correspondence of Gross, Reeder and Yu \cite{GRY}. Its reductive quotient admits a unique nontrivial quadratic character $\mu$. The parahoric $\bP_{\infty}$ has the same type as $\bP_{0}$. For details, we refer to Section \ref{sss:par} and \ref{sss:sym}.

\noindent{\bf Step II.} Show that such automorphic representations do exist and are very limited in number. This part of the argument relies on a detailed study of the structure of the double coset
\begin{equation}\label{coset}
G(F)\backslash G(\AA_{F})/\left(\bP_{0}\times\bI_{1}\times\bP_{\infty}\times\prod_{v\neq0,1,\infty}G(\calO_{v})\right).
\end{equation} 
Up to making a finite extension of $k$, we have
\begin{equation}\label{mult}
\sum_{\pi\textup{as above}}m(\pi)\dim\pi_{0}^{(\bP_{0},\mu)}\dim\pi_{\infty}^{\bP_{\infty}}=\#ZG.
\end{equation}
Here, $m(\pi)$ is the multiplicity of $\pi$ in the automorphic spectrum, $\pi_{0}^{(\bP_{0},\mu)}$ is the $\mu$-eigenspace under $\bP_{0}$, and $ZG$ is the center of $G$. Note that the central character of $\pi$ has to be trivial, so the multiplicity $\#ZG$ in the above formula is {\em not} the contribution from different central characters.

Neither Step I nor Step II actually appear in the main body of the paper. We start directly with a geometric reinterpretation of the previous two steps, which makes sense for any field $k$ with $\textup{char}(k)\neq2$.
 
\noindent{\bf Step III.} We interpret the double coset \eqref{coset} as the $k$-points of a moduli stack $\Bun_{G}(\bP_{0},\bI_{1},\bP_{\infty})$: the moduli stack of principal $G$-bundles over $\PP^{1}$ with three level structures at $0,1$ and $\infty$ as specified by the parahoric subgroups. In fact we will consider a variant $\Bun=\Bun_{G}(\wt{\bP_{0}},\bI_{1},\bP_{\infty})$ of this moduli stack to accommodate the quadratic character $\mu$. These moduli stacks are defined in Section \ref{ss:moduli}. Automorphic functions on the double coset (modified by the character $\mu$) are upgraded to ``odd'' sheaves on $\Bun$, which are studied in Section \ref{ss:sheaves}. Theorem \ref{th:clean} is crucial in understanding the structure of such odd sheaves: they correspond to representations of the finite group $\tilA$ we mentioned before.

\noindent{\bf Step IV.} We take an irreducible odd sheaf $\calF$ on $\Bun$, and apply geometric Hecke operators to it. A geometric Hecke operator $\TT(\calK,-)$ is a geometric analog of an integral transformation, which depends on a ``kernel function'' $\calK$. The kernel $\calK$ is an object in the ``Satake category'', which is equivalent to the category of algebraic representations of $\hatG$. The resulting sheaf $\TT(\calK,\calF)$ is over $\Bun\times(\X{})$. In Theorem \ref{th:eigen}(1), we prove that $\calF$ is an eigen object under geometric Hecke operators: every Hecke operator $\TT(\calK,-)$ transforms $\calF$ to a sheaf $\calF\boxtimes\calE(\calK)$ on $\Bun\times\X{}$, with $\calE(\calK)$ a local system on $\X{}$. The collection $\{\calE(\calK)\}_{\calK}$ forms a tensor functor from the Satake category (which is equivalent to $\Rep(\hatG)$) to the category of local systems on $\X{}$, and gives the desired $\hatG$-local system $\calE$ on $\X{}$.

We see that the local system $\calE$ depends on the odd sheaf $\calF$. In fact there are exactly $\#ZG$ odd central characters $\chi$ of $\tilA$, each giving a unique irreducible representation $V_\chi$ of $\tilA$. Each $V_\chi$ gives an odd sheaf $\calF_{\chi}$ on $\Bun$, and hence a $\hatG$-local system $\calE_{\chi}$.

\noindent{\bf Step V.} When $k=\QQ$, the previous step gives a $\hatG(\Qlbar)$-local system on either $\X{\QQ}$ or $\X{\QQ(i)}$ depending on whether the central character $\chi$ can be defined rationally or not. We need to apply two descent arguments. First we descend $\calE_{\chi}$ from $\X{\QQ(i)}$ to $\X{\QQ}$, which is stated in Theorem \ref{th:eigen}(2) and proved in Section \ref{sss:descent}. Next we descend the coefficient field from $\Qlbar$ to $\Ql$, which is stated in Theorem \ref{th:eigen}(3) and proved in Section \ref{sss:Ql}. For reader who does not care much about the rationality issues, this step can be ignored.

\noindent{\bf Step VI.} Finally in Proposition \ref{p:int} we extend the local system $\calE_{\chi}$ from $\X{\QQ}$ to $\X{\Z}$, such that its restriction along $\X{\FF_{p}}$ is the same as $\calE_{\chi,\FF_{p}}$ we constructed using the base field $\FF_{p}$. 

\subsubsection{} We indicate where the main results are proved.
\begin{itemize} 
\item For $k$ an algebraically closed field, the density of the image of $\rho_k$ is proved in Theorem \ref{th:geommono}, whose proof depends on detailed analysis of local monodromy in Section \ref{ss:loc}.

\item The fact that $\rho_x$ is motivic or potentially motivic is proved in Proposition \ref{p:motivic}.

\item To see there exists a rational number $x$ such that $\rho_{x}$ has dense image, we only need to use a variant of Hilbert irreducibility \cite[Theorem 2]{T}. However, in Corollary \ref{p:Qmono} we prove a more effective result: it gives sufficient conditions for $\rho_{x}$ to have dense image, and also proves there are infinitely non-isomorphic $\rho_x$'s. The proof relies on a general property of Galois representations given in Proposition \ref{p:Galrep}.

\item The application to the inverse Galois problem (Theorem \ref{th:invGalois}) is given in Section \ref{ss:invGal}.
\end{itemize}

\subsection{Conjectures and generalizations} In Section \ref{ss:further}, we list some conjectural properties about the local and global monodromy of the local systems we construct. It is also possible to determine the Hodge structure carried by the Betti realization of our local systems.

Groups without $-1$ in their Weyl groups are not treated in this paper, but they can be treated in a similar way by adding a twisting. Instead of working with the constant group scheme $G\times\PP^{1}$, we can work with a quasi-split form of it which splits over the double covering $\PP^{1}\to\PP^{1}$ ramified at $0$ and $\infty$, with the twisting given by the involution in $\Out(G)$. Modifying the construction in this paper properly, we expect to get $\hatG\rtimes\Out(\hatG)$-local systems on $\X{}$. In particular, we expect to realize $E_{6}$ as a motivic Galois group in this way.


\section{The automorphic sheaves}
Throughout the paper, $k$ is a field with $\textup{char}(k)\neq2$. Let $\ell$ be a prime number different from $\textup{char}(k)$. {\em Spaces without subscripts are defined over $k$}.

\subsection{Group-theoretic data}
\subsubsection{The group $G$} Let $G$ be a split almost simple simply-connected group over $k$. We fix a maximal torus $T$ of $G$ and a Borel $B$ containing it. Therefore we get a based root system $\Delta_G\subset\Phi_G\subset\xch(T)$ (where $\Phi_G$ stands for roots and $\Delta_G$ for simple roots) and a Weyl group $W$. We make the following assumption
\begin{equation}\label{ass:-1}
\textup{\em Assume the longest element in $W$ acts as -1 on $\xch(T)$}.
\end{equation}
This means $G$ is of type $A_{1},B_{n},C_{n}, D_{2n}, E_{7}, E_{8}, F_{4}$ or $G_{2}$.

\subsubsection{Loop group}
We partially follow the notations of \cite[Section 1 and 2]{PR}. Let $LG$ be the loop group associated to $G$: it represents the functor $R\mapsto G(R((t)))$ where $R$ is any $k$-algebra and $t$ is a formal variable. Similarly we define the group of positive loops $L^{+}G:R\mapsto G(R[[t]])$. In practice, we have a smooth curve and we denote its completion at a $k$-point $x$ by $\calO_{x}$, which is isomorphic to $k[[t]]$ but not canonically so. Let $F_{x}$ be the field of fractions of $\calO_{x}$. We define $L^{+}_{x}G$ (resp. $L_{x}G$) to be the group (ind-)scheme over the residue field $k(x)$ representing the functor $R\mapsto G(R\htimes_{k(x)}\calO_{x})$ (resp. $R\mapsto G(R\htimes_{k(x)}F_{x})$).

\subsubsection{Parahoric subgroups}\label{sss:par} By Bruhat-Tits theory, for each facet $\una$ in the Bruhat-Tits building $\frB$ of $G(k((t)))$, there is a smooth group scheme $\calP_{\una}$ over $k[[t]]$ with connected fibers whose generic fiber is $G\times_{\Spec k}\Spec k((t))$. We call such $\calP_{\una}$ a {\em Bruhat-Tits group scheme}. Let $\bP_{\una}$ be the group scheme over $k$ representing the functor $R\mapsto \calP(R[[t]])$. This is a pro-algebraic subgroup of $LG$, called a parahoric subgroup. The conjugacy classes of parahoric subgroups of $LG$ are classified by proper subsets of the nodes of the affine Dynkin diagram of $G$.

The group $L^{+}G$ is a particular parahoric subgroup of $LG$, corresponding to the Bruhat-Tits group scheme $\calG=G\times\Spec k[[t]]$. The Borel $B\subset G$ gives another parahoric subgroup called an Iwahori subgroup $\bI\subset L^{+}G\subset LG$: $\bI$ represents the functor $R\mapsto\{g\in G(R[[t]]);g\mod t\in B(R)\}$.

In \cite{GRY}, Gross, Reeder and Yu define a map from regular elliptic conjugacy classes of $W$ to certain conjugacy classes of parahoric subgroups of $LG$. In particular, the longest element $-1\in W$ corresponds to a parahoric subgroup $\bP$ of $LG$, which we assume to contain the Iwahori subgroup $\bI$. Here is an explicit description of their correspondence in this special case. Recall the maximal torus $T$ gives an apartment $\frA(T)$ in the building $\frB$, which is a torsor under $\xcoch(T)\otimes\RR$. The parahoric subgroup $L^+G$ corresponds to a facet which is a point in $\frA(T)$. Using this point as the origin, we may identify $\frA(T)$ with $\xcoch(T)\otimes\RR$. Let $\rho^{\vee}$ be half of the sum of positive coroots of $G$, which is a vector in $\xcoch(T)\otimes\RR$. Hence we may view $\frac{1}{2}\rho^{\vee}\in\frA(T)$, which lies in a unique facet $\una$, and determines a parahoric subgroup $\bP_{\una}$. This is the parahoric $\bP$ we shall consider.

\subsubsection{The symmetric subgroup}\label{sss:sym}
Let $K$ be the maximal reductive quotient of $\bP$. This is a connected split reductive group over $k$. Since $\bP$ is defined using a facet in the apartment $\frA(T)$, $K$ contains $T$ as a maximal torus. The root system $\Phi_{K}\subset\xch(T)$ of $K$ is then a sub-root system of $\Phi_{G}$. A root $\alpha\in\Phi_{G}$ belongs to $\Phi_{K}$ if and only if there is an affine root $\alpha+n\delta$ ($n\in\ZZ$, $\delta$ is the imaginary root) vanishing at $\frac{1}{2}\rho^{\vee}$, i.e., if and only if $\jiao{\rho^{\vee},\alpha}$ is an even integer. Let $G^{\ad}$ be the adjoint form of $G$ with maximal torus $T^{\ad}=T/ZG$ and note that $\rho^{\vee}\in\xcoch(T^{\ad})$, hence we have an element $\rho^{\vee}(-1)\in G^{\ad}(k)$ of order at most 2. This is a {\em split Cartan involution} of $G$, as we shall see in Section \ref{sss:Cartaninv}. By identifying the root system, we have
\begin{equation*}
K=G^{\rho^{\vee}(-1)}.
\end{equation*}
The Dynkin diagram of $K$ is obtained from the affine Dynkin diagram of $G$ by removing either one node (if $G$ is not of type $A_1$ or $C_{n}$), or two extremal nodes (if $G$ is of type $A_1$ or $C_{n}$), and all the adjacent edges. We tabulate the types of $K$ in each case:

\begin{center}
\begin{tabular}{|l|l|}
\hline
Type of $G$ & Type of $K$ \\ \hline
$A_{1}$ & $\Gm$  \\ \hline
$B_{2n}$ & $B_{n}\times D_{n}$ \\ \hline
$B_{2n+1}$ & $B_{n}\times D_{n+1}$ \\ \hline
$C_{n}$ & $A_{n-1}\times\Gm $ \\ \hline
$D_{2n}$ & $D_{n}\times D_{n}$ \\ \hline
$E_{7}$ & $A_{7}$ \\ \hline
$E_{8}$ & $D_{8}$ \\ \hline
$F_{4}$ & $A_{1}\times C_{3}$ \\ \hline
$G_{2}$ & $A_{1}\times A_{1}$ \\ \hline
\end{tabular}
\end{center}
Note that if $k=\RR$, then $K$ has the same type as a maximal compact subgroup of the split real group $G(\RR)$.

We analyze how far $K$ is from being simply-connected.
\begin{lemma}\label{l:Knotsc} Let $\ZZ\Phi^\vee_K\subset\xcoch(T)$ be the coroot lattice of $K$. Then
\begin{equation*}
\xcoch(T)/\ZZ\Phi_K^\vee\cong\begin{cases}\ZZ & G \textup{ is of type }A_1 \textup{ or } C_n \\ \ZZ/2\ZZ & \textup{otherwise}\end{cases}
\end{equation*}
\end{lemma}
\begin{proof}
When $G$ is not of type $A_1$ or $C_n$, the simple roots of $K$ are $\Delta_{K}=(\Delta\backslash\{\alpha'\})\cup\{-\theta\}$, where $\theta$ is the highest root of $G$ and $\alpha'$ is the node we remove from the affine Dynkin diagram of $G$. Hence the coroot lattice of $K$ is spanned by the simple coroots $\alpha^{\vee}$ of $G$ with the only exception $\alpha'^{\vee}$,  which is replaced by $\theta^{\vee}$. We write $\theta^{\vee}=\sum_{\alpha\in\Delta_G}c(\alpha)\alpha^{\vee}$, then we have a canonical isomorphism
\begin{equation*}
\xcoch(T)/\ZZ\Phi^{\vee}_{K}\cong\ZZ/c(\alpha')\ZZ.
\end{equation*}
By examining all the Dynkin diagrams, we find $c(\alpha')$ is always equal to 2.

If $G$ is of type $A_1$ or $C_n$, then $K\cong\GL_n$ has fundamental group $\ZZ$.
\end{proof}

\subsubsection{Canonical double covering}\label{sss:doublecover} In all cases, there is a canonical double covering
\begin{equation*}
1\to\mu_{2}\to\tilK\to K\to1
\end{equation*}
with $\tilK$ a connected reductive group. To emphasize we denote $\ker(\tilK\to K)$ by $\mu^{\ker}_2$. When $G$ is {\em not} of type $A_1$ or $C_n$, $\tilK$ is the simply-connected form of $K$.

\subsection{Moduli stacks}\label{ss:moduli}
\subsubsection{Moduli of bundles with parahoric level structures} Fix a set of $k$-points $S\subset\PP^{1}(k)$. For each $x\in S$, fix a parahoric subgroup $\bP_{x}\subset L_{x}G$. Generalizing the definition in \cite[Section 4.2]{Yun}, we can define the an algebraic stack $\Bun_{G}(\bP_{x};x\in S)$ classifying $G$-torsors on $\PP^{1}$ with $\bP_{x}$-level structures at $x$ (in \cite[Section 4.2]{Yun} we only considered the case $S$ is a singleton, but the construction obviously generalizes to the case of multiple points).  
 
Let $\bP_{0}\subset L_{0}G$ be a parahoric subgroup of type $\bP$ defined in Section \ref{sss:par}, such that $\bP_{0}$ contains the standard Iwahori subgroup $\bI_{0}$ defined by $B$. At $\infty\in\PP^{1}$, let $\bI^{\opp}_{\infty}\subset L_{\infty}G$ be the Iwahori subgroup defined by the Borel $B^{\opp}\supset T$ opposite to $B$. Let $\bP_{\infty}\subset L_{\infty}G$ be the parahoric subgroup of type $\bP$ such that $\bP_{\infty}\supset\bI^{\opp}_{\infty}$. The maximal reduction quotients of $\bP_0$ and $\bP_\infty$ are denoted by $K_0$ and $K_\infty$, which are both canonically isomorphic to the symmetric subgroup $K$ of $G$. We study the moduli stack $\Bun_{G}(\bP_{0},\bP_{\infty})$ in this section.

\subsubsection{Birkhoff decomposition}\label{sss:birk} Let $\calP_{\AA^1}$ be the trivial $G$-torsor over $\AA^1$ together with the tautological $\bP_0$-level structure at $0$. Let $\Gamma_0$ be the group ind-scheme of automorphisms of $\calP_\AA^1$: for any $k$-algebra $R$, $\Gamma_0(R)=\Aut_{\AA^1_R}(\calP_{\AA^1_R})$. Recall from \cite[Proposition 1.1]{HNY} (where the Iwahori version was stated; our parahoric situation follows easily from the Iwahori version) that we have an isomorphism of stacks
\begin{equation*}
\Gamma_0\backslash L_\infty G/\bP_\infty\isom\Bun_G(\bP_0,\bP_\infty).
\end{equation*}
Moreover, we have the Birkhoff decomposition
\begin{equation}\label{Birkhoff}
G(\kbar)=\bigsqcup_{W_{K}\backslash\tilW/W_{K}}\Gamma_{0}(\kbar)\tilw\bP_{\infty}(\kbar).
\end{equation}
Here $\tilW$ is the affine Weyl group of $G$. In \eqref{Birkhoff}, every double coset $\tilw\in W_{K}\backslash\tilW/W_{K}$ represents a $k$-point $\calP_{\tilw}\in\Bun_G(\bP_0,\bP_\infty)$, which is given by gluing $\calP_{\AA^1}$ with the parahoric subgroup $\Ad(\tilw)\bP_\infty$ of $L_\infty G$. The automorphism group of $\calP_{\tilw}$ is $\Gamma_0\cap\Ad(\tilw)\bP_0$ (as a subgroup of $L_\infty G$).

In particular, the unit coset $1\in W_{K}\backslash\tilW/W_{K}$ corresponds to a point $\star=\calP_0\in\Bun_{G}(\bP_{0},\bP_{\infty})$, which the trivial $G$-torsor $\calP_0$ on $\PP^1$ with tautological $\bP_0$ and $\bP_\infty$-level structures. The automorphism group of $\star$ is $\Gamma_{0}\cap\bP_{\infty}$, which maps isomorphically to both $K_{0}$ and $K_{\infty}$. Hence we get an embedding $\{\star\}/K\hookrightarrow\Bun_{G}(\bP_{0},\bP_{\infty})$.  

\begin{lemma}\label{l:j0affine}
The embedding $j_{0}:\BB K=\{\star\}/K\hookrightarrow\Bun_{G}(\bP_{0},\bP_{\infty})$ is open and affine.
\end{lemma}
\begin{proof}
Let $d=\dim G$ and let $\Bun_{d}$ denote the moduli stack of rank $d$ vector bundles on $\PP^{1}$. We have a natural morphism
\begin{equation}\label{GLg}
\Ad^{+}:\Bun_{G}(\bP_{0},\bP_{\infty})\to\Bun_{d}
\end{equation}
defined as follows. Fix an affine coordinate $t$ on $\AA^1$. According to the definition of parahoric level structures in \cite[Section 4.2]{Yun}, an object $(\calP,\calP_{\bP_{0}},\calP_{\bP_{\infty}})\in\Bun_{G}(\bP,\bP_{\infty})(R)$ ($R$ is a locally noetherian $k$-algebra) is a $G$-torsor $\calP$ together with a $\bP_{0}$-reduction $\calP_{\bP_{0}}$ of $\calP|_{\Spec R((t))}$ ( viewed as an $L_0G$-torsor over $\Spec R$) and a $\bP_{\infty}$-reduction $\calP_{\bP_{\infty}}$ of $\calP|_{\Spec R((t^{-1}))}$ (viewed as an $L_{\infty}G$-torsor over $\Spec R$). Let $\bP_0^+=\ker(\bP_0\to K_0)$. From $(\calP,\calP_{\bP_0},\calP_{\bP_\infty})$, we get an $R[[t]]$-submodule $\Ad^{+}(\calP_{\bP_{0}})\subset\Ad(\calP)|_{\Spec R((t))}$ (associated to the adjoint action of $\bP_{0}$ on $\Lie\bP^{+}_{0}$) and an $R[[t^{-1}]]$-submodule $\Ad(\calP_{\bP_{\infty}})\subset\Ad(\calP|_{\Spec R((t^{-1}))})$ (associated to the adjoint action of $\bP_{\infty}$ on $\Lie\bP_{\infty}$). We may glue these two modules with the adjoint bundle $\Ad(\calP)|_{\GG_{m,R}}$ to get a vector bundle $\Ad^{+}(\calP,\calP_{\bP_{0}},\calP_{\bP_{\infty}})\in\Bun_{d}(R)$. This finishes the definition of the morphism \eqref{GLg}.

Alternatively, $\Ad^{+}(\calP_{\bP_{0}})$ and $\Ad(\calP)|_{\GG_{m,R}}$ glue to give a vector bundle on $\AA^1_R$, whose global sections form an $R[t]$-submodule $\Lambda_{0}$ of the $R[t,t^{-1}]$-module $\Gamma(\GG_{m,R},\Ad(\calP))$; $\Ad(\calP_{\bP_{\infty}})$ and $\Ad(\calP)|_{\GG_{m,R}}$ glue to give a vector bundle on $\PP^1_R-\{\infty\}$, whose global sections form an $R[t^{-1}]$ submodule $\Lambda_{\infty}$ of $\Gamma(\GG_{m,R},\Ad(\calP))$. We therefore get a natural homomorphism
\begin{equation}\label{2term}
\Lambda_{0}\oplus\Lambda_{\infty}\to\Gamma(\GG_{m,R},\Ad(\calP)).
\end{equation}
We may view this as a two-term complex $K_\calP$ placed at degree $0$ and $1$. The cohomology of this complex $\cohog{i}{K_{\calP}}$ computes the ($\check{\text C}$ech) cohomology $\cohog{i}{\PP^{1},\Ad^{+}(\calP,\calP_{\bP_{0}},\calP_{\bP_{\infty}})}$. From the explicit description of the points $\calP_{\tilw}\in\Bun_G(\bP_0,\bP_\infty)$, one easily sees that if $R$ is an algebraically closed field,  $(\calP,\calP_{\bP_{0}},\calP_{\bP_{\infty}})$ is isomorphic to the trivial object $\calP_{0}$ if and only if \eqref{2term} is injective, in which case it is also an isomorphism. For an arbitrary base $R$, $(\calP,\calP_{\bP_0},\calP_{\bP_\infty})$ is in the image of $j_0$ if and only if its base change to every geometric fiber is isomorphic to $\calP_0$,  which is then equivalent to saying that \eqref{2term} is an isomorphism, i.e., $\cohog{i}{\PP^{1},\Ad^{+}(\calP,\calP_{\bP_{0}},\calP_{\bP_{\infty}})}=0$ for all $i$. 

As remarked in \cite[Proof of Corollary 1(1) for $\GL_{d}$]{HNY}, the only vector bundle on $\PP^{1}$ of rank $d$ with trivial cohomology is the bundle $\calO(-1)^{\oplus d}$, which is cut out by the non-vanishing of a section of the inverse of the determinant line bundle on $\Bun_{d}$. As discussed above, we have a 2-Cartesian diagram
\begin{equation}\label{CartBun}
\xymatrix{\{\star\}/K\ar@{^{(}->}[r]^{j_{0}}\ar[d] & \Bun_{G}(\bP_{0},\bP_{\infty})\ar[d]^{\Ad^{+}}\\
\{\calO(-1)^{\oplus d}\}/\GL_{d}\ar@{^{(}->}[r]^{j_{\GL}} & \Bun_{\GL_{d}}}
\end{equation}
where the embedding $j_{\GL}$ is affine because it is given by the non-vanishing of a section of a line bundle. Hence $j_{0}$ is also affine.
\end{proof}

\subsubsection{A variant}
Let $\wt{\bP_{0}}=\bP_{0}\times_{K_{0}}\tilK_{0}$, where $\tilK_0$ is the canonical double cover of $K_0$ as in Section \ref{sss:doublecover}. We would like to define the moduli stack $\Bun_{G}(\wt{\bP_{0}},\bP_{\infty})$. We need to explain what we mean by $\wt{\bP_{0}}$-level structures since $\wt{\bP_{0}}$ is not a subgroup of $L_0G$. We may first form the moduli stack $\Bun_{G}(\bP^{+}_{0},\bI_{1},\bP_{\infty})$ where $\bP^{+}_{0}=\ker(\bP_{0}\to K_{0})$. Although $\bP^{+}_{0}$ is not a parahoric subgroup of $L_{0}G$, the definition in \cite[Section 4.2]{Yun} obviously extends to this situation. This is a $K_{0}$-torsor over $\Bun_{G}(\bP_{0},\bP_{\infty})$. We define $\Bun_G(\wt{\bP_{0}},\bP_{\infty})$ as the quotient stack
\begin{equation*}
\Bun_{G}(\wt{\bP_{0}},\bP_{\infty}):=[\tilK_{0}\backslash\Bun_{G}(\bP^{+}_0,\bP_\infty)]
\end{equation*}

\subsubsection{Level structure at three points}\label{sss:3point} Let $\bI_{1}\subset L_{1}G$ be the Iwahori subgroup defined by $B$. We let
\begin{equation*}
\Bun:=\Bun_{G}(\wt{\bP_0},\bI_1,\bP_\infty).
\end{equation*}
As above, we may first form the moduli stack $\Bun^{+}:=\Bun_{G}(\bP^{+}_{0},\bI_{1},\bP_{\infty})$ and then define $\Bun$ as the quotient $[\tilK_{0}\backslash\Bun^{+}]$.

The preimage of $\star\in\Bun_{G}(\bP_{0},\bP_{\infty})$ in $\Bun^{+}$ can be identified with the flag variety of $\fl_{G}$ of $G$, corresponding to all Borel reductions of the trivial $G$-torsor at $t=1$. Therefore we have a canonical open embedding $j^{+}_{1}:\fl_{G}\subset\Bun^+$, or $j_{1}:\tilK\backslash\fl_{G}\hookrightarrow\Bun$. By \cite[Corollary 4.3(i)]{Sp}, $K$ acts on $\fl_{G}$ with finitely many orbits. Therefore there is a unique open $K$-orbit $U\subset\fl_{G}$. We have open embeddings
\begin{equation*}
j:\tilK\backslash U\subset\tilK\backslash\fl_{G}\xrightarrow{j_1}\Bun.
\end{equation*}

\subsubsection{A base point}\label{sss:basepoint} We would like to fix a point $u_{0}\in U(k)$, and henceforth assume
\begin{equation}\label{ass:u0}
\text{\em The base field $k$ is chosen such that $U(k)\neq\emptyset$.}
\end{equation}
We argue that this assumption is not too restrictive. Let $\unU$ be the canonical model of $U$ over $\ZZ[1/2]$ (obtained as the open $\unK$-orbit on $\fl_{\unG}$, where $\unK$ and $\unG$ are split forms over $\ZZ[1/2]$). Since $U_{\QQ}$ is a rational variety, its $\QQ$-points are Zariski dense. Therefore $\unU(\QQ)\neq\emptyset$. We fix $u_{0}\in\unU(\QQ)$, which then extends to a point $\unu_{0}\in \unU(\ZZ[1/2N_0])$ for some positive integer $N_0$.

From the above discussion, whenever char$(k)=0$ or char$(k)$ does not divide $2N_0$, the point $\unu_{0}$ induces a point of $U(k)$, so that the assumption \eqref{ass:u0} holds. We still denote this $k$-point by $u_{0}$.

\subsection{A central extension of a finite 2-group}\label{ss:A}
The point $u_{0}$ fixed above gives a Borel subgroup $B_{0}\subset G$ which is in the most general position with $K$. Let $A=K\cap B_{0}$ be the stabilizer of $u_{0}$ under $K$. Projection to the Cartan quotient $B_{0}\to T_{0}$ induces an isomorphism
\begin{equation*}
A\cong T_{0}[2].
\end{equation*}
Let $\tilA\subset A$ be the preimage of $A$ in $\tilK$. This is a finite group scheme over $k$, and it fits into an exact sequence
\begin{equation}\label{cenext}
1\to\mu^{\ker}_{2}\to\tilA\to A\to1.
\end{equation}
Since $A$ is a discrete group over $k$, we identify $A$ with $A(k)=A(\kbar)$. However, $\tilA$ may not be a discrete group over $k$, so it is important to distinguish among $A, A(k)$ and $A(\kbar)$.

\subsubsection{The structure of $\tilA(\kbar)$} The discussion below works for any simply-connected almost simple $G$. Fix a pinning of $G$ (including $B$ and $T$ we fixed before) and let $\theta\in N_{G}(T)(k)\rtimes\Out(G)$ be a lifting of $-1\in W\rtimes\Out(G)$ (here $\Out(G)$ is the group of pinned automorphisms of $G$). Let $K'=G^{\theta}$. Now the standard Borel $B$ is in general position with $K'$, so that $A':=K'\cap B=T[2]$. Let $\tilA'$ be the preimage of $A'$ in $K'$. When $-1\in W$, $\theta$ is in fact conjugate to $\rho^{\vee}(-1)$ over $\kbar$ because they are both split Cartan involutions (see Section \ref{sss:Cartaninv}). Therefore $K'$ is conjugate to $K$, and $\tilA'(\kbar)\cong\tilA(\kbar)$. The structure of $\tilA'(\kbar)$ is worked out in \cite{ABPTV}. Below we state the results for $\tilA(\kbar)$ directly, although they are proved for $\tilA'(\kbar)$ in \cite{ABPTV}. 

Associated to the central extension \eqref{cenext} there is a quadratic form
\begin{equation*}
q:A\to\mu^{\ker}_{2}
\end{equation*}
which assigns to $a\in A$ the element $\tila^{2}\in\mu^{\ker}_{2}$, here $\tila\in\tilA(\kbar)$ is any lifting of $a$. Associated to this quadratic form there is a symplectic form (or, which amounts to the same in characteristic 2, a symmetric bilinear form)
\begin{equation*}
\jiao{\cdot,\cdot}:A\times A\to\mu^{\ker}_{2}
\end{equation*} 
defined by
\begin{equation*}
\jiao{a,b}=q(ab)q(a)q(b).
\end{equation*}
The symplectic form may be computed by the commutator pairing
\begin{equation*}
\jiao{a,b}=\tila\tilb\tila^{-1}\tilb^{-1}
\end{equation*}
where $\tila,\tilb\in\tilA(\kbar)$ are liftings of $a$ and $b$.

On the coroot lattice $\Lambda^{\vee}=\ZZ\Phi^{\vee}_{G}$, we have a unique $W$-invariant symmetric bilinear form
\begin{equation*}
(\cdot,\cdot):\Lambda^{\vee}\times\Lambda^{\vee}\to\ZZ
\end{equation*}
such that $(\alpha^{\vee},\alpha^{\vee})=2$ if $\alpha^{\vee}$ is a short coroot. Note that for types $B_{n}, C_{n}$ and $F_{4}$, $(\alpha^{\vee},\alpha^{\vee})=4$ if $\alpha^{\vee}$ is a long coroot; for type $G_{2}$, $(\alpha^{\vee},\alpha^{\vee})=6$ if $\alpha^{\vee}$ is a long coroot.

\begin{lemma}\label{l:A} Let $G$ be a split simply-connected almost simple group over $k$. Then
\begin{enumerate}
\item Identifying (the $k$-points of) $A\cong T_{0}[2]$ with $\Lambda^{\vee}/2\Lambda^{\vee}$, we have
\begin{equation}\label{qform}
q(a)=(-1)^{(a,a)/2}\in\mu^{\ker}_2, \textup{ for }a\in A\cong\Lambda^{\vee}/2\Lambda^{\vee}. 
\end{equation}
\item Let $A_{0}$ be the kernel of the symplectic form $\jiao{\cdot,\cdot}$. If $G$ is oddly-laced or $G_{2}$, then $A_{0}=ZG[2]$.
\end{enumerate}
\end{lemma}
\begin{proof}
(1) Our pairing $(\cdot,\cdot)$ on $\Lambda^{\vee}$ is the same one as in \cite[Section 3]{ABPTV}. The symplectic form $\jiao{\cdot,\cdot}$ is determined by Matsumoto, see \cite[Equation (3.3)]{ABPTV}:
\begin{equation}\label{comm}
\jiao{\alpha^{\vee}(-1),\beta^{\vee}(-1)}=(-1)^{(\alpha^{\vee},\beta^{\vee})}, \textup{ for two coroots }\alpha^{\vee},\beta^{\vee}.
\end{equation}
Also, \cite[Theorem 1.6]{Adams} says that a root $\alpha$ of $G$ is metaplectic (which is equivalent to saying $q(\alpha^{\vee}(-1))=-1\in\mu^{\ker}_{2}$) if and only if $\alpha^\vee$ is {\em not } a long coroot in type $B_n, C_n$ and $F_4$. Therefore, checking the coroot lengths under the pairing $(\cdot,\cdot)$, we find
\begin{equation}\label{qnorm}
q(\alpha^{\vee}(-1))=(-1)^{(\alpha^{\vee},\alpha^{\vee})/2} \textup{ for any coroot }\alpha^{\vee}.
\end{equation}
The two identities \eqref{comm} and \eqref{qnorm} together imply \eqref{qform}. 

(2) follows from \cite[Lemma 3.12]{ABPTV} in which the base field was $\RR$, but clearly $ZG(\RR)=ZG[2]$.
\end{proof}

\subsubsection{Odd (or genuine) representations of $\tilA(\kbar)$} A representation $V$ of $\tilA(\kbar)$ is called {\em odd} or {\em genuine}, if $\mu^{\ker}_{2}$ acts via the sign representation on $V$. Let $\Irr(\tilA(\kbar))_{\odd}$ be the set of irreducible odd representations of $\tilA(\kbar)$. 

Recall $A_{0}\subset A$ is the kernel of the pairing $\jiao{\cdot,\cdot}$. The pairing descends to a nondegenerate symplectic form on $A/A_{0}$. Let $\tilA_{0}\subset\tilA$ be the preimage of $A_{0}$. This is the center of $\tilA$. We have a central extension
\begin{equation*}
1\to \tilA_{0}(\kbar)\to\tilA(\kbar)\to A/A_{0}\to 1
\end{equation*}
whose commutator pairing $A/A_{0}\times A/A_{0}\to\mu^{\ker}_{2}$ is nondegenerate. Let $\tilA_{0}(\kbar)^{*}_{\odd}$ be the set of characters $\chi:\tilA_{0}(\kbar)\to\Qlbar^{\times}$ such that $\chi|_{\mu^{\ker}_{2}}$ is the sign representation. By Stone-von Neumann theorem, for every odd central character $\chi\in\tilA_0(\kbar)^*_\odd$, there is up to isomorphism a unique irreducible $\Qlbar$-representation $V_{\chi}$ of $\tilA(\kbar)$ with central character $\chi$. Therefore we have a canonical a bijection
\begin{equation*}
\Irr(\tilA(\kbar))_{\odd}\isom\tilA_{0}(\kbar)^{*}_{\odd}.
\end{equation*}
In particular, the number of irreducible objects in $\Rep(\tilA,\Qlbar)_{\odd}$ is $\#\tilA_{0}(\kbar)^{*}_{\odd}=\#A_{0}$.

We tabulate the structure of $\tilA_{0}$ for oddly-laced groups:
\begin{center}
\begin{tabular}{|c|c|c|}
\hline
Type of $G$ & $\tilA_{0}$ & $\#\Irr(\tilA(\kbar))_{\odd}$\\
\hline
$E_{8}, G_{2}$ & $\mu_{2}$ & 1\\
\hline
$A_{1}, E_{7}$ & $\mu_{4}$ & 2\\
\hline
$D_{4n}$ & $\mu_{2}^{3}$ & 4\\
\hline
$D_{4n+2}$ & $\mu_{4}\times\mu_{2}$ & 4\\
\hline
\end{tabular}
\end{center}

\subsection{Sheaves}\label{ss:sheaves}

\subsubsection{Equivariant derived category}\label{sss:Keq} Since $\mu_{2}^{\ker}=\ker(\tilK\to K)$ acts trivially on $\Bun^{+}$, an object in $D^{b}(\Bun)=D^{b}_{\tilK}(\Bun^{+})$ is in particular a complex $\calF$ on $\Bun^{+}$ together with an action of $\mu^{\ker}_{2}$. Therefore we have a decomposition
\begin{equation*}
D^{b}_{\tilK}(\Bun^{+})=D^{b}_{\tilK}(\Bun^{+})_{\even}\oplus D^{b}_{\tilK}(\Bun^{+})_{\odd}
\end{equation*}
according to whether $\mu^{\ker}_{2}$ acts trivially (the even part) or through the sign representation (the odd part). We then define
\begin{equation*}
D^b(\Bun)_{\odd}:=D^{b}_{\tilK}(\Bun^{+})_{\odd}.
\end{equation*}
Similarly, we define $D^{b}_{\tilK}(U)_{\odd}$.

\begin{theorem}\label{th:clean}
Assume that $G$ is oddly-laced and $-1\in W$ (i.e., $G$ is of type $A_{1},D_{2n},E_{7},E_{8}$ or $G_{2}$). Then the restriction functor
\begin{equation*}
j^*:D^b(\Bun)_{\odd}\to D^b_{\tilK}(U)_{\odd}
\end{equation*}
is an equivalence of categories with inverse equal to both $j_!$ and $j_*$.
\end{theorem}
The proof of the theorem will occupy Section \ref{ss:pfclean}. If $G$ is of type $B_{n},C_{n}$ or $F_{4}$, the statement in the above theorem does not hold. 

{\em For the rest of the paper, we assume $G$ is of type $A_{1},D_{2n},E_{7},E_{8}$ or $G_{2}$.} 

Next we study the category $\Loc_{\tilK}(U)_{\odd}$, the abelian category of $\Qlbar$-local systems on $U$ which are equivariant under $\tilK$ and on which $\mu^{\ker}_{2}$ acts by the sign representation.

\begin{lemma}\label{l:Gammarep}
Consider the pro-finite group $\Gamma=\tilA(\kbar)\rtimes\Gk$, where $\Gk$ acts in the usual way on the $\kbar$-points of the group scheme $\tilA$. Restriction to the point $u_{0}$ (which was fixed in Section \ref{sss:basepoint}) gives an equivalence of tensor categories
\begin{equation*}
u_0^*:\Loc_{\tilK}(U)_{\odd}\isom\Rep_{\cont}(\Gamma,\Qlbar)_{\odd}
\end{equation*}
where the right side denotes continuous finite-dimensional $\Qlbar$-representations of $\Gamma$ on which $\mu^{\ker}_{2}$ acts by the sign representation.  
\end{lemma}
\begin{proof}
Since $\tilK$ acts on $U$ transitively with stabilizer $\tilA$ at $u_{0}$, pullback to $\{u_{0}\}$ gives an equivalence of tensor categories
\begin{equation*}
u_0^*:\Loc_{\tilK}(U)_{\odd}\isom\Loc_{\tilA}(\Spec k)_{\odd}
\end{equation*}
where the subscript $\odd$ on the right side refers to the action of $\mu^{\ker}_{2}\subset\tilA$.

An $\tilA$-equivariant local system on $\Spec k$ is first of all a lisse $\Qlbar$-sheaf on $\Spec k$, which is the same as a continuous representation $V$ of $\Gk$. The $\tilA$-equivariant structure becomes an $\tilA(\kbar)$-action on $V$. The two actions are compatible in the sense that
\begin{equation*}
\gamma\cdot\mu\cdot v=\gamma(\mu)\cdot\gamma\cdot v,\textup{ for }\gamma\in\Gk,\mu\in\tilA(\kbar),v\in V.
\end{equation*}
Therefore these two actions together give a continuous action of $\Gamma=\tilA(\kbar)\rtimes\Gk$ on $V$. The oddness condition on the local system is equivalent to that the action of $\mu^{\ker}_{2}\subset\Gamma$ is via the sign representation. 
\end{proof}

\subsubsection{Representations of $\Gamma$}\label{sss:RepGamma} Let $V_{\chi}$ be the irreducible odd representation of $\tilA(\kbar)$ with central character $\chi:\tilA_{0}(\kbar)\to\Qlbar^{\times}$. 

When $G$ is of type $D_{2n}, E_{8}$ or $G_{2}$, the group scheme $\tilA_{0}$ is a product of $\mu_{2}$'s, hence $\Gk$ acts trivially on $\tilA_{0}(\kbar)$ and its group of characters. The $\tilA(\kbar)$-module $V_\chi$ then extends to a $\Gamma=\tilA(\kbar)\rtimes\Gk$-module. In fact, $\Gk$ acts on $\tilA(\kbar)$ via some finite quotient $\Gammabar$. For any $\gammabar\in\Gammabar$, let $\leftexp{\gammabar}{V}_\chi$ is the same space $V_\chi$ with $a\in\tilA(\kbar)$ acting as $\gammabar(a)$ in the original action. Since the central characters of $\leftexp{\gammabar}{V}_\chi$ is still $\chi$, there is an $\tilA(\kbar)$-isomorphism $\phi_{\gammabar}:V_\chi\isom\leftexp{\gammabar}{V}_\chi$, unique up to scalar. The obstruction for making the collection $\{\phi_{\gammabar}\}$ into an action of $\Gammabar$ on $V_\chi$ is a class in $\upH^{1}(\Gammabar,\Qlbar^{\times})$, which is trivial since $\Qlbar^{\times}$ is divisible. Therefore we may extend the $\tilA(\kbar)$-action on $V$ to an $\tilA(\kbar)\rtimes\Gammabar$-action, hence a $\Gamma$-action.

When $G$ is of type $A_{1}, D_{4n+2}, E_{7}$, $\mu^{\ker}_2$ is contained in some $\mu_{4}\subset\tilA_0$. Therefore $\Gk$ acts trivially on the odd central characters if only if $\sqrt{-1}\in k$. We henceforth make the assumption:
\begin{equation}\label{ass:i}
\text{\em When $G$ is of type $A_{1}, D_{4n+2}$ or $E_{7}$, we assume that $\sqrt{-1}\in k$.}
\end{equation}
When this is assumed, $\tilA_{0}$ is a discrete group scheme over $k$, and we do not need to distinguish between $\tilA_{0}(k), \tilA_{0}(\kbar)$ and $\tilA_{0}$. The same argument as before shows that we can extend the $\tilA(\kbar)$-action on $V_\chi$ to a $\Gamma$-action.

In all cases, assuming \eqref{ass:i}, the $\Gamma$-module $V_{\chi}$ gives a geometrically irreducible object $\calF_{\chi}\in\Loc_{\tilK}(U,\Qlbar)_{\odd}$ by Lemma \ref{l:Gammarep}. 

\begin{remark} Above we have chosen a $\Gamma$-module structure on $V_{\chi}$ extending the $\tilA(\kbar)$-module structure. A different choice of the extension differs from the chosen one by twisting a character of $\Gk$, and the resulting  $\calF_{\chi}$ also changes by twisting the same character of $\Gk$.
\end{remark}

\subsubsection{Variant of cleanness}\label{sss:varclean} Let $S$ be a scheme over $k$, we may similarly consider $D^{b}(\Bun\times S,\Qlbar)_{\odd}$. The cleanness theorem \ref{th:clean} implies that the functor $(u_{0}\times\id_{S})^{*}$ also induces an equivalence of categories
\begin{equation*}
(u_{0}\times\id_{S})^{*}:D^{b}(\Bun\times S)_{\odd}\isom D^{b}_{\tilK}(U\times S)_{\odd}\isom D^{b}_{\tilA}(S)_{\odd}
\end{equation*}
where $\tilA$ acts trivially on $S$ in the last expression. 

Assuming \eqref{ass:i}, we may decompose an object in the $D^{b}_{\tilA}(S)_{\odd}$ according to the action of $\tilA_{0}$. This way we get a decomposition for every object $\calH\in D^{b}(\Bun\times S)_{\odd}$:
\begin{equation}\label{chidecomp}
(u_{0}\times\id)^{*}\calH=\bigoplus_{\chi\in\tilA^*_{0,\odd}}(u_{0}\times\id)^{*}\calH_{\chi}.
\end{equation}

\subsubsection{Invariants under a finite group scheme}\label{sss:inv} For every object $\calF\in\Perv_{\tilA}(S)$, we may talk about its sub-object $\calF^{\tilA}\in\Perv(S)$. In fact, Let $k'$ be a finite Galois extension of $k$ over which $\tilA$ becomes a discrete group scheme. Let $\calF_{k'}$ be the pull back of $\calF$ to $S_{k'}=S\otimes_{k}k'$. We first take the perverse subsheaf of $\tilA(k')$-invariants $\calF_{k'}^{\tilA(k')}\subset\calF_{k'}$. The descent datum for $\calF_{k'}$ (descending from $S_{k'}$ to $S$) restricts to a descent datum of $\calF_{k'}^{\tilA(k')}$. Since perverse sheaves satisfy \'etale descent, $\calF_{k'}^{\tilA(k')}$ descends to a perverse sheaf $\calF^{\tilA}$ on $S$.

Recall that for each $\chi$, we have fixed an irreducible odd $\Gamma$-module $V_{\chi}$, which gives rise to an irreducible object $\calF_{\chi}\in\Loc_{\tilK}(U)_{\odd}$. The following lemma is easy to prove by descent argument.

\begin{lemma}\label{l:writeE} Assume \eqref{ass:i}. Let $\calH\in D^{b}(\Bun\times S)_{\odd}$ such that $(u_{0}\times\id_{S})^{*}\calH$ is concentrated in a single perverse degree, then we have a canonical decomposition
\begin{equation*}
\calH\cong \bigoplus_{\chi\in\tilA^{*}_{0,\odd}}(j_{!}\calF_{\chi})\boxtimes(V^{*}_{\chi}\otimes(u_{0}\times\id_{S})^{*}\calH_{\chi})^{\tilA}.
\end{equation*}
\end{lemma}

\subsection{Proof of Theorem \ref{th:clean}}\label{ss:pfclean}
To prove the theorem, it suffices to show that for any $\calF\in D^{b}(\Bun)_{\odd}$, the stalks of $\calF$ outside the open substack $\tilK\backslash U$ are zero. This statement being geometric, we may assume $k$ is algebraically closed. The rest of this subsection is devoted to the proof of this vanishing statement.

\subsubsection{A result of Springer}\label{sss:Spr}
Let $c\in ZG$. Let $\Inv_{c}(G)=\{g\in G|g^2=c\}$ be consist of ``$c$-involutions'' in $G$. 

Fix $\tau\in\Inv_{c}(G)\cap T$. We write the adjoint action of $\tau$ on $G$ as $\leftexp{\tau}{g}:=\tau g\tau^-1$. We recall a result of Springer \cite[Lemma 4.1(i)]{Sp}. Let $S'=\{x\in G|x\leftexp{\tau}{x}=1\}$, then we have an embedding $G/G^{\tau}\hookrightarrow S'$ by $g\mapsto g\leftexp{\tau}{g}^{-1}$. Under this embedding, the left translation of $G$ on $G/G^{\tau}$ becomes the following action of $G$ on $S'$: $g*x=gx\leftexp{\tau}{g}^{-1}$. Springer's results says that every $B$-orbit on $S'$ (through the $*$-action) intersects $N_{G}(T)$. 

Right multiplication by $\tau$ gives an isomorphism $S'\isom\Inv_{c}(G)$, which intertwines the $*$-action of $G$ on $S'$ and the conjugation action of $G$ on $\Inv_{c}(G)$. Therefore we may reformulate Springer's result as: any element in $\Inv_{c}(G)$ is $B$-conjugate to an element in $N_{G}(T)\tau=N_{G}(T)$.

\subsubsection{Split Cartan involutions}\label{sss:Cartaninv}
Let $G$ be any semisimple group with Lie algebra $\frg$. E.Cartan proved an inequality: for any involution $\tau\in\Aut(G)$, we have
\begin{equation*}
\dim\frg^{\tau}\geq\#\Phi_{G}/2.
\end{equation*}
The equality holds if and only if $\tau$ is conjugate to a lifting of $-1\in W\rtimes\Out(G)$ to $N_{G^{\ad}}(T^{\ad})\rtimes\Out(G)$. If the equality holds, we say $\tau$ is a {\em split Cartan involution} of $G$.

Suppose the involution $\tau\in N_{G^{\ad}}(T^{\ad})$, then $\tau$ acts on $\frt=\Lie T$ and permutes the roots of $G$. Let $\Phi^{\tau}_G\subset\Phi_{G}$ be those roots which are fixed by $\tau$. We form the Levi subalgebra
\begin{equation*}
\frl=\frt\oplus\left(\bigoplus_{\gamma\in\Phi^{\tau}_G}\frg_{\gamma}\right).
\end{equation*}
Let $L$ be the corresponding Levi subgroup of $G$. Let $\frl^{\der}\subset\frl$ be the derived algebra, which is a semisimple Lie algebra which root system $\Phi^\tau_G$.

We calculate the dimension of $\frg^{\tau}$. For those roots $\gamma$ which are not fixed by $\tau$, $\frg_{\gamma}\oplus\frg_{\tau(\gamma)}$ has a $\tau$-fixed line. Therefore
\begin{equation*}
\dim\frg^{\tau}=\frac{\#\Phi-\#\Phi^{\tau}}{2}+\dim\frl^{\der,\tau}+\dim\frt^{\tau}-\rank\frl^{\der}\geq\frac{\#\Phi-\#\Phi^{\tau}}{2}+\frac{\#\Phi^{\tau}}{2}=\frac{\#\Phi}{2}.
\end{equation*}
 
Now suppose $\tau$ is a split Cartan involution, the above inequality is an equality. In particular, we have $\dim\frl^{\der,\tau}=\#\Phi^{\tau}_G/2$ and $\dim\frt^{\tau}=\rank\frl^{\der}$. This means
\begin{itemize}
\item $\frt^{\tau}$ is the Cartan subalgebra of $\frl^{\der}$, equivalently $\Phi^{\tau}_G$ span $\frt^{*,\tau}$;
\item $\tau$ restricts to a split Cartan involution on $L^{\der}$.
\end{itemize}

\subsubsection{The big cell}\label{sss:bigcell} We first restrict $\calF$ to the open subset $\tilK\backslash\fl_{G}\subset\Bun$. Let $O(x)=KxB/B$ be a $K$-orbit on $\fl_{G}$ which is not open. The goal is to show that $\calF|_{O(x)}=0$. 

Choose any lifting $\tau_{0}$ of $\rho^{\vee}(-1)\in G^{\ad}$ to $G$, then $c:=\tau_{0}^{2}$ is in the center of $G$, and $K=G^{\tau_{0}}$. Consider the $c$-involution $\tau=x^{-1}\tau_{0} x\in G$. By Springer's result reformulated in Section \ref{sss:Spr}, up to multiplying $x$ by an element in $B$ on the right, we may assume that $\tau\in N_{G}(T)$. Let $[\tau]\in W$ be the image of $\tau$, which satisfies $[\tau]^{2}=1\in W$. Then $O(x)$ is the open orbit if and only if $[\tau]=-1\in W$. Conjugating by $x^{-1}$ gives an isomorphism
\begin{equation*}
O(x)=K/K\cap xBx^{-1}\isom xKx^{-1}/xKx^{-1}\cap B=G^{\tau}/G^{\tau}\cap B.
\end{equation*}
which intertwines the $K$-action on $O(x)$ and the $G^{\tau}$-action on $G^{\tau}/G^{\tau}\cap B$. Let $\nu':\wt{G^{\tau}}\to G^{\tau}$ be the canonical double cover. Under this isomorphism, $\calF|_{O(x)}$ can be viewed as an object in $D^{b}_{\wt{G^{\tau}}}(G^{\tau}/G^{\tau}\cap B)_{\odd}$, where oddness refers to the action of $\ker(\nu')$, which we still denote by $\mu^{\ker}_2$ in the sequel. Since $\wt{G^{\tau}}$ acts on $G^{\tau}/G^{\tau}\cap B$ transitively, each cohomology sheaf of $\calF|_{O(x)}$ is a local system on $G^{\tau}/G^{\tau}\cap B$. Therefore it is enough to show that $\Loc_{\wt{G^{\tau}}}(G^{\tau}/G^{\tau}\cap B)_{\odd}=0$.

Springer \cite[Proposition 4.8]{Sp} also shows that $G^{\tau}\cap B$ is the product of $T^{\tau}$ and a unipotent group $N_{\tau}$ as schemes. Let $\wt{T^{\tau}}\subset\wt{G^{\tau}}$ be the preimage of $T^{\tau}$ in $\wt{G^{\tau}}$. Then
\begin{equation*}
\Loc_{\wt{G^{\tau}}}(G^{\tau}/T^{\tau}N_{\tau})_{\odd}\isom\Loc_{\wt{G^{\tau}}}(\wt{G^{\tau}}/\wt{T^{\tau}})_{\odd}=\Loc_{\wt{T^{\tau}}}(\Spec k)_{\odd}.
\end{equation*}
The category $\Loc_{\wt{T^{\tau}}}(\Spec k)$ is equivalent to the category of $\Qlbar$-representations of $\pi_{0}(\wt{T^{\tau}})$. Therefore, for $\Loc_{\wt{T^{\tau}}}(\Spec k)_{\odd}$ to be nonzero, $\mu^{\ker}_{2}$ must not lie in the neutral component of $\wt{T^{\tau}}$. However, we will show the opposite statement: when $[\tau]\neq-1\in W$, $\mu^{\ker}_2$ lies in the neutral component of $\wt{T^\tau}$. This will imply that $\calF|_{O(x)}=0$ whenever $O(x)$ is not the open orbit.

Since $[\tau]\neq-1\in W$, $\frt^{\tau}\neq0$ and $\Phi^{\tau}\neq\emptyset$ because they have to span $\frt^{\tau}$ by our discussion in Section \ref{sss:Cartaninv}. Therefore the Levi $L$ we introduced there is not a torus. We will show that there exists a coroot $\alpha^{\vee}$ of $G$, invariant under $\tau$, such that $\alpha^{\vee}:\Gm\to G^{\tau}$ does not lift to $\Gm\to\wt{G^{\tau}}$. This would imply that the preimage $\wt{\alpha^{\vee}(\Gm)}$ of $\alpha^{\vee}(\Gm)$ in $\wt{G^{\tau}}$ is again connected, hence $\mu^{\ker}_{2}\subset\wt{\alpha^{\vee}(\Gm)}$ lies in the neutral component of $\wt{T^{\tau}}$.

Let $L_{1}$ be a simple factor of $L$ which is not a torus. Let $\frl_{1}=\Lie L_{1}$. Since $\tau|_{L_{1}}$ is a split Cartan involution, in particular it is nontrivial, there exists a root $\alpha$ of $L_{1}$ which is not a root of $L_{1}^{\tau}$. Assume either $G$ is simply-laced or this $\alpha$ is a long root if $G$ is of type $G_2$. Give a grading of $\frg$ according to the adjoint action of the coroot $\alpha^{\vee}$, so that $\Ad(\alpha^{\vee}(s))$ acts on $\frg(d)$ by $s^{d}$. Since $\alpha$ is a long root of $G$, we have $\frg(d)=0$ for $|d|>2$, and that $\frg(2)=\frg_{\alpha}$. Since $\frg_{\alpha}\subset\frl_{1}$ while $\alpha$ is not a root of $L_{1}^{\tau}$, we conclude that $\frg^{\tau}\cap\frg(2)=0$, i.e, the action  of $\Ad(\alpha^{\vee})$ on $\frg^{\tau}$ has weights $-1,0,1$. Suppose $\alpha^\vee$ lifted to a cocharacter of $\wt{G^\tau}$, it would give an element in the coroot lattice of $K$ which is minuscule. By examining the possible simple factors of $K$ (which are of type $A$ or $D$), we conclude that no element in the coroot lattice of $K$ is minuscule. Therefore $\alpha^{\vee}$ is not liftable to $\wt{G^{\tau}}$. For simply-laced $G$, this already finishes the proof.

It remains to consider the case $G=G_{2}$, and $L$ is a Levi corresponding to a short root $\alpha$ of $G$. Using the grading given by $\alpha^{\vee}$, we have $\frg^{\tau}(3)\neq0$ and $\frg^{\tau}(1)\neq0$. However, $\wt{G^{\tau}}\cong\SL_{2}\times\SL_{2}$, and any homomorphism $\Gm\to\SL_{2}$ acts by even weights under the adjoint representation. Therefore $\alpha^{\vee}$ is not liftable to $\wt{G^{\tau}}$, and the $G_{2}$ case is also settled.

\subsubsection{Other cells} Recall the Birkhoff decomposition \eqref{Birkhoff}. Each double coset $\tilw\in W_{K}\backslash\tilW/W_{K}$ gives a $k$ point of $\Bun_{G}(\bP_{0},\bP_{\infty})$, whose stabilizer is $\Stab(\tilw)=\Gamma_{0}\cap \Ad(\tilw)\bP_{\infty}$. We have a canonical homomorphism by evaluating at $t=0$
\begin{equation*}
\ev_{0}:\Stab(\tilw)\subset\Gamma_{0}\to K_{0}=K.
\end{equation*}
Changing the $\bP_{0}$-level to $\wt{\bP_{0}}$-level, we get a similar Birkhoff decomposition for $\Bun_{G}(\wt{\bP_{0}},\bP_{\infty})$. The $\tilw\in\Bun_{G}(\wt{\bP_{0}},\bP_{\infty})$ has stabilizer
\begin{equation*}
\wt{\Stab}(\tilw):=\Stab(\tilw)\times_{K}\tilK.
\end{equation*} 
Let $\Bun_{G}(\wt{\bP_{0}},\bI_{1},\bP_{\infty})_{\tilw}$ be the preimage of the stratum $\{\tilw\}/\wt{\Stab}(\tilw)$ in $\Bun_{G}(\wt{\bP_{0}},\bI_{1},\bP_{\infty})$. The open stratum corresponding to the unit coset in $W_{K}\backslash\tilW/W_{K}$ has been dealt with in Section \ref{sss:bigcell}. Our goal is to show that $D^{b}(\Bun_{G}(\wt{\bP_{0}},\bI_{1},\bP_{\infty})_{\tilw})_{\odd}=0$ whenever $\tilw$ is not the unit coset in $W_{K}\backslash\tilW/W_{K}$. This together with the discussion in Section \ref{sss:bigcell} finishes the proof of Theorem \ref{th:clean}.

By the explicit description of the $G$-torsor $\calP_{\tilw}$ corresponding to $\tilw\in\Bun_{G}(\wt{\bP_{0}},\bP_{\infty})$ in Section \ref{sss:birk}, $\calP_{\tilw}$ has a tautological trivialization over $\Gm$. The fiber $\tpi^{-1}(\tilw)$ consists of all Borel reductions of $\calP_{\tilw}$ at $t=1$, hence canonically identified with the flag variety $\fl_{G}$.  The stabilizer $\Stab(\tilw)$ is a subgroup of $\Gamma_0\subset G[t,t^{-1}]$, and it acts on the fiber $\tpi^{-1}(\tilw)=\fl_{G}$ via the evaluation map at $t=1$:
\begin{equation*}
\ev_{1}:\Stab(\tilw)=\Gamma_{0}\cap\Ad(\tilw)\bP_{\infty}\subset\Gamma_{0}\subset G[t,t^{-1}]\xrightarrow{\ev(t=1)}G.
\end{equation*}
Therefore we have
\begin{equation*}
\Bun_{G}(\wt{\bP_{0}},\bI_{1},\bP_{\infty})_{\tilw}\cong\wt{\Stab}(\tilw)\backslash\fl_{G}.
\end{equation*}
where $\wt{\Stab}(\tilw)$ acts on $\fl_{G}$ via the quotient $\Stab(\tilw)$ and the homomorphism $\ev_{1}$. We conclude that
\begin{equation*}
D^{b}(\Bun_{G}(\wt{\bP_{0}},\bI_{1},\bP_{\infty})_{\tilw})_{\odd}\cong D^{b}_{\wt{\Stab}(\tilw)}(\fl_{G})_{\odd}
\end{equation*}
We need to show that the latter is zero.

Let $G(\tilw)\subset G$ be the image of $\ev_{1}$. This is a subgroup of $G$ containing $T$. To describe this group more explicitly, we need to know which roots of $G$ appear in $G(\tilw)$.

Let $x$ be the vertex in the building of $L_\infty G$ corresponding to the parahoric $\bP_{\infty}$. Then $x$ lies in the apartment $\frA(T)$ corresponding to the torus $T$, which is a torsor under $\xcoch(T)\otimes\RR$. The affine Weyl group $\tilW$ acts on $\frA(T)$ by affine transformations. The stabilizer of $x$ is precisely the Weyl group of $W_{K}$ (here we use the fact that $G$ is simply connected). Affine roots of $L_\infty G$ are affine functions on $\frA(T)$. Affine roots $\alpha$ appearing in $\Gamma_{0}$ are those such that $\alpha(x)\geq0$; affine roots $\alpha$ appearing in $\Ad(\tilw)\bP_{\infty}$ are those such that $\alpha(\tilw x)\leq0$. Therefore affine roots appearing in $\Stab(\tilw)$ are those such that $\alpha(x)\geq0$ and $\alpha(\tilw x)\leq0$. Write such an affine root as $\alpha=\gamma+n\delta$ where $\gamma\in\Phi_G\cup\{0\}$, $\delta$ is the imaginary root of $L_\infty G$ and $n\in\ZZ$. Evaluating at $t=1$ gives $\gamma$ as a root appearing in $G(\tilw)$. In conclusion, we see that $\gamma\in\Phi_G$ appears in $G(\tilw)$ if and only if there exists $n\in\ZZ$ such that $\gamma(x)+n\geq0$ and $\gamma(\tilw x)+n\leq0$. Since $x=\frac{1}{2}\rho^\vee$, hence $\gamma(x),\gamma(\tilw x)\in\frac{1}{2}\ZZ$, therefore such an integer $n$ exists if 
\begin{itemize}
\item either $\jiao{\gamma,x-\tilw x}>0$;
\item or $\jiao{\gamma,x-\tilw x}=0$ and $\gamma(x)\in\ZZ$.
\end{itemize} 
Note that $\lambda=x-\tilw x$ is a well defined vector in $\xcoch(T)_{\QQ}$, and $\jiao{\gamma,x-\tilw x}$ is the pairing between $\xch(T)_{\QQ}$ and $\xcoch(T)_{\QQ}$. When we say $\gamma(x)\in\ZZ$ we view $\gamma$ as an affine function on $\frA(T)$. Note that $\gamma(x)\in\ZZ$ if and only if $\gamma\in\Phi_{K}\cup\{0\}$, therefore the roots of $G(\tilw)$ are
\begin{equation*}
\{\gamma\in\Phi_G|\jiao{\gamma,\lambda}\geq0, \textup{ and, when equality holds, } \gamma\in\Phi_{K}\}.
\end{equation*}
Let $P_{\lambda}\supset G(\tilw)$ be the parabolic subgroup of $G$ whose roots consist of all $\gamma$ such that $\jiao{\gamma,\lambda}\geq0$. Then we have a Levi decomposition $P_{\lambda}=N_{\lambda}L_{\lambda}$ where $N_{\lambda}$ is the unipotent radical of $P_{\lambda}$ and $L_{\lambda}$ is a Levi subgroup of $P_{\lambda}$ whose roots are $\{\gamma\in\Phi_G|\jiao{\gamma,\lambda}=0\}$. Then $G(\tilw)=N_{\lambda}(L_{\lambda}\cap K)$. Let $W_{\lambda}$ be the Weyl group of $L_{\lambda}$.

Notice that $L_{\lambda}\cap K\subset G(\tilw)$ has a canonical lifting to a subgroup of $\Stab(\tilw)$: this is the subgroup generated by $T$ and containing real affine roots $\alpha$ of the form $\alpha(x)=\alpha(\tilw x)=0$. In particular, we can form $\wt{L_{\lambda}\cap K}=(L_{\lambda}\cap K)\times_{K}\tilK$, which is both a subgroup and a quotient group of $\wt{\Stab}(\tilw)$.

Next we analyze the $G(\tilw)$-orbits on $\fl_{G}$.  Fix a Borel $B$ containing $T$, according to the Bruhat decomposition, we have
\begin{equation*}
G(\tilw)\backslash\fl_{G}=\bigsqcup_{v\in W_{\lambda}\backslash W}G(\tilw)\backslash P_{\lambda}vB/B\twoheadleftarrow\bigsqcup_{v\in W_{\lambda}\backslash W}(L_{\lambda}\cap K)\backslash L_{\lambda}vB/B=\bigsqcup_{v\in W_{\lambda}\backslash W}L_{\lambda}\cap K\backslash\fl_{L_{\lambda}}.
\end{equation*}
The arrow above is a map between stacks which is bijective on the level of points. The last equality uses the fact that $L_{\lambda}vB/B=L_{\lambda}/L_{\lambda}\cap\Ad(v)B$ is the flag variety of $L_{\lambda}$. Therefore, in order to show that $D^{b}_{\wt{\Stab}(\tilw)}(\fl_{G})_{\odd}=0$, it suffices to show that $D^{b}_{\wt{L_{\lambda}\cap K}}(\fl_{L_{\lambda}})_{\odd}=0$.  Note that for $\tilw\notin W_{K}$, $\lambda=x-\tilw x\neq0$, therefore $L_{\lambda}$ is a {\em proper} Levi subgroup of $G$. So our goal becomes to show that $D^{b}_{\wt{L\cap K}}(\fl_{L})_{\odd}=0$ for all proper Levi subgroups $L\subset G$ containing $T$.

Since $K=G^{\tau_0}$ where $\tau_0$ is a lifting of $\rho^\vee(-1)$ which also acts on $L$,  we have $L\cap K=L^{\tau_0}$. Fix a Borel $B_{L}\subset L$ containing $T$. Consider the $L^{\tau_0}$-orbit $O(x)=L^{\tau_0}xB_{L}/B_{L}\subset\fl_{L}$. By Springer's lemma again, up to right multiplying $x$ by an element in $B_{L}$, we may assume that the involution $\tau=x\tau_0x^{-1}$ lies in $N_{L}(T)$. By the same argument as in Section \ref{sss:bigcell}, we reduce to show that $\Loc_{\wt{T^{\tau}}}(\pt)_{\odd}=0$. Here $\wt{T^{\tau}}=T^{\tau}\times_{K}\tilK$. In Section \ref{sss:bigcell}, we have shown that $D^{b}_{\wt{T^{\tau}}}(\pt)_{\odd}=0$ provided $[\tau]\neq-1\in W$. In our case, since $\tau\in N_{L}(T)$ and $L$ is proper, $[\tau]$ can never be $-1$. This finishes the proof.


\section{Construction of the motives}\label{s:cons}

\subsection{Geometric Hecke operators}\label{ss:geomHk}

\subsubsection{The geometric Satake equivalence}\label{sss:Satake} We briefly review the main result of \cite{MV}. We have defined the loop group $LG$ and its parahoric subgroup $L^{+}G$ in Section \ref{sss:par}. Let $\Aut_{\calO}$ be the group scheme of formal change of variables: $R\mapsto$ \{continuous (with respect to the $t$-adic topology) $R$-linear automorphisms of $R[[t]]$\}.

We usually denote the quotient $LG/L^{+}G$ by $\Gr$, the {\em affine Grassmannian}. This is an ind-scheme which is a union of projective varieties of increasing dimension. Define
\begin{equation*}
\Sat^{\geom}=\Perv_{\Aut_{\calO,\kbar}}(L^{+}G_{\kbar}\backslash LG_{\kbar}/L^{+}G_{\kbar},\Ql)=\Perv_{L^{+}G_{\kbar}\rtimes\Aut_{\calO,\kbar}}(\Gr_{\kbar},\Ql).
\end{equation*}
The superscript $^{\geom}$ is to indicate that sheaves in $\Sat^{\geom}$ are over the geometric fiber $\Gr_{\kbar}$. As was shown in \cite[Appendix]{MV}, $\Sat^{\geom}$ is in fact equivalent to the category of perverse sheaves on $\Gr_{\kbar}$ which are constant along $L^{+}G_{\kbar}$-orbits. This category is equipped with a convolution product $*:\Sat^{\geom}\times\Sat^{\geom}\to\Sat^{\geom}$ which makes it a Tannakian category over $\Ql$. The global section functor
\begin{eqnarray*}
\upH^*:\Sat^{\geom}&\to&\Vec_{\Ql}\\
\calK&\mapsto&\cohog{*}{\Gr_{\kbar},\calK}
\end{eqnarray*}
is a fiber functor. The Tannakian group $\Aut^\otimes(\textup{H}^*)$ is isomorphic to the Langland dual group $\hatG$ of $G$. Here $\hatG$ is a split reductive group over $\Ql$. All this is proved in \cite[Section 5-7]{MV}.

Let $\tSat$ be the full subcategory of perverse sheaves on $\Gr$ which belong to $\Sat^{\geom}$ after pulling back to $\Gr_{\kbar}$. Then $\tSat$ also has a convolution product. For each dominant coweight $\lambda\in\xcoch(T)^{\dom}$, the $L^+G$-orbit containing $t^\lambda\in\Gr$ is denoted by $\Gr_{\lambda}$ and its closure is denoted by $\Gr_{\leq\lambda}$. The dimension of $\Gr_\lambda$ is $\jiao{2\rho,\lambda}$. Let $\IC_\lambda\in\tSat$ be the (normalized) intersection cohomology sheaf of $\Gr_{\leq\lambda}$:
\begin{equation*}
\IC_\lambda:=j_{\lambda,!*}\Ql[\jiao{2\rho,\lambda}](\jiao{\rho,\lambda}).
\end{equation*}
where $j_{\lambda}:\Gr_\lambda\hookrightarrow{\Gr}$ is the inclusion.

Let $\Sat\subset\tSat$ be the full subcategory consisting of finite direct sums of  $\IC_\lambda$'s for the various $\lambda\in\xcoch(T)^{\dom}$. We claim that $\Sat$ is closed under convolution. In \cite[\S 3.5]{AB}, Arkhipov and Bezrukavnikov show this is the case when $k$ is a finite field, which implies the general case of $\textup{char}(k)>0$. When $\textup{char}(k)=0$, we only need to consider $k=\QQ$. However, since $\Sat^{\geom}$ is a semisimple category with simple objects $\IC_{\lambda,\Qbar}$, we know $\IC_{\lambda}*\IC_{\mu}\cong\oplus_{\nu}\IC_{\nu}\otimes V^{\nu}_{\lambda,\mu}$ for $\Ql[\GQ]$-modules $V^\nu_{\lambda,\mu}=\Hom_{\Gr}(\IC_\nu,\IC_\lambda*\IC_\mu)$. By the standard argument of choosing an integral model, and the claim for finite fields, we see that the $\GQ$-action on $V^\nu_{\lambda,\mu}$ is trivial when restricted to $\GQp$ for almost all primes $p$. By Chebotarev density, this implies that the $\GQ$-action on $V^\nu_{\lambda,\mu}$ is trivial, and the claim is proved. It is then easy to see that the pullback functor gives a tensor equivalence $\Sat\isom\Sat^\geom$, hence also a tensor equivalence
\begin{equation}\label{Satake}
\Sat\isom\Rep(\hatG,\Ql).
\end{equation}

\subsubsection{The Hecke correspondence}\label{sss:Hk} Consider the following correspondence
\begin{equation}\label{Hkcorr}
\xymatrix{ & \Hk\ar[rr]^{\inv}\ar[dl]_{\oleft{h}}\ar[dr]^{\oright{h}} && \left[\dfrac{LG^+\backslash LG/LG^+}{\Aut_{\calO}}\right]\\
\Bun\times(\X{}) & & \Bun\times(\X{})}
\end{equation}
We need to explain some notations:
\begin{itemize}
\item The stack $\Hk$ is the functor which sends $R$ to the category of tuples $(x,\calP,\calP',\iota)$ where $x\in(\X{})(R)$, $\calP,\calP'\in\Bun(R)$ and $\iota$ is an isomorphism between the $\calP$ and $\calP'$ away from the graph of $x$. In particular, $\iota$ preserves the level structures at $0,1$ and $\infty$.  The maps $\oleft{h}$ and $\oright{h}$ are defined by
\begin{eqnarray*}
\oleft{h}(x,\calP,\calP',\iota)&=&(\calP,x);\\
\oright{h}(x,\calP,\calP',\iota)&=&(\calP',x).
\end{eqnarray*}
\item  The map $\inv$ sends $(x,\calP,\calP',\iota)\in\Hk(R)$ to the relative position between $\calP$ and $\calP'$ in the formal neighborhood of $x$. In fact, choosing a local parameter at $x$, we may identify the formal neighborhood of $x$ in $\X{R}$ with $\Spec R[[t]]$. Choosing trivializations of $\calP$ and $\calP'$ over $\Spec R[[t]]$ we may view $\iota$ as an automorphism of the trivial $G$-torsor over $\Spec R((t))$, hence an element $g\in G(R((t)))=LG(R)$. The image of $g$ as an $R$-point of the quotient stack $\left[\dfrac{LG^+\backslash LG/LG^+}{\Aut_{\calO}}\right]$ does not depend on the choices of trivializations and the local parameter. This defines the map $\inv$.
\end{itemize}

\subsubsection{Geometric Hecke operators} The universal geometric Hecke operator is defined as a bifunctor
\begin{eqnarray}\notag
\notag\tTT:\Sat\times D^b(\Bun\times(\X{}),\Qlbar)&\to& D^b(\Bun\times(\X{}),\Qlbar)\\
\label{defineT}(\calK,\calF)&\mapsto&\oright{h}_!(\oleft{h}^*\calF\otimes_{\Ql}\inv^*\calK).
\end{eqnarray}
Note that $\calK\in\Sat$ always has $\Ql$-coefficients. The bifunctor $\tTT$ is compatible with the tensor structure $*$ on $\Sat$ in the sense that the functor
\begin{eqnarray*}
\Sat&\to&\End(D^b(\Bun\times(\X{}),\Qlbar))\\
\calK&\mapsto&(\calF\mapsto\tTT(\calK,\calF))
\end{eqnarray*}
has a natural structure of a monoidal functor (the monoidal structure on the target is given by composition of endofunctors). We more often use the following functor
\begin{eqnarray*}
\TT:\Sat\times D^{b}(\Bun,\Qlbar)&\to& D^{b}(\Bun\times(\X{}),\Qlbar)\\
\calF &\mapsto& \tTT(\calK,\calF\boxtimes\Qlbar).
\end{eqnarray*}

\subsection{Eigen local systems}

\begin{defn}\label{def:HE}
A {\em Hecke eigensheaf} is a triple $(\calF,\calE,\epsilon)$ where
\begin{itemize}
\item $\calF\in D^b(\Bun,\Qlbar)$;
\item $\calE:\Sat\to\Loc(\X{},\Qlbar)$ is a tensor functor;
\item $\epsilon$ is a system of isomorphisms $\epsilon(\calK)$ for each $\calK\in\Sat$
\begin{equation}\label{eigeneq}
\epsilon(\calK):\TT(\calK,\calF)\isom\calF\boxtimes\calE(\calK)
\end{equation}
which is compatible with the tensor structures (see \cite[discussion after Proposition 2.8]{GDJ}).
\end{itemize}
\end{defn}
The functor $\calE$ in the definition defines a $\hatG$-local system on $\X{}$, which is called the {\em eigen local system} of $\calF$. When we say $\calF\in D^b(\Bun)$ is a Hecke eigensheaf, it means there exists $(\calE,\epsilon)$ as above making the triple $(\calF,\calE,\epsilon)$ a Hecke eigensheaf in the sense of Definition \ref{def:HE}.

\begin{theorem}\label{th:eigen} Suppose $G$ is of type $A_{1},D_{2n},E_{7},E_{8}$ or $G_{2}$. Assume \eqref{ass:u0} holds. 
\begin{enumerate}
\item Assume \eqref{ass:i} also holds. Then for every character $\chi\in\tilA^*_{0,\odd}$, the object $j_!\calF_{\chi}=j_*\calF_{\chi}\in D^{b}(\Bun)_{\odd}$ is a Hecke eigensheaf with eigen local system
\begin{equation*}
\calE_{\chi,\Qlbar}:\Sat\cong\Rep(\hatG,\Ql)\to\Loc(\X{},\Qlbar).
\end{equation*}
\item We do {\em not} assume \eqref{ass:i}. Then $\calE_{\chi,\Qlbar}$ is a priori a $\hatG(\Qlbar)$-local system on $\X{k'}$ where $k'=k(\sqrt{-1})$. There is a canonical way to descend $\calE_{\chi,\Qlbar}$ to a $\hatG(\Qlbar)$-local system on $\X{k}$. Moreover, there is a canonical isomorphism $\calE_{\chi,\Qlbar}\cong\calE_{\chibar,\Qlbar}$ over $\X{k}$ (here $\chibar=\chi^{-1}$).  
\item The $\hatG(\Qlbar)$-local system $\calE_{\chi,\Qlbar}$ descends canonically to a $\hatG(\Ql)$-local system $\calE_\chi$. Moreover, the isomorphism $\calE_{\chi,\Qlbar}\isom\calE_{\chibar,\Qlbar}$ descends to $\calE_\chi\isom\calE_{\chibar}$.
\item If $k$ is a finite field, $\calE_{\chi}(\calK)$ is pure of weight zero for any $\calK\in\Sat$.
\end{enumerate}
\end{theorem}

The proof of this theorem occupies Section \ref{ss:pfeigen}.

\subsubsection{The Beilinson-Drinfeld Grassmannian}\label{sss:BD} Beilinson and Drinfeld define a global analog of the affine Grassmannian. We recall the definition in our context. We define $\GR$ by the Cartesian diagram
\begin{equation*}
\xymatrix{\GR\ar[r]\ar[d]^{\pi} & \Hk\ar[d]^{\oright{h}}\\
\X{}\ar[r]^(.4){u_0\times\id} & \Bun\times(\X{})}
\end{equation*}
Using the moduli interpretation of $\Hk$, we see that $\GR$ classifies triples $(x,\calP,\iota)$ where $x\in (\X{})(R)$, $\calP\in\Bun(R)$ and $\iota$ is an isomorphism $\calP|_{\PP^1_R-\Gamma(x)}\isom\calP_0|_{\PP^1_R-\Gamma(x)}$, where $\Gamma(x)\subset\PP^1_R$ is the graph of $x$. Fixing $x$ the fiber of $\GR$ over $x$ is denoted by $\Gr_x$, which is an ind-scheme over $R$. When $R$ is a field, after choosing a local parameter $t$ at $x$, we may identify $\Gr_x$ with the affine Grassmannian $\Gr$ in Section \ref{sss:Satake}.

For every $\calK\in\Sat$, there is a corresponding object $\calK_{\GR}\in D^{b}(\GR,\Ql)$ obtained by spreading it over $\X{}$: for the construction, see \cite[Section 2.1.3]{Ga}. This is the same as the pullback of $\inv^{*}\calK$ from $\Hk$ to $\GR$.

\subsubsection{Description of the local system} From now on till Section \ref{sss:qmin} we work under both assumptions \eqref{ass:u0} and \eqref{ass:i}, and the coefficient field for all sheaves is $\Qlbar$.

Define $\wt{\GR^{U}}$ by the Cartesian diagram
\begin{equation}\label{u0u0}
\xymatrix{\wt{\GR^{U}}\ar[r]\ar[d]^{\tpi} & \Hk\ar[d]^{(\oleft{h},\oright{h})}\\
\X{}\ar[r]^(.4){u_{0}\times u_{0}\times\id} & \Bun\times\Bun\times(\X{})}
\end{equation}
Since the morphism $u_{0}\times u_{0}\times\id:\X{}\to\Bun\times\Bun\times(\X{})$ factors through the $\tilA\times\tilA$-torsor $\X{}\to(\tilK\backslash U)^{2}\times\X{}$, therefore $\wt{\GR^{U}}$ carries a natural $\tilA\times\tilA$-action. The first $\tilA$-action (resp. the second action) comes from the base change through $\oleft{h}$ (resp. $\oright{h}$).

The ind-scheme $\wt{\GR^{U}}$ fits into a diagram
\begin{equation}\label{GRGR}
\xymatrix{& \wt{\GR^{U}}\ar[dl]\ar[r]^{\nu} & \GR^{U}\ar@{^{(}->}[r]^{j_{\GR}}\ar[dl]^{\omega^{U}} & \GR\ar[r]\ar[dl]^{\omega}\ar[dr]^{\pi} & \Hk\ar[dr]^{\oright{h}}\\
\Spec k\ar[r]^{u_{0}} & \tilK\backslash U\ar@{^{(}->}[r] & \Bun & & \X{}\ar[r]^(.4){u_{0}\times\id} & \Bun\times(\X{})} 
\end{equation}
where all squares are Cartesian by definition. Using the moduli interpretation given in Section \ref{sss:BD}, $\omega:\GR\to\Hk\xrightarrow{\oleft{h}}\Bun$ sends the triple $(x,\calP,\iota)$ to $\calP\in\Bun$. Denote the projections to curves by
\begin{equation*}
\pi^U:\GR^U\to\X{};\hspace{1cm}\tpi:\wt{\GR^U}\to\X{}.
\end{equation*}
By proper base change and the definition of Hecke operators \eqref{defineT},
\begin{equation}\label{bch}
(u_{0}\times\id)^{*}\TT(\calK,j_{!}\calF_{\chi})=\pi_{!}(\omega^{*}j_{!}\calF_{\chi}\otimes\calK_{\GR,\Qlbar})=\pi^{U}_{!}(\omega^{U,*}\calF_{\chi}\otimes\calK_{\GR,\Qlbar}).
\end{equation}

On the other hand, it is easy to see that:
\begin{equation*}
\omega^{U,*}\calF_{\chi}\cong(V_{\chi}\otimes\nu_{*}\Qlbar)^{\tilA(1)}.
\end{equation*}
Here we put $\tilA(1)$ to emphasize the action of $\tilA$ on $\nu_{*}\Ql$ is induced from the first $\tilA$-action on $\wt{\GR^{U}}$. Taking $\tilA$-invariants makes sense, see the discussion preceding Lemma \ref{l:writeE}. Tensoring with $\calK_{\GR,\Qlbar}$ we get
\begin{equation*}
(u_{0}\times\id)^{*}\TT(\calK,j_{!}\calF_{\chi})\cong\pi^{U}_{!}(\omega^{U,*}\calF_{\chi}\otimes\calK_{\GR,\Qlbar})\cong\tpi_{!}(V_{\chi}\otimes\nu_{*}\nu^{*}\calK_{\GR,\Qlbar})^{\tilA(1)}.
\end{equation*}

Let $(\nu_{*}\Qlbar)_{\odd}\subset\nu_{*}\Qlbar$ be the subcomplex on which $\mu_{2}^{\ker}\subset\tilA$ acts by the sign representation (either in the first action or the second, they are the same on $\mu^{\ker}_{2}$). Therefore
\begin{equation*}
(\nu_{*}\Qlbar)_{\odd}=\bigoplus_{\chi\in\tilA^{*}_{0,\odd}}V^{*}_{\chi}\boxtimes(V_{\chi}\otimes\nu_{*}\Qlbar)^{\tilA(1)}\cong\bigoplus_{\chi\in\tilA^{*}_{0,\odd}}V^{*}_{\chi}\boxtimes\omega^{U,*}\calF_{\chi}
\end{equation*}
This isomorphism is $\tilA\times\tilA$-equivariant: $\tilA(1)$ acts on $V_\chi^*$ and $\tilA(2)$ acts on $\omega^{U,*}\calF_{\chi}$. This is why we use $\boxtimes$ rather than $\otimes$. Tensoring with $\calK_{\GR,\Qlbar}$ we get
\begin{equation*}
(\nu_{*}\nu^{*}\calK_{\GR,\Qlbar})_{\odd}\cong\bigoplus_{\chi\in\tilA^{*}_{0,\odd}}V^{*}_{\chi}\boxtimes(\omega^{U,*}\calF_{\chi}\otimes\calK_{\GR,\Qlbar}).
\end{equation*}
Taking $\pi^{U}_{!}$ we get an $\tilA\times\tilA$-equivariant isomorphism
\begin{equation}\label{dfdf}
(\tpi_{!}\nu^{*}\calK_{\GR,\Qlbar})_{\odd}
\cong\bigoplus_{\chi\in\tilA^{*}_{0,\odd}}V^{*}_{\chi}\boxtimes\pi^{U}_{!}(\omega^{U,*}\calF_{\chi}\otimes\calK_{\GR,\Qlbar})
\stackrel{\eqref{bch}}{\cong}\bigoplus_{\chi\in\tilA^{*}_{0,\odd}}V^{*}_{\chi}\boxtimes(u_{0}\times\id)^{*}\TT(\calK,j_{!}\calF_{\chi}).
\end{equation}
On the other hand, by the defining equation \eqref{eigeneq} of eigen local systems,  we get an $\tilA$-equivariant isomorphism between local systems on $\X{}$:
\begin{equation*}
(u_0\times\id)^*\TT(\calK,j_!\calF_{\chi})\cong V_{\chi}\otimes\calE_{\chi}(\calK)_{\Qlbar}.
\end{equation*}
Combining this with \eqref{dfdf}, we get a canonical $\tilA\times\tilA$-equivariant isomorphism
\begin{equation}\label{oddchi}
(\tpi_{!}\nu^{*}\calK_{\GR,\Qlbar})_{\odd}\cong\bigoplus_{\chi\in\tilA^{*}_{0,\odd}}(V^{*}_{\chi}\boxtimes V_{\chi})\otimes\calE_{\chi}(\calK)_{\Qlbar}
\end{equation}
On the right side, $\tilA(1)$ acts on $V^{*}_{\chi}$ and $\tilA(2)$ acts on $V_{\chi}$. This isomorphism also shows that $(\tpi_{!}\nu^{*}\calK_{\GR,\Qlbar})_{\odd}$ is a local system on $\X{}$.

\subsubsection{Rationality issue}\label{sss:rat} We use the following convention
\begin{equation*}
\Ql':=\begin{cases}\Ql & \textup{ if } G \textup{ is of type } D_{4n}, E_8 \textup{ or } G_2\\
      					\Ql(\sqrt{-1}) & \textup{ if } G \textup{ is of type } A_1, D_{4n+2} \textup{ or } E_7\end{cases}
\end{equation*}
Let $\tilA(\kbar)\times\tilA(\kbar)$ act on the group ring $\Ql'[\tilA(\kbar)]$ via left and the inverse of right translation on $\tilA(\kbar)$. Then $\Ql'[\tilA(\kbar)]$ is an $\Gamma^{(2)}:=\tilA(\kbar)\times\tilA(\kbar)\rtimes\Gk$-module. We may decompose it under the action of $\tilA_{0}$ via the embedding into the second copy of $\tilA(\kbar)$:
\begin{equation*}
\Ql'[\tilA(\kbar)]=\bigoplus_{\chi\in\tilA^*_{0,\odd}}\Ql'[\tilA(\kbar)]_{\chi}.
\end{equation*}
Note that $\Qlbar[\tilA(\kbar)]_\chi= V^*_\chi\boxtimes V_\chi$ as $\Gamma^{(2)}$-modules. From the construction of the  $\Ql$-structure on $\calE_\chi(\calK)_{\Qlbar}$ we will give in Section \ref{sss:Ql}, it is easy to see that \eqref{oddchi} descends to
\begin{equation}\label{KQl}
(\tpi_{!}\nu^{*}\calK_{\GR,\Ql'})_{\odd}\cong\bigoplus_{\chi\in\tilA^*_{0,\odd}}\Ql'[\tilA(\kbar)]_{\chi}\otimes\calE_{\chi}(\calK)_{\Ql'}.
\end{equation}
Equivalently,
\begin{equation}\label{Echi}
\calE_{\chi}(\calK)_{\Ql'}\cong\left(\Ql'[\tilA(\kbar)]_{\chi}\otimes(\tpi_{!}\nu^{*}\calK_{\GR,\Ql'})_{\odd}\right)^{\tilA\times\tilA}.
\end{equation}

\subsubsection{Intersection cohomology}
Let $\lambda$ be a dominant coweight of $G$, and $\calK=\IC_{\lambda}$. Then the spread-out $\calK_{\GR}$ is, up to twist and shift, the intersection cohomology complex of $\GR_{\leq\lambda}$. More precisely,
\begin{equation*}
\calK_{\GR}=\IC_{\GR_{\leq\lambda}}[\jiao{2\rho,\lambda}](\jiao{\rho,\lambda}).
\end{equation*}
Here we normalize the intersection complex $\IC_{X}$ of a variety $X$ to lie in cohomological degrees $0,\cdots, 2\dim X$ and is pure of weight $0$ as a complex. Since $\wt{\GR^{U}_{\leq\lambda}}\to\GR^{U}_{\leq\lambda}\hookrightarrow\GR_{\leq\lambda}$ is \'etale, $\nu^{*}\calK_{\GR}$ is also the intersection complex of $\wt{\GR^{U}_{\leq\lambda}}$ up to shift and twist. We use the notation
\begin{equation*}
\bIH_c(\wt{\GR^{U}_{\leq\lambda}}/\X{},\Ql'):=\tpi_{!}\IC_{\wt{\GR^{U}_{\leq\lambda}},\Ql'}\in D^{b}(\X{},\Ql')
\end{equation*}
whose stalk at $x\in\X{}$ calculates the intersection cohomology of $\wt{\Gr^{U}_{x,\leq\lambda}}$. Then \eqref{KQl} becomes
\begin{equation}\label{IH}
\bIH_{c}(\wt{\GR^{U}_{\leq\lambda}}/\X{},\Ql')_{\odd}[\jiao{2\rho,\lambda}](\jiao{\rho,\lambda})\cong\bigoplus_{\chi\in\tilA^*_{0,\odd}}\Ql'[\tilA(\kbar)]_{\chi}\otimes\calE_{\chi}(\calK)_{\Ql'}.
\end{equation}
In particular, $\bIH_{c}(\wt{\GR^{U}_{\leq\lambda}}/\X{},\Ql')_{\odd}$ is concentrated in the middle degree $\jiao{2\rho,\lambda}$, and is a local system.

\subsubsection{Quasi-minuscule Schubert variety}\label{sss:qmin} We specialize to the case $\lambda=\alpha^{\vee}$, the coroot corresponding to the maximal root $\alpha$ of $G$. This is called a quasi-minuscule weight of $\hatG$ because in the weight decomposition of the irreducible representation $V^{\qm}:=V_{\alpha^{\vee}}$ of $\hatG$, all nonzero weights  are in the Weyl group orbit of $\alpha^{\vee}$. When $\hatG$ is simply-laced, $V^{\qm}\cong\wh{\frg}$ is the adjoint representation of $\hatG$.

For basic properties of the variety $\Gr_{x,\leq\alpha^\vee}$ we refer to \cite[Section 5.3, especially Lemma 5.22]{HNY}. In particular, it has dimension $\jiao{2\rho,\alpha^{\vee}}=2h^{\vee}-2$, where $h^{\vee}$ is the dual Coxeter number of $G$. It consists of two $L^{+}_{x}G$-orbits: the base point $\{\star\}=\Gr_{x,0}$ (the only singularity) and its complement $\Gr_{x,\alpha^\vee}$, which is in fact isomorphic to the variety $Y$ introduced in Section \ref{ss:explicit}. The open subset $\Gr^U_{x,\leq\alpha^\vee}$ is of the form $Y\backslash D_x$, where $D_x$ is a divisor of $Y$ depending algebraically on $x$.

Let $\tilY=\wt{\GR^{U}_{\alpha^\vee}}$ and $Z$ be its complement in $\wt{\GR^{U}_{\leq\alpha^\vee}}$. We temporarily denote the corresponding open and closed embeddings by $j$ and $i$. Then $Z=\tilA\times(\X{})=\wt{\GR^{U}_{0}}$. Let $\IC^{\qm}_{\wt{\GR}}$ be the intersection complex of $\wt{\GR^{U}_{\leq\alpha^{\vee}}}$. We have a distinguished triangle
\begin{equation*}
j_!\Ql[2h^\vee-2](h^\vee-1)\to\IC^{\qm}_{\wt{\GR}}[2h^{\vee}-2](h^{\vee}-1)\to i_*(C\otimes\QQ_{\ell,Z})\to
\end{equation*}
where $C$ is the stalk of $\IC^{\qm}_{\wt{\GR}}[2h^{\vee}-2](h^{\vee}-1)$ along $Z$, which is the same as the stalk of $\IC_{\alpha^\vee}$ at $\{\star\}$. By the definition of the intersection cohomology complex, $C$ lies in degrees $<0$, therefore, taking compactly supported cohomology of the above distinguished triangle we get an isomorphism 
\begin{equation*}
\bH^{2h^\vee-2}_{c}(\tilY/\X{})\isom\bIH_c^{2h^\vee-2}(\wt{\GR^U_{\leq\alpha^\vee}}/\X{}).
\end{equation*}
Therefore \eqref{IH} implies an isomorphism in $D^b(\X{},\Ql')$
\begin{equation}\label{HQmin}
\bH^{2h^{\vee}-2}_{c}(\tilY/\X{},\Ql')_{\odd}(h^{\vee}-1)\cong\bigoplus_{\chi\in\tilA^*_{0,\odd}}\Ql'[\tilA(\kbar)]_{\chi}\otimes\calE_{\chi}(\IC_{\alpha^{\vee}})_{\Ql'}.
\end{equation}
We will abbreviate $\calE_{\chi}(\IC_{\alpha^{\vee}})$ by $\calE^{\qm}_{\chi}$. Then
\begin{equation}\label{Eqmin}
\calE^{\qm}_{\chi,\Ql'}\cong\left(\Ql'[\tilA(\kbar)]_{\chi}\otimes\bH_{c}^{2h^{\vee}-2}(\tilY/\X{},\Ql')_{\odd}\right)^{\tilA\times\tilA}.
\end{equation}

When we need to emphasize the base field we are working with, we write $\calE_{\chi,k}$ as the eigen local system over $\X{k}$. Next we extend the local system $\calE_{\chi,\QQ}$ from $\X{\QQ}$ to $\X{\Z}$. 
 
\begin{prop}\label{p:int}
There is a positive integer $N$ such that $\calE_\chi$ extends to a $\hatG(\Ql)$-local system $\un{\calE}_{\chi}$ over $\X{\Z}$. Moreover, for any geometric point $\Spec k\to\Spec\Z$, the restriction of $\un{\calE}_{\chi}$ to $\X{k}$ is isomorphic to $\calE_{\chi,k}$.
\end{prop}
\begin{proof}
We define
\begin{equation*}
\QQ':=\begin{cases}\QQ & \textup{ if } G \textup{ is of type } D_{4n}, E_8 \textup{ or } G_2\\
\QQ(i) & \textup{ if } G \textup{ is of type } A_1, D_{4n+2} \textup{ or } E_7
\end{cases}
\end{equation*}
Let $\ZZ'$ be the ring of integers in $\QQ'$. We will first extend $\calE_{\chi,\QQ',\Ql'}$ (over $\X{\QQ'}$ with $\Ql'$-coefficients) to $\X{\ZZ'[1/2\ell N]}$ for some $N$.

In what follows, we use underlined symbols to denote the integral versions of the spaces, sheaves, etc. Over $\ZZ'[1/2]$, the groups $G, K,\cdots$ and the spaces $\Bun, U,\cdots$ have integral models $\un{G},\un{K}, \cdots, \un{\Bun}, \un{U}, \cdots$. All these spaces are defined in a natural way such that their $k$-fibers ($k$ is any field with $\textup{char}(k)\neq2$) are the same as the corresponding spaces over the base field $k$ we defined before. 

As discussed in Section \ref{sss:basepoint}, the rational point $u_{0}$ extends to a point $\unu_{0}:\Spec\ZZ'[1/2N_{0}]\to\unU$. The stabilizer of $\unu_{0}$ under $\un{\tilK}$ is a finite flat group scheme $\un{\tilA}$ over $\ZZ'[1/2N_{0}]$. We may define the integral version of $\tilY$ (in Section \ref{sss:qmin}) by forming the Cartesian diagram
\begin{equation*}
\xymatrix{\un{\tilY}\ar[d]\ar[r] & \un{\GR}_{\alpha^{\vee}}\ar[d]\\
\Spec\ZZ'[1/2N_{0}] \ar[r]^(.6){\unu_{0}} & \un{\Bun}}
\end{equation*}
Then $\un{\tilY}$ maps to $\X{\ZZ'[1/2\ell N_{0}]}$. Let
\begin{equation*}
\un{\calE'}^{\qm}_{\chi}=\left(\Ql'[\tilA(\Qbar)]_{\chi}\otimes\bH^{2h^{\vee}-2}_{c}(\un{\tilY}/\X{\ZZ'[1/2\ell N_{0}]},\Ql')_{\odd}\right)^{\un{\tilA}\times\un{\tilA}}
\end{equation*} 
We need to make sense of the operation of taking invariant of the group scheme $\un{\tilA}$. Note that $\un{\tilA}$ is finite \'etale over $\ZZ'[1/2N_0]$, hence $\Gal(\Qbar/\QQ')$ acts on $\tilA(\Qbar)$ via its quotient $\pi_1(\ZZ'[1/2N_0])$, therefore the construction of invariants by descent in Section \ref{sss:inv} still works over the base $\ZZ'[1/2N_0]$. By proper base change and \eqref{Eqmin}, the restriction of $\un{\calE'}^{\qm}_{\chi}$ to $\X{k}$ is the eigen local system $\calE^{\qm}_{\chi,k}$ for any field-valued point $\Spec k\to \Spec \ZZ'[1/2\ell N_0]$. 

By enlarging $N_{0}$ to some positive integer $N$, we may assume $\un{\calE'}^{\qm}_{\chi}$ is a local system of rank $d=\dim V_{\alpha^{\vee}}$ over $\X{\ZZ'[1/2\ell N]}$ (in fact we do not need to enlarge $N_{0}$ in this step but I do not need this fact). The monodromy representation of $\un{\calE'}^{\qm}_{\chi}$ is 
\begin{equation*}
\rho':\pi_{1}(\X{\ZZ'[1/2\ell N]},*)\to\GL_{d}(\Ql').
\end{equation*}
Since the $\QQ'$-fiber of $\un{\calE'}^{\qm}_{\chi}$ is $\calE^{\qm}_{\chi,\QQ',\Ql'}$, we have a commutative diagram
\begin{equation}\label{4pi}
\xymatrix{\pi_{1}(\X{\QQ'})\ar[d]^{s'}\ar@{^{(}->}[r]& \pi_{1}(\X{\QQ})\ar[d]^{s}\ar[r]^{\rho_{\QQ}} & \hatG(\Ql)\ar@{^{(}->}[d]^{a}\\
\pi_{1}(\X{\ZZ'[1/2N]})\ar@/_2pc/[rr]^{\rho'}\ar@{^{(}->}[r] & \pi_{1}(\X{\Z})\ar@{-->}[r]\ar@{-->}[ur]^{\rho} & \GL_{d}(\Ql')}
\end{equation}
where $\rho_{\QQ}$ is the monodromy representation of $\calE_{\chi,\QQ}$ and the vertical map $a$ is the homomorphism giving the quasi-minuscule representation of $\hatG$. Since the vertical arrows $s'$ and $s$ are surjections with the same kernel, there is a unique way to fill in the dotted arrows. This dotted arrow $\rho$ gives the desired $\hatG(\Ql)$-local system $\un{\calE}_{\chi}$ over $\X{\Z}$.
\end{proof}

\subsection{Description of the motives}\label{ss:description}
\subsubsection{Motives of an open variety}\label{sss:motive} 
In this section, we assume the base field $k$ is a number field, and {\em we assume the Standard Conjectures}. For every smooth projective variety $X$ over $k$, and every $i\in\ZZ$, there is a well-defined motive $h^i(X)\in\Mot_{k}$ such that under the $\ell$-adic cohomology functor, $h^i(X)$ is sent to $\upH^i(X_{\kbar},\Ql)$.

More generally, suppose $X$ is a smooth quasi-projective variety over $k$, one can also define an object $h^i_{c,\pur}(X)\in\Mot_{k}$ as follows. Let $\Xbar$ be a compactification of $X$ over $\QQ$, which is a smooth projective variety such that $D=\Xbar\backslash X$ is a union of smooth divisors $\cup_{s\in S} D_{s}$ with normal crossings. Such an $\Xbar$ exists by Hironaka's resolution of singularities. Let
\begin{equation*}
h^i_{c,\pur}(X)=\ker(h^i(\Xbar)\to \oplus_{s\in S}h^i(D_{s})).
\end{equation*}
This may depend on the choice of the compactification. However we will see below that the $\ell$-adic realization of $h^i_{c,\pur}(X)$ only depends on $X$. Consider the following maps induces by inclusion and restriction
\begin{equation*}
\cohoc{i}{X_{\kbar},\Ql}\to\cohog{i}{\Xbar_{\kbar},\Ql}\to\oplus_{s\in S}\cohog{i}{D_{s,\kbar},\Ql}. 
\end{equation*}
Each cohomology group carries a weight filtration (we can choose an integral model and look at the weights given by $\Frob_v$ for almost all $v$) and the maps are strictly compatible with the weight filtrations. Taking the weight $i$ pieces, we get
\begin{equation*}
\Gr^W_i\cohoc{i}{X_{\kbar},\Ql}=\ker(\cohog{i}{\Xbar_{\kbar},\Ql}\to\oplus_{s\in S}\cohog{i}{D_{\kbar},\Ql)}=\upH_\ell(h^i_{c,\pur}(X)).
\end{equation*}
In other words, the $\ell$-adic realization of the motive $h^i_{c,\pur}(X)$ is the $\Gk$-module $\Gr^W_i\cohoc{i}{X_{\kbar},\Ql}$.

Suppose the smooth quasi-projective variety $X$ is equipped with an action of a finite group scheme $A$ over $k$, then we may first find an $A$-equivariant projective embedding. In fact, let $\act:A\times X\to X$ and $\pr:A\times X\to X$ be the action and projection map. If $\calL$ is an ample line bundle on $X$, then its average under the $A$-action $\calL_{A}=\det(\act_{*}\pr^{*}\calL)$ is again ample and $A$-equivariant. Using a high power of $\calL_{A}$ we have an $A$-equivariant projective embedding $X\hookrightarrow\PP^{N}$, and the closure of its image gives an $A$-equivariant compactification $\Xbar$ of $X$. Using the equivariant version of Hironaka's resolution of singularities, we may assume that $\Xbar-X=\cup_{s\in S} D_{s}$ is again a union of smooth divisors with normal crossings, and the whole situation is $A$-equivariant. 

Every {\em closed point} $a\in|A|$ gives a self-correspondence $\Gamma(a)$ of $(X,\Xbar, D_s)$ which is the graph of the $a$-action. All such correspondences span an algebra isomorphic to $\QQ[\tilA(\kbar)]^{\Gk}$, and its action on $h^i_{c,\pur}(X)$ gives a homomorphism
\begin{equation*}
\QQ[\tilA(\kbar)]^{\Gk}\to\End_{\Mot_k(\QQ)}(h^i_{c,\pur}(X)).
\end{equation*}
If $L$ is a number field and $e\in L[\tilA(\kbar)]^{\Gk}$ is an idempotent, then $eh^i_{c,\pur}(X)$ is an object in $\Mot_k(L)$.

\subsubsection{The idempotents}\label{sss:idem} In this section we assume the number field $k$ satisfies the assumption \eqref{ass:i}. Fix $\chi\in\tilA^{*}_{0,\odd}$. Let $\varphi_{\chi}:\tilA(\kbar)\to\Qbar$ be the characters of irreducible representations $V_{\chi}$. For $G=D_{4n}, E_{8}$ or $G_{2}$, $\varphi_{\chi}$ takes $\QQ$-values because  $\GQ$ acts on the set of irreducible characters; for $G=A_{1}, D_{4n+2}$ or $E_{7}$, $\varphi_{\chi}$ takes values in $\QQ(i)$ for the same reason. We let $\QQ'=\QQ$ in the former case and let $\QQ'=\QQ(i)$ in the later. Similarly we define $\Ql'$ as in Section \ref{sss:rat}. We will consider the category $\Mot_{k}(\QQ')$ of motives over $k$ with coefficients in $\QQ'$.

Under the $\Gk$-action on $\tilA(\kbar)$, $\varphi_{\chi}$ is also invariant because this action fixes the central character $\chi$. We can make the following idempotent element in $\QQ'[\tilA(\kbar)\times\tilA(\kbar)]$
\begin{equation*}
e_{\chi}:=\frac{1}{2^{2r+2}}\sum_{(a_{1},a_{2})\in\tilA(\kbar)}\varphi_{\chi}(a_{1}a_{2}^{-1})\cdot a_{1}\otimes a_{2}\in\QQ'[\tilA(\kbar)\times\tilA(\kbar)].
\end{equation*}
Since $\varphi_{\chi}$ is constant on $\Gk$-orbits of $\tilA(\kbar)$, we actually have
\begin{equation*}
e_{\chi}\in\QQ'[\tilA(\kbar)\times\tilA(\kbar)]^{\Gk}.
\end{equation*}
The action of $\epsilon_{\chi}$ on the $\tilA(\kbar)\times\tilA(\kbar)$-module $\Qlbar[\tilA(\kbar)]_{\chi}$ is the projector onto $\id\in\End(V_{\chi})=\Qlbar[\tilA(\kbar)]_{\chi}$. The action of $\epsilon_{\chi}$ on $\Qlbar[\tilA(\kbar)]_{\chi'}$ is zero if $\chi'\neq\chi$. Let
\begin{equation*}
M_{\chi,x}:=e_{\chi}h^{2h^{\vee}-2}_{c,\pur}(\tilY_{x})(h^{\vee}-1)\in\Mot_{k}(\QQ').
\end{equation*}

\begin{prop}\label{p:motivic} There is an isomorphism of $\Gk$-modules
\begin{equation*}
\upH(M_{\chi,x},\Ql')\cong\calE^{\qm}_{\chi,x,\Ql'}.
\end{equation*}
In other words, the $\Gk$-module $\calE^{\qm}_{\chi, x,\Ql'}$ is isomorphic to the $\Ql'$-realization of the motive $M_{\chi,x}\in\Mot_{k}(\QQ')$.
\end{prop}
\begin{proof}
Taking stalk of \eqref{HQmin} at $x\in\PP^{1}(k)-\{0,1,\infty\}$, we get an isomorphism of $\Gk$-modules:
\begin{equation}\label{Hqmin}
\cohoc{2h^{\vee}-2}{\tilY_{x},\Ql'}_{\odd}(h^{\vee}-1)\cong\bigoplus_{\chi\in\tilA^{*}_{0,\odd}}\Ql'[\tilA(\kbar)]_{\chi}\otimes_{\Ql}\calE^{\qm}_{\chi,x}.
\end{equation}
By Proposition \ref{p:int}, $\calE^{\qm}_{\chi,x}$ is unramified at large enough prime $p$, and is pure of weight zero under $\Frob_p$ according the purity result proved in Theorem \ref{th:eigen}(4). Using \eqref{Hqmin}, we see that the same purity property holds for $\cohoc{2h^{\vee}-2}{\tilY_{x},\Ql'}_{\odd}(h^\vee-1)$. Moreover, $e_{\chi}$ projects $\cohoc{2h^{\vee}-2}{\tilY_{x},\Ql'}$ to a subspace of the odd part, therefore
\begin{equation*}
\upH(M_{\chi,x},\Ql')\cong e_{\chi}\Gr^{W}_{2h^{\vee}-2}\cohoc{2h^{\vee}-2}{\tilY_{x},\Ql'}(h^{\vee}-1)=e_{\chi}\cohoc{2h^{\vee}-2}{\tilY_{x},\Ql'}_{\odd}(h^{\vee}-1)
\end{equation*}
Making $e_\chi$ act on the right side of \eqref{Hqmin}, we get
\begin{equation*}
\upH(M_{\chi,x},\Ql')\cong e_{\chi}\Ql'[\tilA(\kbar)]_{\chi}\otimes_{\Ql}\calE^{\qm}_{\chi,x}=\Ql'\otimes_{\Ql}\calE^{\qm}_{\chi,x}.
\end{equation*}
\end{proof}

\subsection{Proof of Theorem \ref{th:eigen}}\label{ss:pfeigen}
\begin{prop}\label{p:perv} Let $\calK\in\Sat$ and $\calF\in\Loc_{\tilK}(U)_{\odd}$
\begin{enumerate}
\item $(u_{0}\times\id)^{*}\TT(\calK,j_{!}\calF)[1]\in D^{b}(\X{})$ is a perverse sheaf.
\item If $k$ is a finite field, $(u_{0}\times\id)^{*}\TT(\calK,j_{!}\calF)$ is pure of weight zero.
\end{enumerate}
\end{prop}
\begin{proof}
The proof is a variant of \cite[Section 4.1]{HNY}. Let $\lambda$ be a large enough coweight of $G$ such that $\calK$ is supported in $\Gr_{\leq\lambda}$. We restrict the diagram \eqref{GRGR} to $\GR_{\leq\lambda}$ without changing notations of the morphisms. By \eqref{bch} we have
\begin{equation*}
(u_{0}\times\id)^{*}\TT(\calK,j_{!}\calF)\cong\pi^{U}_{!}(\omega^{U,*}\calF\otimes\calK_{\GR})
\end{equation*}
The complex $\calK_{\GR}$ is in perverse degree 1 on $\GR$ (because $\GR$ has one more dimension than the affine Grassmannian) and it is pure of weight 0. Therefore $\omega^{U,*}\calF\otimes\calK_{\GR}$ is in perverse degree 1 and pure of weight 0. In Lemma \ref{l:affine} below we will show that $\pi^{U}:\GR^{U}_{\leq\lambda}\to\X{}$ is affine, therefore $(u_{0}\times\id)^{*}\TT(\calK,j_{!}\calF)\in \pD^{\geq1}(\Bun\times(\X{}))$ by \cite[Th\'er\`eme 4.1.1]{BBD}. By \cite[Variant 6.2.3 of Th\'eor\`eme 3.3.1]{WeilII}, it has weight $\leq0$.

On the other hand, consider the Cartesian diagram
\begin{equation*}
\xymatrix{\GR^U\ar@/^2pc/[rr]_{\omega^U}\ar[r]^{v^U}\ar@{^{(}->}[d]^{j_{\GR}} & \Hk^U \ar@{^{(}->}[d]^{j_{\Hk}}\ar[r]^{\oleft{h}^U} & \tilK\backslash U\ar@{^{(}->}[d]^{j}\\
\GR\ar@/_2pc/[rr]^{\omega}\ar[r]^{v} & \Hk\ar[r]^{\oleft{h}} & \Bun}
\end{equation*}
Since $\oleft{h}$ is a locally trivial fibration in smooth topology by \cite[Remark 4.1]{HNY}, the natural transformation $\oleft{h}^*j_*\to j_{\Hk,*}\oleft{h}^{U,*}$ is an isomorphism. Therefore
\begin{equation*}
\omega^*j_*=v^*\oleft{h}^*j_*\isom v^*j_{\Hk,*}\oleft{h}^{U,*}\isom j_{\GR,*}v^{U,*}\oleft{h}^{U,*}= j_{\GR,*}\omega^{U,*}.
\end{equation*}
where the second isomorphism uses the fact that $v$ is \'etale. Hence
\begin{eqnarray}\label{T*}
(u_{0}\times\id)^{*}\TT(\calK,j_{*}\calF)&\cong&\pi_{!}(\omega^{*}j_{*}\calF\otimes\calK_{\GR})\cong\pi_{*}\left((j_{\GR,*}\omega^{U,*}\calF)\otimes\calK_{\GR}\right)\\
\notag&\cong&\pi_{*}j_{\GR,*}(\omega^{U,*}\calF\otimes\calK_{\GR})=\pi^{U}_{*}(\omega^{U,*}\calF\otimes\calK_{\GR}).
\end{eqnarray}
Here $\pi_{!}=\pi_{*}$ because $\pi:\GR_{\leq\lambda}\to\X{}$ is proper. Since $\pi^{U}$ is affine and $\omega^{U,*}\calF\otimes\calK_{\GR}$ is in perverse degree 1, we have $(u_{0}\times\id)^{*}\TT(\calK,j_{*}\calF)\in \pD^{\leq1}(\X{})$ by \cite[Corollaire 4.1.2]{BBD}. Since $\omega^{U,*}\calF\otimes\calK_{\GR}$ is pure of weight 0, by the dual statement of \cite[6.2.3]{WeilII}, $(u_{0}\times\id)^{*}\TT(\calK,j_{*}\calF)$ has weight $\geq0$.

Since $j_{!}\calF=j_{*}\calF$ by Theorem \ref{th:clean}, the above argument shows that $(u_{0}\times\id)^{*}\TT(\calK,j_{*}\calF)$ is concentrated in perverse degree 1 and is pure of weight 0.
\end{proof}

\begin{lemma}\label{l:affine}
The morphism  $\pi^{U}:\GR^{U}_{\leq\lambda}\to\X{}$ is affine.
\end{lemma}
\begin{proof}
Fix a point $x\in(\X{})(R)$ for some finitely generated $k$-algebra $R$. We will argue that the fiber $\Gr^{U}_{x,\leq\lambda}$ of $\GR^{U}_{\leq\lambda}$ over $x$ is an affine scheme. 

The embedding $j$ factors as $\tilK\backslash U\hookrightarrow\tilK\backslash\fl_{G}\xrightarrow{j_{1}}\Bun$, therefore we have $\Gr^{U}_{x,\leq\lambda}\hookrightarrow\GR^{\star}_{x,\leq\lambda}\hookrightarrow\Gr_{x,\leq\lambda}$, where $\Gr^{\star}_{x,\leq\lambda}=\omega^{-1}(\tilK\backslash\fl_{G})$. Since $U$ is itself affine, so are the open embedding $\tilK\backslash U\hookrightarrow\tilK\backslash\fl_{G}$ and its base change $\Gr^{U}_{x,\leq\lambda}\hookrightarrow\Gr^{\star}_{x,\leq\lambda}$. Therefore it suffices to show that $\Gr^{\star}_{x,\leq\lambda}$ is affine. 

Consider the morphism
\begin{equation*}
\beta:\Gr_{x,\leq\lambda}\xrightarrow{\omega}\Bun\to\Bun_{G}(\bP_{0},\bP_{\infty})\xrightarrow{\Ad^{+}}\Bun_{d}.
\end{equation*}
Here $\Ad^{+}$ is defined in the proof of Lemma \ref{l:j0affine}. By definition, $\Gr^{\star}_{x,\leq\lambda}$ is the preimage of $\tilK\backslash\fl_{G}\subset\Bun$ under $\oleft{h}$, hence the preimage of $\{\star\}/\tilK\subset\Bun_{G}(\bP_{0},\bP_{\infty})$, which is in turn the preimage of the open substack $\{\calO(-1)^{d}\}/\GL_{d}\subset\Bun_{d}$ under $\beta$, by the Cartesian diagram \eqref{CartBun}. The open substack $\{\calO(-1)^{d}\}/\GL_{d}$ is the non-vanishing locus of a section of the inverse determinant line bundle $\calL_{\det}^{-1}$ on $\Bun_{d}$. Therefore, $\Gr^{\star}_{x,\leq\lambda}$ is the non-vanishing locus of a section of $\omega^{*}\calL_{\det}^{-1}$. In order to show that $\Gr^{\star}_{x,\leq\lambda}$ is affine, it suffices to show that $\omega^{*}\calL_{\det}^{-1}$ is ample on $\Gr_{x,\leq\lambda}$ relative to the base $\Spec R$. 

Every point $(\calP,\iota)\in\Gr_{x,\leq\lambda}(S)$ ($S$ is an $R$-algebra) consists of $\calP\in\Bun_{G}(\bP_{0},\bP_{\infty})(S)$ together with an isomorphism $\iota:\calP|_{\PP^{1}_{S}-\Gamma(x_{S})}\isom\calP_{0}|_{\PP^{1}_{S}-\Gamma(x_{S})}$, where $\Gamma(x_{S})\subset\PP^{1}_{S}$ is the graph of $x$ viewed as an $S$-point of $\PP^{1}$, and $\calP_{0}\in\Bun_{G}(\bP_{0},\bP_{\infty})(k)$ is the fixed trivial $G$-torsor. The line bundle $\omega^{*}\calL_{\det}^{-1}$ assigns to $(\calP,\iota)$ the line $\det_{S}\bR\Gamma(\PP^{1}_{S},\Ad^{+}(\calP))^{-1}$.
 
The $R$-point $x$ corresponds to an algebra homomorphism $k[t,t^{-1}, (t-1)^{-1}]\to R$. Let $t_R$ be the image of $R$ and let $t_{x}=t-t_{R}\in R[t]$. Then $\calO_{x,S}=S[[t_{x}]]$ is the completion of $\PP^1_S$ along $\Gamma(x_{S})$. Let $F_{x,S}=S((t_{x}))$. The restriction of $\iota$ to $\Spec F_{x,S}$ gives a trivialization of $\calP$ there, and the adjoint vector bundle $\Ad^{+}(\calP)|_{\Spec \calO_{x,S}}$ determines an $\calO_{x,S}$-submodule $\Lambda\subset\Lambda_{0,S}\otimes_{\calO_{x,S}}F_{x,S}$, where $\Lambda_{0,S}=\Ad^{+}(\calP_{0})|_{\Spec\calO_{x,S}}$. The following discussion is taken from \cite[Section 2]{Floop}. Since we consider only $\Gr_{x,\leq\lambda}$ rather than the whole $\Gr_{x}$, there is a universal constant $N$ such that
\begin{equation*}
t^{N}_x\Lambda_{0,S}\subset\Lambda\subset t_x^{-N}\Lambda_{0,S}
\end{equation*}
for all $\Lambda$ coming from $(\calP,\iota)\in\Gr_{x,\leq\lambda}$, and $\Lambda/t^{N}_x\Lambda_{0}$ is a locally free $S$-module of rank equal to $M:=\rk_{S}(\Lambda_{0,S}/t^{N}_x\Lambda_{0,S})=N\dim G$. Moreover, the datum of $\Lambda$ uniquely determines $(\calP,\iota)$, so that we get a closed embedding
\begin{equation*}
\Gr_{x,\leq\lambda}\hookrightarrow\Gr(M,2M)\otimes_{k}R.
\end{equation*} 
Here $\Gr(M,2M)$ is the usual Grassmannian of $M$-dimensional linear subspaces in $\AA^{2M}$. The assignment $\calL_{\Gr}:\Lambda\mapsto\det_{S}(t^{-N}_x\Lambda_{0,S}/\Lambda)$ is a very ample line bundle on $\Gr(M, 2M)$ since it gives the Pl\"ucker embedding. Therefore $\calL_{\Gr}$ is also ample on $\Gr_{x,\leq\lambda}$ (relative to the base $R$). 

Since $\Ad^{+}(\calP)$ and $\Ad^{+}(\calP_{0})$ only differ at $x$,
\begin{equation*}
\det_{S}\bR\Gamma(\PP^{1}_{S},\Ad^{+}(\calP))^{-1}\otimes_{S}\det_{S}\bR\Gamma(\PP^{1}_{S},\Ad^{+}(\calP_{0}))\cong\det_{S}(t^{-N}_x\Lambda_{0,S}/\Lambda)\otimes_{S}\det_{S}(t^{-N}_x\Lambda_{0,S}/\Lambda_{0,S})^{-1}.
\end{equation*}
Therefore $\omega^{*}\calL_{\det}^{-1}\cong\calL_{\Gr}$ since the other lines appearing above are constant lines only depending on $\calP_{0}$. Hence $\omega^{*}\calL_{\det}^{-1}$ is also ample on $\Gr_{x,\leq\lambda}$ relative to $R$. This proves the lemma.
\end{proof}

\subsubsection{Preservation of the character $\chi$}\label{sss:chi}
We first show a general result which works for all $G$ satisfying \eqref{ass:-1}. We assume $k$ is algebraically closed and we work with $\Qlbar$-coefficients. Let $\tZG$ be the preimage of the center of $G$ (which is contained in $K$) in $\tilK$. Every object $\calP\in\Bun_G(\bP_0,\bI_1,\bP_\infty)$ has automorphisms by $ZG$; similarly, every object $\calP\in\Bun$ has automorphisms by $\tZG$. In other words, the classifying stack $\BB(\tZG)$ acts on the stack $\Bun$. In fact, the whole Hecke correspondence diagram \eqref{Hkcorr} is $\BB(\tZG)$-equivariant. For every character $\psi:\tZG\to\Qlbar^\times$, we have the subcategory $D^b(\Bun)_\psi\subset D^b(\Bun)$ consisting of complexes on which $\tZG$ acts through $\psi$. By the $\BB(\tZG)$-equivariance of the diagram (the action is trivial on $\calK\in\Sat$), the geometric Hecke operators $\TT(\calK,-)$ necessarily send $D^b(\Bun)_\psi$ to $D^b(\Bun\times(\X{}))_\psi$.

Now back to the situation of Theorem \ref{th:eigen}, where we have $\tZG=\tilA_0$. For any $\chi\in\tilA_0(\kbar)^*_{\odd}$ and any $\calK\in\Sat$, the above discussion shows that $\TT(\calK,j_!\calF_\chi)\in D^b(\Bun\times(\X{}))_\chi$.

\subsubsection{Hecke eigen property with assumption \eqref{ass:i}}\label{sss:withi} By Proposition \ref{p:perv}(1), $\TT(\calK,j_{!}\calF_{\chi})$ is concentrated in perverse degree 1 so Lemma \ref{l:writeE} is applicable. Therefore we can write canonically
\begin{equation*}
\TT(\calK,j_{!}\calF_{\chi})=\TT(\calK,j_{!}\calF_{\chi})_{\chi}\cong j_{!}\calF_{\chi}\boxtimes\left(V^{*}_{\chi}\otimes(u_{0}\times\id)^{*}\TT(\calK,j_{!}\calF_{\chi})\right)^{\tilA}
\end{equation*}
Here the first equality follows from the discussion in Section \ref{sss:chi}. We define
\begin{equation}\label{defineE}
\calE_{\chi}(\calK)_{{\Qlbar}}:=\left(V^{*}_{\chi}\otimes(u_{0}\times\id)^{*}\TT(\calK,j_{!}\calF_{\chi})\right)^{\tilA}
\end{equation}
which is a $\Qlbar$-complex on $\X{}$ concentrated in perverse degree 1.

Recall that the construction of $\calF_{\chi}$ (or the $\Gamma$-module $V_{\chi}$) is not completely canonical: we are free to twist it by a continuous character $\psi$ of $\Gk$. However, we have a canonical isomorphism
\begin{equation*}
\left(V_{\chi}(\psi)^{*}\otimes(u_{0}\times\id)^{*}\TT(\calK,j_{!}\calF_{\chi}(\psi))\right)^{\tilA}\cong\left(V^{*}_{\chi}\otimes(u_{0}\times\id)^{*}\TT(\calK,j_{!}\calF_{\chi})\right)^{\tilA}.
\end{equation*}
Therefore $\calE_{\chi}(\calK)_{\Qlbar}$ is completely canonical.
 
Using the monoidal structure of $\TT(-,j_{!}\calF_{\chi})$ as spelled out in Section \ref{ss:geomHk}, we get canonical isomorphisms
\begin{equation}\label{Etensor}
\varphi_{\calK_{1},\calK_{2}}:\calE_{\chi}(\calK_{1})_{\Qlbar}\otimes\calE_{\chi}(\calK_{2})_{\Qlbar}\isom\calE_{\chi}(\calK_{1}*\calK_{2})_{\Qlbar}
\end{equation}
for any two $\calK_{1},\calK_{2}\in\Sat$. These isomorphisms are compatible with the associativity and commutativity constraints (for the commutativity constraint, we use Mirkovi\'c and Vilonen's construction of the fusion product for the Satake category). Once we have the tensor property \eqref{Etensor}, we can use the argument of \cite[Section 4.2]{HNY} to show that each $\calE_{\chi}(\calK)_{\Qlbar}$ must be a local system on $\X{}$. Therefore $\calE_{\chi}(-)_{\Qlbar}$ gives a tensor functor
\begin{equation*}
\calE_{\chi,\Qlbar}:\Sat\to\Loc(\X{},\Qlbar)
\end{equation*}
which serves as the eigen local system of the Hecke eigensheaf $j_{!}\calF_{\chi}$. This proves statement (1) of Theorem \ref{th:eigen}(1). The purity statement (4) of Theorem \ref{th:eigen} follows from Proposition \ref{p:perv}(2).

\subsubsection{Descent of base field}\label{sss:descent} We now prove statement (2) in Theorem \ref{th:eigen}. We only need to consider the case $\sqrt{-1}\notin k$ and $G$ is of type $A_{1}, D_{4n+2}$ or $E_{7}$. Let $k'=k(\sqrt{-1})$. In the previous section we constructed $\calE_{\chi,k',\Qlbar}$ over $\X{k'}$. For a stack $X$ over $k$, we denote its base change to $k'$ by $X_{k'}$ and let $\sigma:X_{k'}\to X_{k'}$ be the involution induced from the nontrivial involution $\sigma\in\Gal(k'/k)$. 

Let $\Gamma'=\tilA(\kbar)\rtimes\Gal(\kbar/k')$, which is a subgroup of $\Gamma$ of index 2. The construction in Section \ref{sss:RepGamma} only gives a $\Gamma'$-module $V_{\chi}$ since we assumed \eqref{ass:i} there. The action of the involution $\sigma\in\Gal(k'/k)$ changes the central character $\chi$ to $\chibar$. Therefore we have an isomorphism of $\Gamma'$-modules $\alpha:\sigma^{*}V_{\chi}\isom V_{\chibar}(\psi)$ for some character $\psi$ of $\Gal(\kbar/k')$, unique up to a scalar. Similarly, $\alpha$ induces an isomorphism $\sigma^{*}\calF_{\chi}\cong\calF_{\chibar}(\psi)$, and hence an isomorphism
\begin{eqnarray*}
\sigma^{*}\left(V^{*}_{\chi}\otimes(u_0\times\id)^*\TT(\calK, j_!\calF_{\chi})\right)^{\tilA}&\isom&\left(V_{\chibar}(\psi)^{*}\otimes(u_0\times\id)^*\TT(\calK, j_!\calF_{\chibar}(\psi))\right)^{\tilA}\\
&=&\left(V^{*}_{\chibar}\otimes(u_0\times\id)^*\TT(\calK, j_!\calF_{\chibar})\right)^{\tilA}
\end{eqnarray*}
which is completely canonical because we used $\alpha$ twice. Using \eqref{defineE}, we get a canonical isomorphism
\begin{equation*}
\sigma^{*}\calE_{\chi,k'}(\calK)_{\Qlbar}\isom\calE_{\chibar,k'}(\calK)_{\Qlbar}
\end{equation*}
which is clearly compatible with the convolution on $\calK\in\Sat$. Therefore we get an isomorphism of tensor functors
\begin{equation}\label{tensorsigma}
\jmath_{\chi}:\sigma^{*}\calE_{\chi,k'}(-)_{\Qlbar}\isom\calE_{\chibar,k'}(-)_{\Qlbar}:\Sat\to\Loc(\X{k'},\Qlbar).
\end{equation}

We now exhibit a canonical isomorphism $\calE_{\chi,k'}(\calK)^{\vee}_{\Qlbar}\isom\calE_{\chibar,k'}(\DD\calK)_{\Qlbar}$. For a stack $S$ over $\X{}$ with structure morphism $s:S\to\X{}$, we denote by $\DD^{\rel}$ the Verdier duality functor on $D^{b}(S)$ relative to $\X{}$: $\DD^{\rel}(\calF)=\bR\uHom_{S}(\calF,\DD_{s})$, where $\DD_{s}=s^{!}\Qlbar$ is the relative dualizing complex of $s$. 

We have an isomorphism
\begin{eqnarray}\notag
\calE_{\chi,k'}(\calK)^{\vee}_{\Qlbar}&=&\left(V_{\chi}^{*}\otimes(u_{0}\times\id)^{*}\TT(\calK,j_{!}\calF_{\chi})\right)^{\tilA,\vee}\\
\label{DDE}&\cong&\left(V_{\chi}\otimes\DD^{\rel}(u_{0}\times\id)^{*}\TT(\calK,j_{!}\calF_{\chi})\right)^{\tilA}.
\end{eqnarray}
By \eqref{bch}, we have
\begin{eqnarray*}
\DD^{\rel}(u_{0}\times\id)^{*}\TT(\calK,j_{!}\calF_{\chi})&\cong&\DD^{\rel}\pi^{U}_{!}(\omega^{U,*}\calF\otimes\calK_{\GR,\Qlbar})\cong\pi^{U}_{*}(\omega^{U,*}\calF^{\vee}_{\chi}\otimes\DD^{\rel}\calK_{\GR,\Qlbar})\\
&\cong&\pi^{U}_{*}(\omega^{U,*}\calF^{\vee}_{\chi}\otimes(\DD\calK)_{\GR,\Qlbar})
\cong(u_{0}\times\id)^{*}\TT(\DD\calK,j_{!}\calF_{\chi}^{\vee}).
\end{eqnarray*}
The last equality follows from \eqref{T*}. Plugging this into \eqref{DDE}, we get a canonical isomorphism
\begin{equation}\label{DE}
\calE_{\chi,k'}(\calK)^{\vee}_{\Qlbar}\cong\left(V_{\chi}\otimes(u_{0}\times\id)^{*}\TT(\DD\calK,j_{!}\calF_{\chi}^{\vee})\right)^{\tilA}.
\end{equation}
Fix a $\Gamma'$-isomorphism $\beta: V^{*}_{\chi}\cong V_{\chibar}(\psi)$ for some character $\psi$ of $\Gal(\kbar/k')$. This also induces an isomorphism $\calF^{\vee}_{\chi}\cong\calF_{\chibar}(\psi)$. Using $\beta$ twice, we may rewrite \eqref{DE} as
\begin{equation}\label{dualE}
\calE_{\chi,k'}(\calK)^{\vee}_{\Qlbar}\isom\left(V_{\chibar}(\psi)^{*}\otimes(u_{0}\times\id)^{*}\TT(\DD\calK,j_{!}\calF_{\chibar}(\psi))\right)^{\tilA}\cong\calE_{\chibar,k'}(\DD\calK)_{\Qlbar}.
\end{equation}
which is completely canonical. One checks that \eqref{dualE} is compatible with the convolution structure of $\Sat$, hence giving an isomorphism of tensor functors
\begin{equation*}
\calE_{\chi,k'}(-)_{\Qlbar}\isom\calE_{\chibar,k'}(\DD(-))^{\vee}_{\Qlbar}:\Sat\to\Loc(\X{k'},\Qlbar).
\end{equation*}
Since $-1\in W$, Verdier duality $\DD$ on $\Sat$ serves as duality in this tensor category. Therefore the tensor functor $\calE_{\chi,k'}$ should intertwine $\DD$ and the linear duality on $\Loc(\X{k'})$: we have a canonical isomorphism of tensor functors $(\calE_{\chibar,k',\Qlbar}\circ\DD)^{\vee}\cong\calE_{\chibar,k',\Qlbar}$. Therefore \eqref{dualE} induces a canonical tensor isomorphism of tensor functors
\begin{equation}\label{nobar}
\calE_{\chi,k',\Qlbar}\isom\calE_{\chibar,k',\Qlbar}:\Sat\to\Loc(\X{k'},\Qlbar).
\end{equation}
Combining the tensor isomorphisms \eqref{tensorsigma} and \eqref{nobar}, we get a canonical tensor isomorphism
\begin{equation*}
\sigma^{*}\calE_{\chi,k',\Qlbar}\isom \calE_{\chi,k',\Qlbar}:\Sat\to\Loc(\X{k'},\Qlbar).
\end{equation*}
Canonicity guarantees that this isomorphism gives a descent datum for each $\calE_{\chi,k'}(\calK)_{\Qlbar}$ from $\X{k'}$ to $\X{k}$, which is compatible with the convolution structure on $\Sat$ as $\calK$ varies, so that $\calE_{\chi,k',\Qlbar}$ descends to a tensor functor which we give the same name
\begin{equation*}
\calE_{\chi,k',\Qlbar}:\Sat\to\Loc(\X{k},\Qlbar).
\end{equation*}
The isomorphism \eqref{nobar} is also compatible with the descent datum so it gives a canonical isomorphism of functors $\calE_{\chi,k',\Qlbar}\isom\calE_{\chibar,k',\Qlbar}$ taking values in $\Loc(\X{k},\Qlbar)$. This finishes the proof of Theorem \ref{th:eigen}(2).

\subsubsection{Rationality of the coefficients}\label{sss:Ql} Finally we prove the statement (3) in Theorem \ref{th:eigen}. It is clear from the construction that we may replace $\Qlbar$ by a finite Galois extension $E$ over $\Ql$. Therefore, for each $\calK\in\Sat$, we have an $E$-local system $\calE_{\chi}(\calK)_{E}$ on $\X{}$. For each $\tau\in\Gal(E/\Ql)$, and any complex or vector space $\calH$ with $E$-coefficients, we use $\leftexp{\tau}{\calH}$ to denote $E\otimes_{E,\tau}\calH$.

First suppose $G$ is of type $D_{4n}, E_{8}$ or $G_{2}$, then all $\chi\in\tilA^{*}_{0,\odd}$ take values in $\pm1$. The $\Gamma$-module $\leftexp{\tau}{V}_{\chi}$ is again an irreducible $\tilA(\kbar)$-module with central character $\chi$, therefore there exists some character $\psi$ of $\Gk$ and an isomorphism $\alpha:\leftexp{\tau}{V}_{\chi}\cong V_{\chi}(\psi)$, well defined up to a scalar in $E^{\times}$. This $\alpha$ also induces an isomorphism $\leftexp{\tau}{\calF}_{\chi}\isom\calF_{\chi}(\psi)$. Since $\calK$ is defined with $\Ql$-coefficients, we have an isomorphism
\begin{equation}\label{fff}
\leftexp{\tau}{\left(V_{\chi}^{*}\otimes(u_0\times\id)^*\TT(\calK, j_!\calF_{\chi})\right)}^{\tilA}\isom\left(V_{\chi}(\psi)^{*}\otimes(u_0\times\id)^*\TT(\calK, j_!\calF_{\chi}(\psi))\right)^{\tilA}
\end{equation}
which uses $\alpha$ twice, hence completely canonical even though $\alpha$ is only defined up to a scalar. This gives a canonical isomorphism
\begin{equation*}
\varphi_{\tau}:\leftexp{\tau}{\calE}_{\chi}(\calK)_{E}\cong\calE_{\chi}(\calK)_{E}.
\end{equation*}
Canonicity guarantees that the collection of isomorphism $\{\varphi_{\tau}\}_{\tau\in\Gal(E/\Ql)}$ gives a descent datum of $\calE_{\chi}(\calK)_{E}$ to $\Ql$-coefficients, and these descent data are compatible with the tensor structure as $\calK$ varies. Therefore $\calE_{\chi,E}$ descends to a $\hatG(\Ql)$-local system $\calE_{\chi}$ on $\X{}$.

Now suppose $G$ is of type $A_{1}, D_{4n+2}$ or $E_{7}$. Then all characters $\chi\in\tilA_{0}(\kbar)^{*}_{\odd}$ take values in $\Ql'=\Ql(\sqrt{-1})$. The above argument gives descent datum of $\calE_{\chi,E}$ to a $\hatG(\Ql')$-local system $\calE_{\chi,\Ql'}$. If $\Ql'\neq\Ql$, we need to further descend from $\Ql'$ to $\Ql$. Let $\tau\in\Gal(\Ql'/\Ql)$ be the nontrivial involution. Then the argument above shows there is a canonical isomorphism between local systems on $\X{k'}$ ($k'=k(\sqrt{-1})$):
\begin{equation*}
\leftexp{\tau}{\calE}_{\chi,k'}(\calK)_{\Ql'}\isom\calE_{\chibar,k'}(\calK)_{\Ql'}.
\end{equation*}
The canonical isomorphism $\calE_{\chi,k',\Qlbar}\cong\calE_{\chibar,k',\Qlbar}$ in statement (2) of Theorem \ref{th:eigen} actually descends to $\Ql'$ by checking the proofs, and we get a canonical isomorphism
\begin{equation*}
\leftexp{\tau}{\calE}_{\chi,k'}(\calK)_{\Ql'}\isom\calE_{\chibar,k'}(\calK)_{\Ql'}.
\end{equation*}
It is easy to check that this give a descent datum of $\calE_{\chi,k',\Ql'}$ to a $\hatG(\Ql)$-local system $\calE_{\chi,k'}$ on $\X{k'}$. One can also check that this descent datum is compatible with the descent datum of $\calE_{\chi,k',\Ql'}$ from $\X{k'}$ to $\X{k}$ given in Section \ref{sss:descent}. We omit details here. This finishes the proof of Theorem \ref{th:eigen}.


\section{Local and global monodromy}
For most part of this this section, we fix $\chi\in\tilA_{0}(\kbar)^{*}_{\odd}$, and denote $\calE_{\chi}$ simply by $\calE$. The results in this section will be insensitive to $\chi$.

To emphasize the dependence on the base field, we use $\calE_{k}$ to denote the local system $\calE$ over $\X{k}$. The monodromy representation of $\calE_{k}$ is
\begin{equation*}
\rho_{k}:\pi_1(\X{k})\to\hatG(\Ql).
\end{equation*}
The goal of this section is to study the image of $\rho_{k}$ as well as its restriction to certain subgroups of $\pi_{1}(\X{k})$.

\subsection{Remarks on Gaitsgory's nearby cycles}\label{ss:genloc}
We first make some general remarks about the calculation of local monodromy via Gaitsgory's nearby cycle construction \cite{Ga}. 

\subsubsection{Hecke operators at ramified points} Let $X$ be a complete smooth connected curve over an algebraically closed field $k$ and $S$ be a finite set of points on $X$. Let $\{\bP_{x}\}$ be a set of level structures, one for each $x\in S$. For each $x\in S$, we have a Hecke correspondence
\begin{equation*}
\xymatrix{& \Hk_{x}\ar[rr]^{\inv_{x}}\ar[dl]_{\oleft{h}_{x}}\ar[dr]^{\oright{h}_{x}} && \bP_{x}\backslash L_{x}G/\bP_{x}\\
\Bun_{G}(\bP_{x};x\in S) & & \Bun_{G}(\bP_{x};x\in S)}
\end{equation*}
which classifies triples $(\calP,\calP',\iota)$ where $\calP,\calP'\in\Bun_{G}(\bP_{x};x\in S)$ and $\iota:\calP|_{X-S}\isom\calP'|_{X-S}$ preserving the level structures at $S-\{x\}$.

Analogous to the geometric Hecke operators, we may define an action of the monoidal category $D^{b}(\bP_{x}\backslash LG/\bP_{x})$ on $D^{b}(\Bun_{G}(\bP_{x};x\in S))$ as follows. For $\calK\in D^{b}(\bP_{x}\backslash LG/\bP_{x})$ and $\calF\in D^{b}(\Bun_{G}(\bP_{x};x\in S))$, we define
\begin{equation*}
\TT_{x}(\calK,\calF):=\oright{h}_{x,!}(\oleft{h}^{*}_{x}\calF\otimes\inv_{x}^{*}\calK).
\end{equation*}

We also need the Hecke modifications at two points, one moving point not in $S$ and the other is $x\in S$. Let $S^{x}=S-\{x\}$, we have a diagram
\begin{equation*}
\xymatrix{& \Hk_{X-S^{x}}\ar[rr]^{\inv}\ar[dl]_{\oleft{h}}\ar[dr]^{\oright{h}} && \left[\dfrac{L^{+}G\backslash LG/L^{+}G}{\Aut_{\calO}}\right]\\
\Bun_{G}(\bP_{x};x\in S) & & \Bun_{G}(\bP_{x};x\in S)\times(X-S^{x})}
\end{equation*}
The stack $\Hk_{X-S^{x}}$ classifies $(y,\calP,\calP',\iota)$ where $y\in X-S^x, \calP,\calP'\in\Bun_{G}(\bP_{x};x\in S)$ and $\iota:\calP|_{X-S-\{y\}}\isom\calP'|_{X-S-\{y\}}$ preserving the level structures at $S-\{x\}$. The map ``inv'' records the relative position of $\calP$ and $\calP'$ on the formal neighborhood of $y$. In particular, $\Hk_{x}\subset\Hk_{X-S^{x}}$ is the fiber at $y=x$.

\subsubsection{Parahoric version of Gaitsgory's construction} We recall the setting of \cite{Ga}. Fixing a base point $u_{0}:\Spec k\to\Bun_{G}(\bP_{x};x\in S)$. Let $\GR_{X\backslash S^{x}}=\oright{h}^{-1}(\{u_{0}\}\times(X\backslash S^{x}))$. Let $\Fl_{\bP_{x}}=L_{x}G/\bP_{x}$ be the affine partial flag variety associated to $\bP_{x}$. The family $\GR_{X-S^{x}}\to X-S^{x}$ interpolates $\Gr_{y}\times\Fl_{\bP_{x}},y\notin S$ and $\Fl_{\bP_{x}}$ at $x$:
\begin{equation*}
\xymatrix{\Fl_{\bP_{x}}\ar[r]\ar[d] & \GR_{X-S^{x}}\ar[d] & \GR_{X- S}\times\Fl_{\bP_{x}}\ar[d]\ar[l]\\
\{x\}\ar[r] & X-S^{x} & X- S\ar[l]}
\end{equation*}
The nearby cycles functor defines a functor
\begin{equation*}
\Psi_{\bP_{x}}:\Sat\times D^{b}(\Fl_{\bP_{x}})\to D^{b}(\Fl_{\bP_{x}})
\end{equation*}
sending $(\calK,\calF)$ to the nearby cycles of $\calK_{\GR}\boxtimes\calF$, where $\calK_{\GR}\in D^{b}(\GR_{X- S})$ is the spread-out of $\calK\in\Sat$ as in Section \ref{sss:BD}. Setting $\calF$ to be the skyscraper sheaf $\delta_{\bP_{x}}$ at the base point of $\Fl_{\bP_{x}}$, we get a t-exact functor (by the exactness of nearby cycles functor)
\begin{eqnarray*}
Z_{\bP_{x}}:\Sat &\to&\Perv(\Fl_{\bP_{x}})\\
\calK&\mapsto&\Psi(\calK\boxtimes\delta_{\bP_{x}}).
\end{eqnarray*}
When $\bP_{x}=\bI_{x}$, we recover Gaitsgory's original nearby cycles functor
\begin{equation*}
Z_{x}:=Z_{\bI_{x}}:\Sat\to\Perv(\Fl_{x})
\end{equation*}
In \cite[Theorem 1]{Ga}, Gaitsgory proves that $Z_{x}(\calK)$ is in fact left-$\bI_{x}$-equivariant and is convolution exact and central: for any $\calF\in\Perv(\bI_{x}\backslash L_{x}G/\bI_{x})$, $Z_{x}(\calK)*\calF$ is perverse and there is a canonical isomorphism $Z_{x}(\calK)*\calF\cong\calF*Z_{x}(\calK)$. Here $*$ denotes the convolution product on $D^{b}(\bI_{x}\backslash L_{x}G/\bI_{x})$.

Now we relate $Z_{\bP_{x}}$ to $Z_{x}$. Define $\GR'_{X- S^{x}}$ in a similar way as $\GR_{X- S^{x}}$, except we replace $\bP_{x}$ by $\bI_{x}$. We have a commutative diagram
\begin{equation*}
\xymatrix{\Fl_{x}\ar@{^{(}->}[r]\ar[d]^{p_{x}} & \GR'_{X- S^{x}}\ar[d]^{p_{X- S^{x}}} & \GR_{X- S}\times\Fl_{x}\ar@{_{(}->}[l]\ar[d]^{\id\times p_{x}}\\
\Fl_{\bP_{x}}\ar@{^{(}->}[r]\ar[d] & \GR_{X- S^{x}}\ar[d] & \GR_{X- S}\times \Fl_{\bP_{x}}\ar@{_{(}->}[l]\ar[d]\\
\{x\}\ar@{^{(}->}[r] & X- S^{x} & X- S\ar@{_{(}->}[l]}
\end{equation*}
where all squares are Cartesian and $p_{x}$ and $p_{X- S^{x}}$ are proper. Since nearby cycles commute with proper base change, we have an $\Gal(F^{s}_{x}/F_{x})$-equivariant isomorphism
\begin{equation}\label{px}
Z_{\bP_{x}}(\calK)=\Psi(\calK\boxtimes\delta_{\bP_{x}})=\Psi(\calK\boxtimes p_{x,*}\delta_{\bI_{x}})\cong p_{x,*}\Psi(\calK\boxtimes\delta_{\bI_{x}})=p_{x,*}Z_{x}(\calK).
\end{equation}
Let $C_{\bP_{x}}\in D^{b}(\bI_{x}\backslash L_{x}G/\bI_{x})$ be the constant sheaf on $\bI_{x}\backslash \bP_{x}/\bI_{x}$, which, up to a shift, belongs to $\Perv(\bI_{x}\backslash L_{x}G/\bI_{x})$. Therefore we have a canonical isomorphism
\begin{equation}\label{pullconv}
p^{*}_{x}Z_{\bP_{x}}(\calK)=p^{*}_{x}p_{x,*}Z_{x}(\calK)=Z_{x}(\calK)*C_{\bP_{x}}\cong C_{\bP_{x}}*Z_{x}(\calK)
\end{equation}
In the second equality above we have used the centrality of $Z_{x}(\calK)$. Any object of the form $C_{\bP_{x}}*(-)$ is equivariant under the left action of $\bP_{x}$ on $\Fl_{x}$ because $C_{\bP_{x}}$ is, therefore $p^{*}_{x}Z_{\bP_{x}}(\calK)$ is also left $\bP_{x}$-equivariant, hence descends to an object in $\Perv(\bP_{x}\backslash L_{x}G/\bP_{x})$.

Let $\calF\in D^{b}(\Bun_{G}(\bP_{x};x\in S))$ be a Hecke eigensheaf with eigen local system $\calE:\Sat\to\Loc(X- S)$. Let $u_{0}:\Spec k\to\Bun_{G}(\bP_{x};x\in S)$ be a base point.

\begin{lemma}\label{l:near}
For each $\calK\in\Sat$, there is a canonical isomorphism
\begin{equation*}
u_{0}^{*}\calF\otimes\calE(\calK)|_{\Spec F^{s}_{x}}\cong u_{0}^{*}\TT_{x}(Z_{\bP_{x}}(\calK),\calF).
\end{equation*}
which is equivariant with respect to the $\Gal(F^{s}_{x}/F_{x})$-action (on the right side it acts on the nearby cycles $Z_{\bP_{x}}(\calK)$).
\end{lemma}
\begin{proof}
The argument is the same as in \cite[Section 4.3]{HNY}.
\end{proof}

\begin{cor}\label{c:tame}
For any $\calK\in\Sat$, $\calE(\calK)\in\Loc(X-S)$ is tamely ramified and the monodromy at every $x\in S$ is unipotent.
\end{cor}
\begin{proof}
Pick a point $u_{0}$ such that $u_{0}^{*}\calF\neq0$. By Gaitsgory's result \cite[Proposition 7]{Ga}, $\Gal(F^{s}_{x}/F_{x})$ acts on $Z_{x}(\calK)$ tamely and unipotently, hence the same is true on $Z_{\bP_{x}}(\calK)$ by \eqref{px}, and on $\calE(\calK)|_{\Spec F_{x}}$ by Lemma \ref{l:near}.
\end{proof}

In our applications, we will consider $S=\{0,1,\infty\}$ and consider $\Bun$ instead of $\Bun_{G}(\bP_{0},\bP_{1},\bP_{\infty})$. The category $D^{b}(\Bun)_{\odd}$ is clearly preserved by the Hecke operators at $1$ or $\infty$. Corollary \ref{c:tame} shows that $\calE$ is tame at $1$ and $\infty$.

\subsection{Local monodromy}\label{ss:loc}
For a closed point of $x\in \PP^1$, let $I_{x}\subset\Gal(F^{s}_{x}/F_{x})$ be the inertia group at $x$. Let $I_x^{\tame}$ be the tame quotient of $I_x$. Since we only care about the action of the inertia groups in this section, we assume $k$ is algebraically closed.

\begin{prop}\label{p:mono1}
Under the homomorphism $\rho_{k}$, a topological generator of $I_1^{\tame}$ gets mapped to a regular unipotent element in $\hatG(\Ql)$.
\end{prop}
\begin{proof}
We would like to use the argument of \cite[Section 4.3]{HNY}. The only thing we need to show is that, for each irreducible object $\IC_{\tilw}\in D^{b}(\bI_{1}\backslash L_{1}G/\bI_{1})$ (indexed by an element $\tilw$ in the affine Weyl group $\tilW$), where $\tilw\neq 1$, we have $\TT_{1}(\IC_{\tilw},\calF)=0$ for any object $\calF\in D^{b}(\Bun)_{\odd}$. Since $\tilw\neq 1$, there exists a simple reflection $s_{i}$ such that $\tilw=\tilw's_{i}$ and $\ell(\tilw)=\ell(\tilw')+1$. Let $\bP_{1,i}$ be the parahoric subgroup of $L_{1}G$ generated by $\bI_{1}$ and the root subgroup of $-\alpha_{i}$ (the simple root corresponding to $s_{i}$). Let $p_{i}:\Fl=L_{1}G/\bI_{1}\to\Fl_{i}=L_{1}G/\bP_{1,i}$ be the projection. Then $\IC_{\tilw}\cong p_{i}^{*}\IC_{i,\tilw}$ for some object $\IC_{i,\tilw}\in D^{b}(\bI_{1}\backslash L_{1}G/\bP_{1,i})$. With the $\bP_{1,i}$-level structure, we may define another Hecke correspondence at $t=1$:
\begin{equation*}
\xymatrix{& \Hk_{1,i}\ar[rr]^{\inv}\ar[dl]_{\oleft{h}}\ar[dr]^{\oright{h}} && \bI_{1}\backslash L_{1}G/\bP_{1,i}\\
\Bun & & \Bun_{G}(\wt{\bP_{0}},\bP_{1,i},\bP_{\infty})}
\end{equation*} 
and define a Hecke operator:
\begin{eqnarray*}
\TT_{1}(\IC_{i,[\tilw]},-):D^{b}(\Bun)_{\odd}&\to& D^{b}(\Bun(\wt{\bP_{0}},\bP_{1,i},\bP_{\infty}))_{\odd}\\
\calF&\mapsto&\oright{h}_{!}(\oleft{h}^{*}\calF\otimes\inv^{*}\IC_{i,\tilw}).
\end{eqnarray*}
Let $p_{\Bun,i}:\Bun=\Bun(\wt{\bP_{0}},\bI_{1},\bP_{\infty})\to\Bun_{G}(\wt{\bP_{0}},\bP_{1,i},\bP_{\infty})$ be the projection. Then we have
\begin{equation*}
\TT_{1}(\IC_{\tilw},\calF)\cong p_{\Bun,i}^{*}\TT_{1}(\IC_{i,[\tilw]},\calF).
\end{equation*}
We now claim that the category $D^{b}(\Bun(\wt{\bP_{0}},\bP_{1,i},\bP_{\infty}))_{\odd}$ is zero, which then implies $\TT_{1}(\IC_{\tilw},\calF)=0$ and completes the proof. 

Let $\calH\in D^{b}(\Bun_{G}(\wt{\bP_{0}},\bP_{1,i},\bP_{\infty}))_{\odd}$ be a nonzero object. We view $\calH$ as a $\tilK$-equivariant complex on $\Bun_{G}(\bP_{0}^{+},\bP_{1,i},\bP_{\infty})$. Let $v:\Spec k\to\Bun_{G}(\bP_{0}^{+},\bP_{1,i},\bP_{\infty})$ be a point where the stalk of $\calH$ is nonzero. Let $q_{i}:\Bun^{+}\to\Bun_{G}(\bP_{0}^{+},\bP_{1,i},\bP_{\infty})$ be the projection, whose fibers are isomorphic to $\bP_{1,i}/\bI_{1}\cong\PP^{1}$. Then $q_{i}^{-1}(v)\in U$ because $\calH$ has to vanish outside $U\subset\Bun^{+}$ be Theorem \ref{th:clean}. We have the following Cartesian diagram
\begin{equation*}
\xymatrix{\PP^{1}\ar[d]\ar[r] & [\PP^{1}/\Aut(v)]\ar@{^{(}->}[r]^(.6){i'_{v}}\ar[d] & U\ar@{^{(}->}[r] & \Bun^{+}\ar[d]^{q_{i}}\\
\{v\}\ar[r] & [\{v\}/\Aut(v)]\ar@{^{(}->}[rr]^{i_{v}} & & \Bun_{G}(\bP_{0}^{+},\bP_{1,i},\bP_{\infty})}
\end{equation*}
Since $i_{v}$ is representable, so is its base change $i'_{v}$, which implies that the action of $\Aut(v)$ on $\PP^{1}$ is free. On the other hand, the morphism $\PP^{1}\to[\PP^{1}/\Aut(v)]\hookrightarrow U$ has to be constant because $\PP^{1}$ is proper while $U$ is affine. Therefore $\Aut(v)$ acts on $\PP^{1}$ both freely and transitively. This implies $\PP^{1}$ is a torsor under $\Aut(v)$. However $\PP^{1}$ is not isomorphic to any algebraic group. Contradiction! Hence $\calH$ has to be zero everywhere.
\end{proof}

\subsubsection{An involution in $\hatG$}\label{sss:kappa} By the construction of the canonical double cover in Section \ref{sss:doublecover}, we have $\xcoch(T)/\xcoch(\tilT)\cong\ZZ/2\ZZ$, where $\tilT$ is the preimage of $T$ in $\tilK$. This defines an order two character
\begin{equation*}
\xcoch(T)\twoheadrightarrow\xcoch(T)/\xcoch(\tilT)\cong\ZZ/2\ZZ\isom\{\pm1\}
\end{equation*}
and hence an element $\kappa\in\hatT[2]$. One can check case by case that when $G$ is of type $A_{1},D_{2n},E_{7},E_{8}$ or $G_{2}$, $\kappa$ is always a split Cartan involution in $\hatG$ (see Section \ref{sss:Cartaninv}).

\begin{prop}\label{p:mono0}
The local system $\calE_{k}$ is tame at $0$. Under the homomorphism $\rho_{k}$, a topological generator of $I_0^{\tame}$ gets mapped to an element with Jordan decomposition $g_{s}g_{u}\in\hatG(\Ql)$, where the semisimple part $g_{s}$ is conjugate to $\kappa\in\hatT[2]$.
\end{prop}
\begin{proof}
Pulling back the double covering $\wt{\bP_{0}}\to\bP_{0}$ to $\bI_{0}\subset\bP_{0}$, we get a double covering $\wt{\bI_{0}}\to\bI_{0}$. The reductive quotient of $\wt{\bI_{0}}$ is $\tilT$. We will consider the moduli stack $\Bun_{G}(\wt{\bI_{0}},\bI_{1},\bP_{\infty})$, defined similarly as $\Bun$. As we discussed in Section \ref{sss:Keq}, we can define the category such as $D^{b}(\Bun_{G}(\wt{\bI_{0}},\bI_{1},\bP_{\infty}))_{\odd}$ and $D^{b}(L_{0}G/\wt{\bI_{0}})_{\odd}$ etc. The inclusion $\wt{\bI_{0}}\hookrightarrow\wt{\bP_{0}}$ gives a projection
\begin{equation*}
p:\Bun_{G}(\wt{\bI_{0}},\bI_{1},\bP_{\infty})\to\Bun_{G}(\wt{\bP_{0}},\bI_{1},\bP_{\infty})=\Bun.
\end{equation*}
Let $j_!\calF$ be the Hecke eigensheaf with eigen local system $\calE$. The complex $p^{*}j_{!}\calF$ is clearly also a Hecke eigensheaf on $\Bun_{G}(\wt{\bI_{0}},\bI_{1},\bP_{\infty})$ for the Hecke operators on $\X{}$ with the same eigen local system $\calE$. Therefore it suffices to prove a stronger statement: for any nonzero Hecke eigensheaf $\calF$ on $\Bun_{G}(\wt{\bI_{0}},\bI_{1},\bP_{\infty})$ with eigen local system $\calE$, the $\hatG$-local system $\calE$ is tamely ramified and the semisimple part of the local monodromy is conjugate to $\kappa$.

We need another variant of Gaitsgory's nearby cycles construction allowing sheaves which are monodromic with respect to the torus action. This variant is sketched in \cite[Section 2.1 and 2.2]{BFO}. Let us be more specific about the version we need. There is a family $\GR'_{\PP^{1}-\{1,\infty\}}$ interpolating $\Gr_{x}\times L_{0}G/\wt{\bI_{0}}$ and $L_{0}G/\wt{\bI_{0}}$
\begin{equation*}
\xymatrix{L_{0}G/\wt{\bI_{0}}\ar@{^{(}->}[r]\ar[d] & \GR'_{\PP^{1}-\{1,\infty\}}\ar[d] & \GR_{\X{}}\times L_{0}G/\wt{\bI_{0}}\ar@{_{(}->}[l]\ar[d]\\
\{0\}\ar@{^{(}->}[r] & \PP^{1}-\{1,\infty\} & \X{}\ar@{_{(}->}[l]}
\end{equation*} 
Let $\delta_{\odd}\in\Perv(L_{0}G/\wt{\bI_{0}})$ be the rank one local system supported on $\bI_{0}/\wt{\bI_{0}}$ on which $\mu_{2}^{\ker}$ acts by the sign representation. Using the nearby cycles of the above family, we define
\begin{eqnarray*}
Z'_{0}:\Sat&\to&\Perv(\wt{\bI_{0}}\backslash L_{0}G/\wt{\bI_{0}})_{\odd}\\
\calK&\mapsto&\Psi(\calK_{\GR}\boxtimes\delta_{\odd}).
\end{eqnarray*}
Here the subscript ``odd'' in $\Perv(\wt{\bI_{0}}\backslash L_{0}G/\wt{\bI_{0}})_{\odd}$ means taking those objects on which both actions of $\mu_{2}^{\ker}$ (from left and right) are through the sign representation. The derived category $D^{b}(\wt{\bI_{0}}\backslash L_{0}G/\wt{\bI_{0}})_{\odd}$ still acts on $D^{b}(\Bun_{G}(\wt{\bI_{0}},\bI_{1},\bP_{\infty}))_{\odd}$, using the same construction as in Section \ref{ss:genloc}. We denote this action by
\begin{equation*}
\TT'_{0}:D^{b}(\wt{\bI_{0}}\backslash L_{0}G/\wt{\bI_{0}})_{\odd}\times D^{b}(\Bun_{G}(\wt{\bI_{0}},\bI_{1},\bP_{\infty}))_{\odd}\to D^{b}(\Bun_{G}(\wt{\bI_{0}},\bI_{1},\bP_{\infty}))_{\odd}.
\end{equation*}
We also have a variant of Lemma \ref{l:near}: there is an $I_{0}$-equivariant isomorphism
\begin{equation}\label{near0}
u_{0}^{*}\calF\otimes\calE(\calK)|_{\Spec F^{s}_{0}}\cong u_{0}^{*}\TT'_{0}(Z'_{0}(\calK),\calF).
\end{equation}
for any $\calK\in\Sat$. By \cite[Section 5.2, Claim 2]{AB}, the monodromy action on $Z'_{0}(\calK)$ factors through the tame quotient (the Claim in {\em loc.cit.} requires the family to live over $\AA^{1}$ and carry a $\Gm$-action compatible with the rotation action on $\AA^{1}$. In our situation, we can extend $\GR'_{\PP^{1}-\{1,\infty\}}$ to $\GR'_{\AA^{1}}$ by ignoring the level structure at $1$, and the rotation action on $\PP^{1}$ induces the desired action on $\GR'_{\AA^{1}}$). Therefore the local system $\calE(\calK)$ is tame at $0$ for any $\calK$.

The following argument is borrowed from \cite[Section 2.4 and 2.5]{BFO}, to which we refer more details. There is a filtration $F_{\lambda}$ (indexed by $\lambda\in\xcoch(T)$, partially ordered using the positive coroot lattice) of $Z'_{0}(\calK)$ with $\Gr^{F}_{\lambda}Z'_{0}(\calK)$ isomorphic to a direct sum of the Wakimoto sheaf $J_{\lambda}$. Moreover, the monodromy operator $m(\calK)$ on $Z'_{0}(\calK)$ preserves this filtration, and acts on $\Gr^{F}_{\lambda}Z'_{0}(\calK)$ by $\kappa(\lambda)$ (recall any element in $\hatT(\Ql)$ is a homomorphism $\xcoch(T)\to\Ql^{\times}$). Therefore, we can write $m(\calK)$ into Jordan normal form $m(\calK)_{s}m(\calK)_{u}=m(\calK)_{u}m(\calK)_{s}$. Here both $m(\calK)_{u}$ and $m(\calK)_{s}$ preserve the filtration $F_{\lambda}$, $m(\calK)_{s}$ is an involution which agrees with the $m(\calK)$-action on $\Gr^{F}_{\lambda}Z_{0}'(\calK)$ and $m(\calK)_{u}$ is unipotent which acts as identity on $\Gr^{F}_{\lambda}Z_{0}'(\calK)$. On the other hand, let $\zeta_{0}\in I^{\tame}_{0}$ be a topological generator, and let  $\rho_{k}(\zeta_{0})=g_{s}g_{u}=g_{u}g_{s}$ be the Jordan decomposition of $\rho_{k}(\zeta_{0})$ in $\hatG(\Ql)$. Since the isomorphism \eqref{near0} intertwines the action of $\rho_{k}(\zeta_{0})$ on $\calE(\calK)|_{\Spec F^{s}_{0}}$ and the action of $m(\calK)$ on $Z'_{0}(\calK)$, it must also intertwine the $g_{s}$-action and the $m(\calK)_{s}$-action by the uniqueness of Jordan decomposition. 

Let $\calA$ be the full subcategory of $\Perv(\wt{\bI_{0}}\backslash L_{0}G/\wt{\bI_{0}})_{\odd}$ consisting of those objects admitting filtrations with graded pieces isomorphic to Wakimoto sheaves. Let $\Gr\calA$ be the full subcategory of $\calA$ consisting of direct sums of Wakimoto sheaves. Then $\Gr\calA\cong\Rep(\hatT)$ as tensor categories, and the functor
\begin{equation*}
\Rep(\hatG,\Ql)\cong\Sat\xrightarrow{Z'_{0}}\calA\xrightarrow{\oplus\Gr^{F}_{\lambda}}\Gr\calA\cong\Rep(\hatT,\Ql)
\end{equation*}
is isomorphic to the restriction functor $\Res^{\hatG}_{\hatT}:\Rep(\hatG)\to\Rep(\hatT)$ induced by the inclusion $\hatT\hookrightarrow\hatG$. The semisimple part of the monodromy operator $\{m(\calK)_{s}\}_{\calK\in\Sat}$ acts as an automorphism of $\Res^{\hatG}_{\hatT}$, hence determines an element $\tau\in\hatT(\Ql)=\Aut^{\otimes}(\Res^{\hatG}_{\hatT})$. The above discussion has identified the action of $\tau$ on Wakimoto sheaves (which correspond to irreducible algebraic representations of $\hatT$ under the equivalence $\Gr\calA\cong\Rep(\hatT)$), hence $\tau=\kappa$. To summarize, $\{m(\calK)_{s}\}_{\calK\in\Sat}$ gives a tensor automorphism of the fiber functor $\omega:\Sat\cong\Rep(\hatG)\xrightarrow{\Res}\Rep(\hatT)\to\Vec_{\Ql}$ which is the same as $\kappa$.

On the other hand, $g_{s}$ gives an automorphism of another fiber functor $\omega':\Sat\to\Vec_{\Ql}$ given by taking the fiber at $\Spec F^{s}_{0}$. Since $g_{s}$ corresponds to $m(\calK)_{s}$ under \eqref{near0}, upon identifying the fiber functors $\omega$ and $\omega'$, $g_{s}$ and $\kappa$ should be the same automorphism, i.e., they are conjugate elements in $\hatG(\Ql)$. This proves the Proposition.
\end{proof}

\begin{prop}\label{p:monoinf}
Suppose $G$ is not of type $A_1$. Under the homomorphism $\rho_{k}$, a topological generator of $I_\infty^{\tame}$ gets mapped to a unipotent element in $\hatG(\Ql)$ which is neither regular nor trivial.
\end{prop}
\begin{proof}
Let $\zeta_{\infty}\in I_{\infty}^{\tame}$ be a topological generator. Let $N=\log(\rho_{k}(\zeta_{\infty}))$, which is a nilpotent element in $\hatg$. 

We first argue that $N$ is not regular. Suppose it is, then it acts on $\hatg$ with a Jordan block of size $2h-1$, where $h$ is the Coxeter number of $\hatG$. In other words, $\Ad(N)^{2h-2}\neq0$ on $\hatg$. Let $\gamma$ be the coroot of $G$ corresponding to the highest root of $\hatG$. Then $\IC_{\gamma}\in\Sat$ corresponds to the adjoint representation of $\hatG$ under the Satake equivalence \eqref{Satake}. In Section \ref{ss:genloc} we recalled the parahoric variant of Gaitsgory' nearby cycles functor $Z_{\bP_{\infty}}$. By Lemma \ref{l:near}, we have an $I_{\infty}$-equivariant isomorphism
\begin{equation}\label{Adnear}
V_{\chi}\otimes\calE(\IC_{\gamma})|_{\Spec F^{s}_{\infty}}\cong u_{0}^{*}\TT_{\infty}(Z_{\bP_{\infty}}(\IC_{\gamma}),j_{!}\calF_{\chi}).
\end{equation}
where $I_{\infty}$ acts trivially on $V_{\chi}$. In particular, the action of $N$ on the left side is reflected from the logarithm of the monodromy action on $Z_{\bP_{\infty}}(\IC_{\gamma})$. By \eqref{pullconv}, we have an isomorphism which respects the monodromy operators
\begin{equation*}
p^{*}_{\infty}Z_{\bP_{\infty}}(\IC_{\gamma})\cong Z_{\infty}(\IC_{\gamma})*C_{\bP_{\infty}}.
\end{equation*}
Let $M$ be the logarithm of the monodromy operator on $Z_{\infty}(\IC_{\gamma})$, and $M'$ be the induced endomorphism of $Z_{\infty}(\IC_{\gamma})*C_{\bP_{\infty}}$, which intertwines with the logarithm monodromy on $Z_{\bP_{x}}(\IC_{\gamma})$, hence with $N$ via \eqref{Adnear}.
Since we assumes $\Ad(N)^{2h-2}\neq0$, we must have $(M')^{2h-2}\neq0$ on $Z_{\infty}(\IC_{\gamma})*C_{\bP_{\infty}}$, and $M^{2h-2}$ is also nonzero on $Z_{\infty}(\IC_{\gamma})$. 

The following argument uses the theory of weights. For $\textup{char}(k)>0$, every object in concern comes from a finite base field $\FF_q$, hence we may assume $k=\overline{\FF}_q$ and use the weight theory of Weil sheaves. For $\textup{char}(k)=0$, every object in concern comes from a number field $L$, and we may still talk about weights by choosing a place $v$ of $L$ at which all objects have good reduction.

The object $Z_{\infty}(\IC_{\gamma})$ has a Jordan-Holder series whose associated graded are Tate twists of irreducible objects $\IC_{\tilw}$ for $\tilw\in\tilW$. In \cite[Theorem 1.1]{GH}, G\"ortz and Haines give an estimate of the weights of the twists of $\IC_{\tilw}$ appearing in $Z_{\infty}(\IC_{\gamma})$: $\IC_{\tilw}$ appears with weight in the range $[\ell(\tilw)-\ell(\gamma), \ell(\gamma)-\ell(\tilw)]$ (note the different normalization we take here and in \cite{GH}: we normalize $\IC_{\tilw}$ and $\IC_{\gamma}$ to have weight zero while in \cite{GH} they have weight $\ell(\tilw)$ and $\ell(\gamma)$ respectively). Here $\ell(\gamma)$ is the length of the translation element $\gamma$ in $\tilW$, and in fact $\ell(\gamma)=\jiao{2\rho,\gamma}=2h-2$. Since the logarithmic monodromy operator decreases weight by 2, the subquotients isomorphic to (a twist of) $\IC_{\tilw}$ are killed after applying $M$ for $\ell(\gamma)-\ell(\tilw)=2h-2-\ell(\tilw)$ times (if this is negative, this means $\IC_{\tilw}$ does not appear in $Z_{\infty}(\IC_{\gamma})$ at all). Therefore, only the skyscraper sheaf $\delta=\IC_{e}$ survives after applying $M$ for $2h-2$ times, i.e., $M^{2h-2}$ factors as
\begin{equation*}
M^{2h-2}:Z_{\infty}(\IC_{\gamma})\twoheadrightarrow\delta(1-h)\hookrightarrow Z_{\infty}(\IC_{\gamma})(2-2h)
\end{equation*}
Here the first arrow is the passage to $\Gr^{W}_{2h-2}Z_{\infty}(\IC_{\gamma})$, the maximal weight quotient, and the second arrow is induced from $\delta(h-1)\cong\Gr^{W}_{2-2h}Z_{\infty}(\IC_{\gamma})\hookrightarrow Z_{\infty}(\IC_{\gamma})$, the inclusion of the lowest weight piece. Consequently, $M'^{2h-2}$ factors as
\begin{equation}\label{factorC}
M'^{2h-2}:Z_{\infty}(\IC_{\gamma})*C_{\bP_{\infty}}\to C_{\bP_{\infty}}(1-h)\to Z_{\infty}(\IC_{\gamma})*C_{\bP_{\infty}}(2-2h).
\end{equation}
Making a shift of $C_{\bP_{\infty}}$, we may normalize it to be a perverse sheaf, so that the above sequence is in the category $\calP=\Perv(\bI_{\infty}\backslash L_{1}G/\bI_{\infty})$. 

To proceed, we need Lusztig's theory of two-sided cells and Bezrukavnikov's geometric result on cells. Let $\unc$ be the cell containing the longest element of $w_{K}\in W_K$ (note that $\IC_{w_{K}}$ is a twist of $C_{\bP_{\infty}}$). Let $\calP_{\leq\unc}$ (resp. $\calP_{<\unc}$) be the full subcategory of $\calP$ generated (under extensions) by $\{\IC_{\tilw}\}$ where $\tilw$ belongs to some two-sided cell $\leq\unc$ (resp. $<\unc$). Let $\calP_{\unc}=\calP_{\leq\unc}/\calP_{<\unc}$ be the Serre quotient. Since $C_{\bP_{\infty}}\in\calP_{\leq\unc}$, we have $Z_{\infty}(\calK)*C_{\bP_{\infty}}\subset\calP_{\leq\unc}$. The element $w_{K}\in\unc$ is a {\em distinguished (or Duflo) involution}(see \cite[Section 1.3]{Cell2} for definition, which uses the $a$-function on two sided-cells \cite[Section 2.1]{Cell1}. In our case, we may apply \cite[Proposition 2.4]{Cell1} to show $a(w_K)=\ell(w_K)$, which forces $w_K$ to be a Duflo involution by definition). In \cite[Section 4.3]{B}, Bezrukavnikov introduces a full subcategory $\calA_{w_{K}}\subset\calP_{\unc}$ generated by the image of all subquotients of $Z_{\infty}(\calK)*\IC_{w_{K}}$ in the category $\calP_{\unc}$, as $\calK$ runs over $\Sat$. It is proved there that $\calA_{w_{K}}$ has a natural structure of a monoidal abelian category with unit object the image of $C_{\bP_{\infty}}$, such that the functor
\begin{eqnarray*}
\Res_{w_{K}}:\Sat&\to&\calA_{w_{K}}\\
\calK&\mapsto& [Z_{\infty}(\calK)*C_{\bP_{\infty}}]
\end{eqnarray*} 
is monoidal. Here we use $[-]$ to denote the passage from $\calP_{\leq\unc}$ to $\calP_{\unc}$. By \cite[Theorem 1]{B}, there is a subgroup $\hatH\subset\hatG$ and a unipotent element $u_{w_{K}}\in\hatG(\Qlbar)$ commuting with $\hatH$, and an equivalence of tensor categories $\Phi_{w_{K}}:\calA_{w_{K}}\cong\Rep(\hatH)$ such that the following diagram is commutative (by a natural isomorphism)
\begin{equation*}
\xymatrix{\Sat\ar[d]^{\textup{Satake}}_{\wr}\ar[r]^{\Res_{w_{K}}} & \calA_{w_{K}}\ar[d]^{\Phi_{w_{K}}}_{\wr}\\
\Rep(\hatG)\ar[r]^{\Res^{\hatG}_{\hatH}} & \Rep(\hatH)}
\end{equation*}
where $\Res^{\hatG}_{\hatH}$ is the restriction functor. Moreover, the monodromy operator on $Z_{\infty}(-)$ induces a natural automorphism of $\Res_{w_{K}}$, which corresponds to the natural automorphism of $u_{w_{K}}$ on $\Res^{\hatG}_{\hatH}$. There is a bijection defined by Lusztig \cite[Theorem 4.8(b)]{Cell4}
\begin{equation*}
\{\textup{two sided cells of } \tilW\}\isom\{\textup{unipotent classes in } \hatG\}
\end{equation*}
By \cite[Theorem 2]{B}, the unipotent element $u_{w_{K}}$ is in the unipotent class corresponding to $\unc$ under Lusztig's bijection. Since $G$ is not of type $A_{1}$, $\ell(w_{K})>0$, hence $u_{w_{K}}$ is not the regular unipotent class (this is because $a(w_K)=\ell(w_K)>0$, while the two-sided cell corresponding to the regular class has $a$-value equal to 0). Therefore $\log(u_{w_{K}})^{2h-2}$ is zero on $\hatg$, because only regular nilpotent elements have a Jordan block of size $2h-1$ under the adjoint representation. Because $\log(u_{w_{K}})$ intertwines with the logarithmic monodromy on $\Res_{w_{K}}(-)$, the logarithmic monodromy $M'$ on $[Z_{\infty}(\IC_{\gamma})*C_{\bP_{\infty}}]=\Res_{w_{K}}(\IC_{\gamma})$ must satisfy $M'^{2h-2}=0$, as a morphism in the Serre quotient $\calP_{\unc}$. Hence $M'^{2h-2}$ as a morphism in the category $\calP_{\leq\unc}$ factors as
\begin{equation*}
M'^{2h-2}:Z_{\infty}(\IC_{\gamma})*C_{\bP_{\infty}}\twoheadrightarrow Q\hookrightarrow Z_{\infty}(\IC_{\gamma})*C_{\bP_{\infty}}(2-2h).
\end{equation*}
where the image $Q$ lies in $\calP_{<\unc}$. However, we have another factorization \eqref{factorC}, hence there must be an arrow $q:Q\to C_{\bP_{\infty}}(1-h)$ through which the inclusion $Q\hookrightarrow Z_{\infty}(\IC_{\gamma})*C_{\bP_{\infty}}(2-2h)$ factors. However, $C_{\bP_{\infty}}$ is an irreducible object of $\calP_{\leq\unc}$ which does not lie in $\calP_{<\unc}$ while $Q\in\calP_{<\unc}$, such an arrow $q$ must be zero. This means $M'^{2h-2}=0$ as a morphism in $\calP_{\leq\unc}$ or $\calP$, which contradicts our assumption. This proves that $N$ is not regular.

It remains to show that $N\neq0$. Suppose $N=0$, then the local system $\calE$ is unramified on $\PP^{1}-\{0,1\}$ and tame at $0$ (by Proposition \ref{p:mono0}) and at $1$ (by Corollary \ref{c:tame}). The tame fundamental group $\pi_{1}^{\tame}(\PP^{1}-\{0,1\})$ is topologically generated by one element, which is both the generator of $I_{0}^{\tame}$ and $I_{1}^{\tame}$. Since the local monodromy at $1$ is unipotent element while the local monodromy at $0$ is not according to Proposition \ref{p:mono0}, this is impossible. Therefore $N$ is neither zero nor regular, proving the Proposition.
\end{proof}

\begin{remark} When $G$ is of type $A_{1}$, $\bP_{\infty}$ is also an Iwahori subgroup. The same argument as in Proposition \ref{p:mono1} shows that the local monodromy at $\infty$ is also regular unipotent.
\end{remark}

\subsection{Global geometric monodromy}\label{ss:global}
In this section, we assume $k$ to be an algebraically closed field, and we study the Zariski closure of the image of $\rho_{k}$ (also called the {\em geometric monodromy group} of the local system $\calE$).

\begin{theorem}\label{th:geommono}
Let $\hatG_{\rho}$ be the Zariski closure of the image of $\rho_k$. Then the neutral component of $\hatG_{\rho}$ is a semisimple group, and
\begin{equation*}
\hatG_{\rho}\begin{cases}=\hatG & \textup{ if }\hatG \textup{ is of type } A_{1}, E_{7}, E_{8}\textup{ or } G_{2}\\
\supset\SO_{4n-1} & \textup{ if }\hatG=\PSO_{4n}, n\geq3\\
\supset G_{2} & \textup{ if }\hatG=\PSO_{8}.\end{cases}
\end{equation*}
\end{theorem}
\begin{proof}
Let $\hatH$ be the neutral component of the Zariski closure of the image of $\rho_{k}$. 

We first claim $\hatH$ is a semisimple group. 

For char$(k)>0$, we may assume $k=\FF_p$ for some prime $p\nmid2\ell$. Then by Theorem \ref{th:eigen}(4), $\calE(\calK)$ is pure of weight zero, hence geometrically semisimple by \cite[Corollaire 5.4.6]{BBD}. Our claim then follows from \cite[Corollaire 1.3.9]{WeilII} (in \cite{WeilII}, Deligne remarked that the proof only uses the fact that the local system is {\em geometrically} semisimple). 

For char$(k)=0$, we may reduce to the case $k=\CC$. By the description of $\calE(\calK)$ given in Section \eqref{T*}, it is a direct summand of a semisimple perverse sheaf of geometric origin along the proper map $\GR_{\leq\lambda}\to\X{}$, therefore it is semisimple by the Decomposition Theorem \cite[Th\'eor\`eme 6.2.5]{BBD}. This implies $\hatH$ is a reductive group. Now suppose $S^{0}=\hatH^{ab}$ is nontrivial. Then $\pi_{1}(\X{})$ maps densely into an algebraic group $S(\Ql)$ with neutral component $S^{0}$ being a torus. The group $\pi_{1}(\X{})$ is topologically generated by the loops $\zeta_{0},\zeta_{1}$ and $\zeta_{\infty}$ around the punctures. We have already seen from Proposition \ref{p:mono1} and \ref{p:monoinf} that $\rho_{k}(\zeta_{1})$ and $\rho_{k}(\zeta_{\infty})$ are unipotent, hence have trivial image in $S(\Ql)$. By Proposition \ref{p:mono0}, $\rho_{k}(\zeta_{0})$ has semisimple part of order 2, therefore the image of $\pi_{1}(\X{})$ in $S(\Ql)$ is generated by an element of order at most 2, so cannot be Zariski dense. This contradiction implies that $S^{0}$ is trivial, i.e., $\hatH$ is semisimple.

By Proposition \ref{p:mono1}, $\hatH$ contains a regular unipotent element. Hence $\hatH$ is a semisimple subgroup of $\hatG$ containing a principal $\PGL_2$. According to Dynkin's classification (see \cite[Page 1500]{FG}), either $\hatH$ is the principal $\PGL_2$, or $\hatH=\hatG$, or $\hatH=\SO_{4n-1}$ if $\hatG=\PSO_{4n}$, or $\hatH=G_{2}$ if $\hatG=\PSO_{8}$.

If $G$ is not of type $A_{1}$, by Proposition \ref{p:monoinf}, the geometric monodromy at $\infty$ is unipotent but neither trivial nor regular, therefore not in the principal $\PGL_2$. Hence $\hatH$ cannot be equal to the principal $\PGL_2$. This finishes the proof.
\end{proof}

\subsection{Image of Galois representations}
In this section, we work with the base field $k=\QQ$. For any rational number $x\in\QQ\backslash\{0,1\}$, we get a closed point $i_{x}:\Spec\QQ\hookrightarrow\X{\QQ}$ and hence an embedding $i_{x,\#}:\GQ\to\pi_{1}(\X{\QQ})$, well-defined up to conjugacy. Restricting the representation $\rho_{\QQ}$ (attached to the $\hatG$-local system $\calE_{\QQ}$) using $i_{x,\#}$, we get a continuous Galois representation
\begin{equation*}
\rho_{x}:\GQ\to\hatG(\Ql).
\end{equation*}
A variant of Hilbert irreducibility \cite[Theorem 2]{T} shows that for $x$ away from a thin set of $\QQ$, the image of $\rho_{x}$ is the same as the image of $\rho_{\QQ}$, and therefore Zariski dense if $G$ is of type $A_{1},E_{7},E_{8}$ or $G_{2}$ by Theorem \ref{th:geommono}. We would like to give an effective criterion for $\rho_{x}$ to have large image.

\begin{prop}\label{p:Galrep}
Let $\rho:\GQ\to\hatG(\Ql)$ be a continuous $\ell$-adic representation. Suppose
\begin{enumerate}
\item For almost all primes $p$, $\rho_{p}:=\rho|_{\GQp}$ is unramified and pure of weight 0;
\item For some prime $p'\neq\ell$, $\rho_{p'}$ is tamely ramified, and a topological generator of the tame inertia group at $p'$ maps to a regular unipotent element in $\hatG(\Ql)$;
\item The local representation $\rho_{\ell}$ is Hodge-Tate (i.e., for any algebraic representation $V$ of $\hatG$, the induced action of $\Gal(\Qlbar/\Ql)$ on $V$ is Hodge-Tate).
\end{enumerate}
Let $\hatG_{\rho}\subset\hatG$ be the Zariski closure of the image of $\rho$. Then the neutral component $\hatG^{\circ}_{\rho}$ is a reductive subgroup of $\hatG$ containing a principal $\PGL_{2}$.
\end{prop}
\begin{proof} The proof is partially inspired by Scholl's argument in \cite[Proposition 3]{Scholl}.

Let $R\subset\hatG_{\rho}^{\circ}$ be the unipotent radical and let $\hatH$ be the connected reductive quotient $\hatG_{\rho}^{\circ}/R$. All these groups are over $\Ql$.

Let $V$ be an algebraic representation of $\hatG$, viewed as a $\GQ$-module via $\rho$. We first claim that for any subquotient $V'$ of $V$ as a $\GQ$-module, $\det(V')$ is a finite order character of $\GQ$. In fact, since $\rho$ is Hodge-Tate, so is $V'$. Hence we can write $\det(V')=\chi_\ell^N\epsilon$ for some finite order character $\epsilon$ of $\GQ$ and some integer $N$, where $\chi_{\ell}$ is the $\ell$-adic cyclotomic character. By (1), $\det(V')$ is pure of weight zero at almost all $p$, hence $\Frob_p$ acts on $\det(V')$ by a number with archimedean norm $1$. However, $\chi_\ell(\Frob_p)=p$, hence $N=0$. This implies $\det(V')$ is a finite order character of $\GQ$.

By assumption $\hatG_{\rho}^{\circ}(\Ql)$ contains a regular unipotent element $u$, which is the image of a topological generator of the tame inertia group at $p'$. Let $N=\log(u)\in\hatg$. Let $s$ be the image of a (lifting of the) Frobenius at $p'$. Then $\Ad(s)N=p'N$. Since $N$ is regular nilpotent, the element $s$ is semisimple and well-defined up to conjugacy. The action of $s$ on $\hatg$ determines a grading
\begin{equation}\label{grading}
\hatg=\bigoplus_{n\in\ZZ}\hatg(n)
\end{equation}
such that $s$ acts on $\hatg(n)$ by $(p')^n$, and $N:\hatg(n)\to\hatg(n+1)$. In fact, taking $N$ to be the sum of simple root generators, this grading is the same as the grading by the height of the roots. In particular, $\hatg(n)\neq0$ only for $1-h\leq n\leq h-1$, where $h$ is the Coxeter number of $\hatG$, and $\dim\hatg(1-h)=\dim\hatg(h-1)=1$. Moreover, the map
\begin{equation}\label{hardLef}
N^{2n}:\hatg(-n)\to\hatg(n)
\end{equation}
is an isomorphism for any $n\geq0$. 

For any subquotient $\hatG_{\rho}$-module $V$ of $\hatg$, we can similarly define a grading $V=\bigoplus_nV(n)$ under the action of $s$. We say $V$ is {\em symmetric} under the $s$-action if $\dim V(-n)=\dim V(n)$ for any $n$.

The action of the unipotent group $R$ on $\hatg$ gives a canonical increasing filtration of $\hatg$: $F_0\hatg=0$, $F_i\hatg/F_{i-1}\hatg=(\hatg/F_{i-1}\hatg)^R$. Therefore $\hatG_{\rho}$ acts on each $\Gr^F_i\hatg$ via the reductive quotient $\hatG_{\rho}/R$.

We claim that for each $i$, the associated graded $\hatG_\rho$-modules $\Gr^F_i\hatg$ are symmetric under the $s$-action. In fact, it suffices to show that each $F_i\hatg$ is symmetric under the $s$-action. Since \eqref{hardLef} is an isomorphism, $N^{2n}:F_i\hatg(-n)\to F_i\hatg(n)$ is injective. Hence $\dim F_i\hatg(-n)\leq\dim F_i\hatg(n)$ for any $n\geq0$.  On the other hand, we argued that $\det(F_i\hatg)$ is a finite order character of $\GQ$, so $s$ can only act on $\det(F_i\hatg)$ as identity. This implies that $\dim F_i\hatg(-n)=\dim F_i\hatg(n)$ for all $n$.

Let $\Nbar\in\hath=\Lie\hatH$ be the image of $N$. There is a unique $i$ such that $\Gr^F_i\hatg(-h+1)\neq0$, hence $\Gr^F_i\hatg(m)\neq0$ by the Claim. The iteration
\begin{equation*}
\Nbar^{2h-2}:\Gr^F_i\hatg(-h+1)\to\Gr^F_i\hatg(h-1)
\end{equation*}
is necessarily an isomorphism: for otherwise $N^{2h-2}\hatg(-h+1)=\hatg(h-1)$ would fall inside $F_{i-1}\hatg$, contradiction. Therefore $\Nbar$ acts on $\hatg$ with a Jordan block of size $2h-1$.

Let $\iota:\hatH\to\hatG^{\circ}_\rho$ be any section (a group homomorphism). Now $\iota(\Nbar)$ acts on $\Gr^F_i\hatg$ with a Jordan block of length $2h-1$, therefore the action on $\hatg$ has a Jordan block of size $\geq2h-1$. However, the only nilpotent class in $\hatg$ with a Jordan block of size $\geq2h-1$ under the adjoint representation is the regular nilpotent class. Hence $\iota(\Nbar)$ is a regular nilpotent class in $\hatg$. Since $\iota(\hatH)$ is a reductive group, it contains a principal $\PGL_2\subset\hatG$. 

Now fix such a principal $\PGL_{2}\subset\iota(\hatH)$. Consider the adjoint action of $\PGL_{2}$ on $\hatg$. The action of the maximal torus of $\PGL_2$ induces a similar grading as in \eqref{grading}, with $\hatg(0)$ a Cartan subalgebra of $\hatg$. The Lie algebra $\frr$ of $R$ is a $\PGL_{2}$-submodule of $\hatg$ consisting entirely of nilpotent elements. However, any nonzero $\PGL_2$-submodule of $\hatg$ has a nonzero intersection with $\hatg(0)$, hence containing nonzero semisimple elements. This forces $R$ to be trivial. Therefore $\hatG_\rho$ is a reductive subgroup of $\hatG$ containing a principal $\PGL_2$.
\end{proof}

Recall that in Proposition \ref{p:int} we have extended the local system $\calE_{\QQ}$ to $\X{\Z}$ for some integer $N$. 
\begin{prop}\label{p:Qmono} When $G$ is of type $A_{1}, E_{7},E_{8}$ or $G_{2}$.
\begin{enumerate}
\item []
\item Let $(a,b)$ be nonzero coprime integers such that both $a-b$ and $b$ have prime divisors not dividing $2\ell N$. Let $x=\frac{a}{b}$. Then the image of $\rho_{x}$ is Zariski dense in $\hatG(\Ql)$. 
\item There are infinitely many rational numbers $x\in\QQ-\{0,1\}$ such that $\rho_{x}$ are mutually non-isomorphic and all have Zariski dense image in $\hatG(\Ql)$.
\end{enumerate}
\end{prop}
\begin{proof}
(1) For each prime $p$, let $I_{p}\subset\GQp$ be the inertia group. Let $Z\subset\PP^{1}_{\Z}$ be the closure of the $\QQ$-point $x$, then the projection $\pi:Z\to\Spec\Z$ is an isomorphism. For any prime $p\nmid2\ell N$, the reduction of $Z$ at $p$ is a point $x_{p}\in\PP^{1}(\FF_{p})$. Then the proof of \cite[Proposition A.4]{DR} shows the following fact: $\rho_{x}(I_{p})$ is contained in the image of $\rho_{\FF_{p}}|_{I_{x_{p}}}$ ($I_{x_{p}}\subset\Gal(\FF_{p}(t)^{s}/\FF_{p}(t))$ being the inertia of the point $x_{p}\in\PP^{1}(\FF_{p})$); moreover, if $\rho_{\FF_{p}}|_{I_{x_{p}}}$ is tame and unipotent, $\rho_{x}|_{I_{p}}$ is also tame and maps to the same unipotent class. Using the calculation of the local monodromy in Section \ref{ss:loc}, we have for any $p\nmid2\ell N$,
\begin{itemize}
\item If $p\mid a-b$ (hence $x_{p}=1$), then $\rho_{x}(I_{p})$ contains a regular unipotent element;
\item If $p\mid b$ (hence $x_{p}=\infty$) and $G$ is not of type $A_{1}$, then $\rho_{x}(I_{p})$ contains a unipotent element which is neither trivial nor regular.
\item If $p$ does not divide $ab(a-b)$, then $\rho_{x}$ is unramified at $p$. 
\end{itemize}
We check that $\rho_{x}$ satisfies the assumptions of Proposition \ref{p:Galrep}: the unramifiedness of condition (1) is check above and purity is proved in Theorem \ref{th:eigen}(4); condition (2) is also checked above since we do have a prime $p\mid a-b$ and $p\nmid2\ell N$; the Hodge-Tate condition (3) follows from the motivic interpretation (Proposition \ref{p:motivic}) and the theorem of Faltings \cite{Faltings}. Applying Proposition \ref{p:Galrep} to $\rho_{x}$, we conclude that the Zariski closure $\hatG_{\rho_{x}}$ of the image of $\rho_{x}$ is a reductive subgroup of $\hatG$ containing a principal $\PGL_{2}$. Moreover, if $G$ is not of type $A_{1}$, it contains another unipotent element which is neither regular nor trivial. The argument of Theorem \ref{th:geommono} then shows that the derived group of $\hatG_{\rho_{x}}$ (which is semisimple and contains a principal $\PGL_{2}$) is already the whole $\hatG$.

(2) Choose an increasing sequence of prime numbers $2\ell N<p_{1}<p_{2}<\cdots$. Define $x_{i}=\frac{p_{i}+p_{i+1}}{p_{i+1}}$, $i=1,2,\cdots$. Then $x_{i}$ satisfies the conditions in (1), and the place where $\rho_{x_{i}}$ has regular unipotent monodromy is $p_{i}$ together with possibly certain places dividing $2\ell N$. Therefore these $\rho_{x_{i}}$ are mutually non-isomorphic and all have Zariski dense image in $\hatG(\Ql)$.  
\end{proof}

\subsection{Application to the Inverse Galois Problem}\label{ss:invGal}

\subsubsection{Betti realization of the local system}
We work with $k=\CC$ and use analytic topology instead of \'etale topology. We change $\ell$-adic cohomology to singular cohomology. The main results in the previous discussion still hold in this situation. In particular, for each $\chi\in\tilA^{*}_{0,\odd}$, we get a representation of the topological fundamental group
\begin{equation}
\rho^{\topo}_{\chi}:\pi_{1}^{\topo}(\X{\CC})\to\hatG(\QQ).
\end{equation}

\begin{theorem}\label{th:invGal} For sufficiently large prime $\ell$, the finite simple groups $E_{8}(\FF_{\ell})$ and $G_{2}(\FF_{\ell})$ are Galois groups over $\QQ$.
\end{theorem}
\begin{proof}
Let $G$ be of type $E_{8}$ or $G_{2}$. We write $\rho^{\topo}_{\chi}$ simply as $\rho^{\topo}$.

In the motivic interpretation of $\calE^{\qm}$ given in \eqref{Eqmin}, we may work with $\ZZ$-coefficients in the analytic topology (resp. $\ZZ_{\ell}$-coefficients in the \'etale topology). For sufficiently large prime $\ell$, $\bH^{2h^{\vee}-2}_{c}(\tilY_{\QQ}/\X{\QQ},\ZZ_{\ell})_{\odd}$ is a free $\ZZ_{\ell}$-module of rank $n=\dim V^{\qm}$. Also, there exists an integer $N_{1}$ such that $\upH^{2h^{\vee}-2}_{c}(\tilY^{\an}_{\CC}/\PP^{1,\an}_{\CC}-\{0,1,\infty\},\ZZ[1/N_{1}])_{\odd}$ a free $\ZZ[1/N_{1}]$-module of rank $n$. This integral lattice gives an integral form of $\hatG$ over $\ZZ[1/N_{1}]$, which, after enlarging $N_{1}$, is assume to be isomorphic to the restriction of the Chevalley group scheme of the same type as $\hatG$. Therefore for primes $\ell\nmid N_{1}$, the comparison isomorphism between singular cohomology and $\ell$-adic cohomology gives a commutative diagram
\begin{equation}\label{topet}
\xymatrix{\pi^{\topo}(\X{\CC})\ar[r]^{\rho^{\topo}}\ar[d] & \GL_{n}(\ZZ[1/N_{1}])\cap\hatG(\QQ)\ar[d]\ar[r]^(.6){\mod\ell} & \hatG(\FF_{\ell})\\
\pi_{1}(\X{\QQ})\ar[r]^{\rho_{\QQ}} & \GL_{n}(\ZZ_{\ell})\cap\hatG(\QQ_{\ell})\ar[ur]^{r_{\ell}}}
\end{equation}

We claim that the first row in the above diagram is surjective for large $\ell$. Let $\Pi\subset\hatG(\QQ)$ be the image of $\rho^{\topo}$, which is finitely generated because $\pi_{1}^{\topo}(\X{\CC})$ is a free group of rank 2. Since $\rho_{\Qbar}$ has Zariski dense image in $\hatG(\Ql)$ (by Theorem \ref{th:geommono}) and $\pi_{1}^{\topo}(\X{\CC})$ is dense in $\pi_{1}(\X{\Qbar})$, $\rho^{\topo}$ also has Zariski dense image in $\hatG(\QQ)$. Therefore $\Pi\subset\hatG(\QQ)$ is both finitely generated and Zariski dense. By \cite[Theorem in the Introduction]{MVW}, for sufficiently large prime $\ell$, the reduction modulo $\ell$ of $\Pi$ surjects onto $\hatG(\FF_{\ell})$.

Consequently, in the diagram \eqref{topet}, the composition $r_{\ell}\circ\rho_{\QQ}$ is also surjective for large $\ell$. Finally we apply Hilbert irreducibility theorem (see \cite[Theorem 3.4.1]{SerreGal}) to conclude that there exists a rational number $x\in\QQ-\{0,1\}$ such that $r_{\ell}\circ\rho_{x}:\GQ\to\GL_{n}(\ZZ_{\ell})\cap\hatG(\QQ_{\ell})\to\hatG(\FF_{\ell})$ is surjective. This solves the inverse Galois problem for the finite simple group $\hatG(\FF_{\ell})$.
\end{proof}

\subsection{Conjectural properties of the local system}\label{ss:further}
\subsubsection{Local monodromy}\label{sss:conjloc}
Lusztig has defined a map \cite{L}
\begin{equation*}
\{\textup{conjugacy classes in } W\}\to\{\textup{unipotent classes in } \hatG\}
\end{equation*}
In particular, if $-1\in W$ gets mapped to a unipotent class $v$ in $\hatG$, which has the property that
\begin{equation*}
\dim Z_{\hatG}(v)=\frac{\#\Phi_G}{2}.
\end{equation*}
We tabulate these unipotent classes using Bala-Carter classification (see \cite{Carter})

\begin{center}
\begin{tabular}{|c|c|}
\hline
Type of $G$ & the unipotent class $v$ \\
\hline
$D_{2n}$ & Jordan blocks $(1,2,\cdots,2,3)$\\
\hline
$E_7$ & $4A_1$\\
\hline
$E_8$ & $4A_1$\\
\hline
$G_2$ & $\tilA_1$\\
\hline
\end{tabular}
\end{center}

\begin{conj}\label{c:0}
Under $\rho_{k}$, a topological generator of the tame inertia group $I_{\infty}^{\tame}$ gets mapped to a unipotent element in the unipotent class $v$.
\end{conj}

We also make the following conjecture on the local monodromy at $0$.
\begin{conj}\label{c:inf}
Under $\rho_{k}$, a topological generator of the tame inertia group $I_{0}^{\tame}$ gets mapped to an element conjugate to $\kappa\in\hatT[2]$ in Section \ref{sss:kappa}. In other words, the unipotent part of the local monodromy at $0$ is trivial.
\end{conj}

Conjecture \ref{c:0} and Conjecture \ref{c:inf} would imply that in the case of $G_{2}$, our local system is isomorphic to the one constructed by Dettweiler and Reiter in \cite{DR}, because they checked there is up to isomorphism only one such local system with the same local monodromy as theirs.

\subsubsection{Rigidity} Assume $G$ is of type $E_{7}, E_{8}$ or $G_{2}$. Consider the adjoint local system $\Ad(\calE)$ associated to the adjoint representation of $\hatG$. Let $j_{!*}\Ad(\calE)$ be the middle extension of $\Ad(\calE)$ to $\PP^{1}$. Then we have an exact sequence (cf. \cite[proof of Proposition 5.3]{HNY})
\begin{equation}\label{ex}
0\to\hatg^{I_{0}}\oplus\hatg^{I_{1}}\oplus\hatg^{I_{\infty}}\to\cohoc{1}{\X{},\Ad(\calE)}\to\cohog{1}{\PP^{1},j_{!*}\Ad(\calE)}\to0.
\end{equation}

\begin{conj}\label{c:rigid}
The local system $\calE$ is {\em cohomologically rigid}, i.e., $\cohog{1}{\PP^{1},j_{!*}\Ad(\calE)}=0$.
\end{conj}

We make the following simple observation.
\begin{lemma} Any two of the three Conjectures \ref{c:0}, \ref{c:inf} and \ref{c:rigid} imply the other.
\end{lemma}
\begin{proof}
Since $\Ad(\calE)$ is tame and has no global sections, $\dim\cohoc{1}{\X{},\Ad(\calE)}=\dim\hatG$ by  the Grothendieck-Ogg-Shafarevich formula. Also note that $\hatg^{I_{1}}=\rk\hatG$ because a generator of $I^{\tame}_{1}$ maps to a regular unipotent element by Proposition \ref{p:mono1}. Therefore the exact sequence \eqref{ex} implies
\begin{equation*}
\dim\hatg^{I_{0}}+\dim\hatg^{I_{\infty}}+\dim\cohog{1}{\PP^{1},j_{!*}\Ad(\calE)}=\dim\hatG-\rk\hatG=\#\Phi_{G}.
\end{equation*}

Conjectures \ref{c:0} and \ref{c:inf} imply Conjecture \ref{c:rigid} because then $\hatg^{I_{\infty}}=\hatg^{v}$ has dimension $\#\Phi_{G}/2$ and $\hatg^{I_{0}}=\hatg^{\kappa}$ also has dimension $\#\Phi_{G}/2$ because $\kappa$ is a split Cartan involution (see Section \ref{sss:kappa}).

Conjectures \ref{c:inf} and \ref{c:rigid} imply Conjecture \ref{c:0} because then $\dim\hatg^{I_{0}}=\#\Phi_{G}-\dim\hatg^{v}=\#\Phi_{G}/2$ while $\hatg^{I_{0}}\subset\hatg^{\kappa}$ (which has dimension $\#\Phi_{G}/2$) by Proposition \ref{p:mono0}, which forces $\hatg^{I_{0}}=\hatg^{\kappa}$ and the unipotent part of the local monodromy at $0$ must be trivial.

Conjectures \ref{c:0} and \ref{c:rigid} imply Conjecture \ref{c:inf} because then $\dim\hatg^{I_{\infty}}=\#\Phi_{G}-\dim\hatg^{\kappa}=\#\Phi_{G}/2$ and $v$ is the only unipotent class with this property by checking tables in \cite{Carter}.
\end{proof}

\subsubsection{Global monodromy} Theorem \ref{th:geommono} does not completely determine the geometric monodromy group of $\calE$ in type $D_{2n}$. It would be interesting to determine them. Note that the pinned automorphism group $\Out(G)$ permutes the odd central characters $\tilA^{*}_{0,\odd}$ nontrivially. For those $\chi$ which are fixed by $\Out(G)$, we expect the geometric monodromy of $\calE_{\chi}$ to be $\hatG^{\Out(\hatG)}$, for the same reason as \cite[Section 6.1]{HNY}. How about those $\chi$ which are not fixed by $\Out(G)$? In the case $G$ is of type $D_{4}$, is the geometric monodromy $G_{2}$ or $\SO_{7}$ or $\PSO_{8}$?

\subsubsection{Rigid triples}
In inverse Galois theory, people use ``rigid triples'' in a finite group $\Gamma$ to construct \'etale $\Gamma$-coverings of $\X{\QQ}$, and hence getting $\Gamma$ as a Galois group over $\QQ$ using Hilbert irreducibility. For details see Serre's book \cite{SerreGal}.

Let $C_{1}$ be the regular unipotent class in $\hatG(\FF_{\ell})$. Let $C_{\infty}$ be the unipotent class of $v$ in $\hatG(\FF_{\ell})$ defined in Section \ref{sss:conjloc}. Let $C_{0}$ be the conjugacy class of the reduction of $\kappa$ in $\hatG(\FF_{\ell})$.

\begin{conj} Let $\hatG$ be of type $E_{8}$ or $G_{2}$. Then $(C_{0}, C_{1}, C_{\infty})$ form a strictly rigid triple in $\hatG(\FF_{\ell})$. In other words, the equation
\begin{equation}
g_{0}g_{1}g_{\infty}=1, g_{i}\in C_{i} \textup{ for } i=0,1,\infty
\end{equation}
has a unique solution up to conjugacy in $\hatG(\FF_{\ell})$, and any such solution $\{g_{0},g_{1},g_{\infty}\}$ generate $\hatG(\FF_{\ell})$.
\end{conj}  

\noindent\textbf{Acknowledgement}
The author would like to thank B.Gross for many discussions and encouragement. In particular, the application to the inverse Galois problem was suggested by him. The author also thanks D.Gaitsgory, S.Junecue, N.Katz, G.Lusztig, B-C.Ng\^o, D.Vogan and L.Xiao for helpful discussions. The work is supported by the NSF grant DMS-0969470.

\end{document}